\pdfoutput=1
\documentclass[a4paper]{amsart}
\usepackage[hmargin=3cm,vmargin=3.5cm]{geometry}
\usepackage{amsmath,amsthm,amssymb,amscd}
\usepackage[usenames,dvipsnames]{color}
\definecolor{darkgreen}{rgb}{0,0.5,0}
\definecolor{darkblue}{rgb}{0,0,0.5}
\usepackage[colorlinks=true,urlcolor=darkblue,citecolor=darkgreen,linkcolor=darkblue]{hyperref}
\usepackage{graphicx}
\usepackage{esint}
\usepackage[all,cmtip]{xy}
\usepackage[latin1]{inputenc}
\usepackage[T1]{fontenc}
\usepackage{lmodern}
\usepackage{url}
\newcommand{\relmiddle}[1]{\mathrel{}\middle#1\mathrel{}}
\renewcommand{\bar}{\overline}

\renewcommand{\epsilon}{\varepsilon}
\newcommand{\eps}{\epsilon}
\renewcommand{\phi}{\varphi}
\newcommand{\MOD}{\mathrm{mod}}

\newcommand{\abs}[1]{\lvert #1 \rvert}

\newcommand{\frakg}{\mathfrak{g}}

\newcommand{\fraksl}{\mathfrak{sl}}

\newcommand{\bbA}{\mathbb{A}}
\newcommand{\bbC}{\mathbb{C}}

\newcommand{\bbZ}{\mathbb{Z}}
\newcommand{\bbN}{\mathbb{N}}
\newcommand{\bbQ}{\mathbb{Q}}
\newcommand{\bbR}{\mathbb{R}}
\newcommand{\bbP}{\mathbb{P}}

\newcommand{\inv}{\varpi}

\newcommand{\Vol}{\mathrm{Vol}}

\newcommand{\calA}{\mathcal{A}}

\newcommand{\calD}{\mathcal{D}}

\newcommand{\calG}{\mathcal{G}}
\newcommand{\calH}{\mathcal{H}}

\newcommand{\calM}{\mathcal{M}}

\newcommand{\frakt}{\mathfrak{t}}

\newcommand{\su}{\mathfrak{su}}
\newcommand{\SU}{\mathrm{SU}}

\newcommand{\SL}{\mathrm{SL}}
\newcommand{\PSL}{\mathrm{PSL}}

\newcommand{\End}{\text{End}}

\newcommand{\isom}{\cong}

\newcommand{\dd}{\partial}

\DeclareMathOperator{\sgn}{sgn}

\DeclareMathOperator{\diag}{diag}

\DeclareMathOperator{\U}{U}

\DeclareMathOperator{\Id}{Id}

\DeclareMathOperator{\Lie}{Lie}

\DeclareMathOperator{\Tr}{Tr}
\DeclareMathOperator{\tr}{tr}
\DeclareMathOperator{\hol}{hol}

\DeclareMathOperator{\CS}{CS}
\DeclareMathOperator{\vol}{vol}

\DeclareMathOperator{\Ad}{Ad}

\def\XXint#1#2#3{{\setbox0=\hbox{$#1{#2#3}{\int}$}
\vcenter{\hbox{$#2#3$}}\kern-.5\wd0}}

\newcommand{\id}{\mathrm{id}}

\newcommand{\Aut}{\mathrm{Aut}}
\newcommand{\Hom}{\mathrm{Hom}}

\newcommand{\MCG}{\mathrm{MCG}}

\newcommand{\Span}{\mathrm{span}}

\newcommand{\LHS}{\mathrm{LHS}}
\newcommand{\RHS}{\mathrm{RHS}}
\newtheorem{thm}{Theorem}[subsection]

\newtheorem{prop}[thm]{Proposition}
\newtheorem{cor}[thm]{Corollary}

\newtheorem{conjec}[thm]{Conjecture}
\newtheorem{lem}[thm]{Lemma}

\theoremstyle{remark}

\newtheorem{rem}[thm]{Remark}

\theoremstyle{definition}

\newtheorem{example}[thm]{Example}

\begin{document}
\title{On the Witten--Reshetikhin--Turaev invariants of torus bundles}
\author{Jørgen Ellegaard Andersen}
\address{Centre for Quantum Geometry of Moduli Spaces\\ Faculty of Science\\
  Aarhus University\\
  DK-8000, Denmark}
  \email{andersen@qgm.au.dk}
\author{Søren Fuglede Jørgensen}
  \address{Department of Mathematics\\
  Box 480\\
  Uppsala University\\
  SE-75106 Uppsala, Sweden}
  \email{soren.fuglede.jorgensen@math.uu.se}
\thanks{Supported in part by the center of excellence grant ``Center for Quantum Geometry of Moduli Space'' from the Danish National Research Foundation.}

\begin{abstract}
  By methods similar to those used by Lisa Jeffrey \cite{Jef}, we compute the quantum $\SU(N)$-invariants for mapping tori of trace $2$ homeomorphisms of a genus $1$ surface when $N = 2,3$ and discuss their asymptotics. In particular, we obtain directly a proof of a version of Witten's asymptotic expansion conjecture for these $3$-manifolds. We further prove the growth rate conjecture for these $3$-manifolds in the $\SU(2)$ case, where we also allow the $3$-manifolds to contain certain knots. In this case we also discuss trace $-2$ homeomorphisms, obtaining -- in combination with Jeffrey's results -- a proof of the asymptotic expansion conjecture for all torus bundles.
\end{abstract}
\maketitle
\tableofcontents

\section{Introduction}
\subsection{The conjectures}
Assume in the following that the gauge group is $G =  \SU(N)$ (even if most of what follows makes sense for more general groups), and let $M$ be a closed oriented 3-manifold containing a framed link $L$. Let $P \to M$ be a principal $G$-bundle, let $\calA_P$ denote the space of all connections in $P$, and let $\calG_P$ denote the space of gauge transformations of $P$ acting on $\calA_P$. Recall that the Chern--Simons action defines a map
\begin{align*}
	\CS : \calA_P/\calG_P \to \bbR/\bbZ,
\end{align*}
given explicitly by
\begin{align*}
	\CS([A]) = \frac{1}{8\pi^2} \int_M \tr(dA \wedge A + \tfrac{2}{3} A \wedge A \wedge A).
\end{align*}
Let now $k \in \bbZ_{\geq 0}$ be a \emph{level} and choose for every component $L_i$ of $L$ a finite-dimensional representation $R_i$ of $G$, referred to in the following as a \emph{colouring}. Let $\hol_A(L_i) \in G$ denote the holonomy of a connection $A$ about a component $L_i$ of $L$. We then define the (physical) Chern--Simons partition function
\begin{align*}
	Z^\mathrm{phys}_{k,G}(M,L) = \int_{\calA_P/\calG_P} \prod_i \tr(R_i(\hol_A(L_i))) \exp(2\pi i k \CS(A))\calD A.
\end{align*}
Witten argued in \cite{WitJones} that this expression defines a topological quantum field theory. Much of quantum topology is inspired by the interpretation of the above integral, which a priori makes no sense from a mathematical point of view as the measure $\calD A$ is not defined. Attemps to make out of this a rigorous definition quickly arose, using the representation theory of quantum groups in \cite{RT2} and \cite{TW} (see also \cite{Tu} for a more modern account on these invariants); in this paper we stick primarily to the setup of skein theory from \cite{Bla}, as will be described in Section~\ref{quantuminvariants}. In particular the authors defined topological invariants $Z_k^G(M,L)$ -- which we will refer to as the level $k$ quantum $G$-invariants of $(M,L)$ -- defined to behave under surgery the way Witten had shown $Z^\mathrm{phys}_{k,G}$ would. The main conjecture that we seek out to discuss is that $Z_{k,G}^\mathrm{phys} = Z_k^G$. Of course, such an equality is hard to check when only one side is defined. However, by applying the method of stationary phase approximation to the left hand side, one arrives at expressions that may be directly checked for the right hand side.

Since any principal $G$-bundle over $M$ is trivializable, we refer to the moduli space $\calM(M)$ of flat $G$-bundles on $M$ simply as the moduli space of flat $G$-connections on $M$, without mention of a particular bundle. The moduli space $\calM(M)$ is compact, and the Chern--Simons action is constant on components of $\calM(M)$. We can therefore formulate the following conjecture, which first appeared in \cite{Oht}.

\begin{conjec}[Asymptotic expansion conjecture (AEC)]
  \label{asympconjec}
   Let $M$ be a closed oriented 3-manifold. Let $r = k+N$. Let $\{c_0 = 0,\dots,c_m\}$ be the finitely many values of the Chern--Simons action on the moduli space $\calM(M)$. Then there exist $d_j \in \tfrac{1}{2}\bbZ$, $b_j \in \bbC$, and $a_j^l \in \bbC$ for $j = 0, \dots, m$, $l = 1, 2, \dots$ such that
\begin{align*}
  Z_k(M) \sim_{k \to \infty} \sum_{j=0}^m \exp(2\pi i rc_j)r^{d_j}b_j\left(1+\sum_{l=1}^\infty a_j^l r^{-l/2}\right)
\end{align*}
in the sense that
\begin{align*}
  \left\lvert Z_k(M) - \sum_{j=0}^m \exp(2\pi i rc_j)r^{d_j}b_j\left(1+\sum_{l=1}^L a_j^l r^{-l/2}\right) \right\rvert = O(r^{d-(L+1)/2}),
\end{align*}
for $L= 0, 1, \dots$, where $d = \max_j d_j$.
\end{conjec}
\begin{rem}
  \label{framings}
  In all of the previous, we have silently ignored an important feature of quantum invariants: they actually depend on a choice of \emph{$2$-framing} on $M$ (see \cite{AtiFraming}). Such $2$-framings are in bijective correspondence with the integers, and it is well-known that given two $2$-framings on $M$, the corresponding quantum $G$-invariants differ by a factor of $\exp(2\pi i c/24)$ to some power, where here $c$ is the \emph{central charge} of $G$; we elaborate on this in Remark~\ref{framingrem}. In particular, the asymptotic significance of this is simply a root of unity (e.g. $\exp(\pi i/4)$ for the case $G = \SU(2)$), which is irrelevant to the conjecture as stated above. For this reason, we will largely ignore this issue.
\end{rem}
\begin{rem}
  \label{asympexpanunique}
  Taking into account the issue of $2$-framings, if an asymptotic expansion as the above exists, it is well-known that the constants $c_j$, $d_j$, $b_j$, and $a_j^l$ are uniquely determined -- see \cite{Andfiniteorder} and \cite[Theorem~8.2]{Han}. In particular they are topological invariants of the $3$-manifold. Moreover, the $c_j$ and $d_j$ are independent of the $2$-framing.
\end{rem}
Following Witten's original calculation \cite[Sect.~2]{WitJones} for the case when $\calM(M)$ is finite, and subsequent generalizations in \cite[(1.36)]{FG}, \cite[(5.1)]{Jef}, conjecturally, the constants appearing in this Conjecture~\ref{asympconjec} have various topological interpretations in terms of e.g. Reidemeister torsion and spectral flow (see also \cite{Oht}, \cite{AH}). We will be particularly interested in the behaviour of the $d_j$.

For a flat connection $A$ in $P \to M$, we obtain a complex 
\begin{align*}
  \cdots \stackrel{d_A}{\rightarrow} \Omega^{k-1}(M,\Ad_P) \stackrel{d_A}{\rightarrow} \Omega^k(M,\Ad_P) \stackrel{d_A}{\rightarrow} \Omega^{k+1}(M,\Ad_P) \stackrel{d_A}{\rightarrow} \cdots.
\end{align*}
Let $H^i(M,\Ad_P)$ denote the cohomology of this complex, and let
\begin{align*}
  h^i_A = \dim H^i(M,\Ad_P).
\end{align*}

In the case where Conjecture~\ref{asympconjec} holds true for a $3$-manifold $M$, we conjecture the following (see \cite{Oht} for more details).
\begin{conjec}[The growth rate conjecture]
  \label{growthrateconjec}
  Let $c_j, d_j$ as in Conjecture~\ref{asympconjec} and let $\calM_j$ denote the subspace of $\calM$ corresponding to Chern--Simons action $c_j$. Then
  \begin{align*}
    d_j = \frac{1}{2} \max_{[A] \in \calM_j} (h^1_A - h^0_A),
  \end{align*}
  where the $\max$ denotes the maximum over all Zariski open subsets of $\calM_j$ with the property that $h^1_A - h^0_A$ is constant on the subset.
\end{conjec}

\subsection{Known results on the AEC}
Papers on these various conjectures are plenty, and we try here to give an overview of what is known on the subject but apologize in advance to any authors that might have been left out. Jeffrey, \cite{Jef}, proved the AEC for all lens spaces as well as for mapping tori of Anosov torus homeomorphisms on a torus in the case $G = \SU(2)$ and for a certain subfamily of these mapping tori for general $G$. Jeffrey's result for Anosov torus homeomorphisms has been extended to general $G$ in \cite{Chaasymp}. Andersen, \cite{Andfiniteorder}, proved the AEC for mapping tori of finite order diffeomorphisms of surfaces of genus $g \geq 2$ in the case $G = \SU(N)$, and in \cite{AH} the authors expanded on this result, identifying the leading order terms of the asymptotic expansion in terms of classical topological-geometric invariants. The case of lens spaces for $G = \SU(N)$ is discussed in \cite{HT} (see the comment preceding \cite[Conj.~4.3]{HT}). Hikami, \cite{Hik}, also considered the case of Brieskorn homology spheres for $G = \SU(2)$ in greater detail. The paper \cite{KSV} suggests numerical evidence \emph{against} the conjectures for the $3$-manifolds $S^3(4_1(-n/1))$, $n = 7, 16, 22$, demonstrating a contribution from a non-Chern--Simons-value phase. The asymptotics of quantum invariants of surgeries on the figure-eight knot have also been analyzed by Andersen and Hansen in \cite{AndHan}.

Rozansky, \cite{Rozformulas}, discusses the AEC for general Seifert manifolds for $G = \SU(2)$. The methods used in that paper are somewhat heuristic and some error estimates have been left out. These are supplied in Hansen's PhD thesis \cite{Han} and collected in \cite{Han1} and the preprint \cite{Han2}, proving the AEC for all Seifert manifolds in the case $G = \SU(2)$. Our present analysis should be viewed in this context, as the manifolds $M^b$ and $\tilde{M}^b$ we will be considering are Seifert manifolds of symbol $\{b;(o_1,1)\}$ and $\{(b;(n_2,2)\}$ for $b \in \bbZ$ (see e.g. \cite[pp.~124--125]{Orl}). The Thurston geometries of the manifolds are Euclidean when $b = 0$, and they have nil geometry for $n \not= 0$.

As far as preprints go, the results are as follows: in \cite{CMII}, Charles and Marché prove the $\SU(2)$-AEC for the manifolds $S^3(4_1(p/q))$ with $p$ not divisible by $4$ and $H^1(M,\mathrm{Ad}_\rho) = 0$ for all irreducible flat connections $\rho : \pi_1(M) \to \SU(2)$. Using similar techniques, Charles \cite{Cha} proves the AEC for $S^3(T_{a,b}(p/q))$, whenever $p/q \not= 2ab l/m$ for all $l,m$ with $a$ and $b$ not dividing $m$, also in the $\SU(2)$ case. Here $T_{a,b}$ is the $(a,b)$-torus knot. Charles extends his results to other mapping tori in \cite{Chaasympmcg} under certain regularity conditions. 

In \cite{Agenerel} Andersen has recently given a geometric formula for the leading order asymptotics of quantum invariants of any three manifold and the asymptotic analysis of this formula will certainly lead to considerable progress on the AEC.

\subsection{Results of this paper}
The paper is organized as follows: In Section~\ref{quantuminvariants}, we recall the definition of the quantum $\SU(N)$-representation and the corresponding quantum $\SU(N)$-invariants of mapping tori following \cite{Bla}. In particular we discuss what happens when the underlying surface is a torus and the mapping torus has monodromy an element of $\SL(2,\bbZ)$ with trace $\pm 2$. Notice that all of these manifolds are Seifert manifolds; for trace $2$ diffeomorphisms they have symbol $\{b;(o_1,1)\}$ for $b \in \bbZ$ and for trace $-2$ they have symbol $\{ b;(n_2,2)\}$ (see \cite{Orl}).

In Section~\ref{su2afsnit} we consider the case $G = \SU(2)$ in detail. We give a formula (Corollary~\ref{stnudenlinks}) for the quantum invariants of the mapping tori mentioned above and find the Chern--Simons values of the corresponding moduli space, proving Conjecture~\ref{asympconjec} for this family of $3$-manifolds (Theorem~\ref{su2cs}). As it poses no extra difficulty, we also allow the $3$-manifolds to be endowed with a link along a regular Seifert fiber and consider the dependence of the quantum invariant on the colouring of the link. We describe the moduli space of the $3$-manifold explicitly (Proposition~\ref{su2modulirum}) and are able to calculate the dimensions $h^i_A$ of the twisted cohomology groups, proving Conjecture~\ref{growthrateconjec} for these manifolds (Theorem~\ref{growthratesu2}).

In Section~\ref{sunafsnit} we turn to the case $G = \SU(N)$, giving a similar expression for the quantum invariants for $N = 3$ (Theorem~\ref{invariantssun}), this time for the bare $3$-manifolds containing no links. Whereas we give no explicit description of the moduli space in this general case, we find the Chern--Simons values for general $N$ (Proposition~\ref{suncs}) and prove Conjecture~\ref{asympconjec} for $N = 3$ (Corollary~\ref{sunasympconjec}).

Finally, in Section~\ref{pseudoafsnit} we discuss the problem (see \cite{AMU}) of using quantum representations to determine stretch factors of pseudo-Anosov homeomorphisms.

The main technical tool used in this paper is a generalization of the quadratic reciprocity law of Gauss sums by Jeffrey, \cite{Jef}. As mentioned above, Jeffrey addresses the conjecture for the torus mapping tori whose monodromy have trace different from $\pm 2$, thereby avoiding particular technical issues that will be apparent in our proofs. Combining her results with ours, we therefore end up with a description of the quantum invariants of all mapping tori over a torus, i.e. all torus bundles, at least in the case $G = \SU(2)$ (Theorem~\ref{summarythm}).

Note also that some of the results of this paper are written up in the second author's (unpublished) theses \cite{Joer}, \cite{JoerThesis} -- written under supervision of the first author -- which contain a somewhat more thorough introduction to many of the concepts discussed here.

\textbf{Acknowledgements.} The second author would like to thank his office mates at Aarhus Universitet for coping with innumerable discussions about sign conventions.

\section{Review of quantum invariants}
\subsection{Quantum invariants and quantum representations}
\label{quantuminvariants}
The $\SU(N)$-quantum invariants of mapping tori can be described using only the $2$-dimensional part of the TQFT, which we briefly recall in this section; if not to make the paper self-contained then at least to make it accessible to anyone familiar with \cite{Bla}.

Let $k \in \bbZ_{\geq 0}$ be a \emph{level}. The main objects in the theory are certain vector spaces $V_k(\Sigma)$ associated to an oriented closed surface $\Sigma$ as well as projective representations $\rho_k : \mathrm{MCG}(\Sigma) \to \bbP \Aut(V_k(\Sigma))$ of the mapping class group of $\Sigma$. These so-called \emph{quantum representations} have been described in a plethora of different setups including the representation theory of quantum groups, more general modular categories (see e.g. \cite{Tu}), $2$-dimensional conformal field theory (see e.g. \cite{TUY}), and geometric quantization (see e.g. \cite{asympgeom}). We stick to the framework of Homflypt skein theory as described in \cite{Bla} and collect here the parts of it that are most important for our present study. That this theory corresponds to the one arising from conformal field theory (and therefore to the one of geometric quantization by the result of \cite{Las}), is shown in the series of papers \cite{AU1}, \cite{AU2}, \cite{AU3}, and the preprint \cite{AU4}.

\begin{figure}
  \centering
  \includegraphics{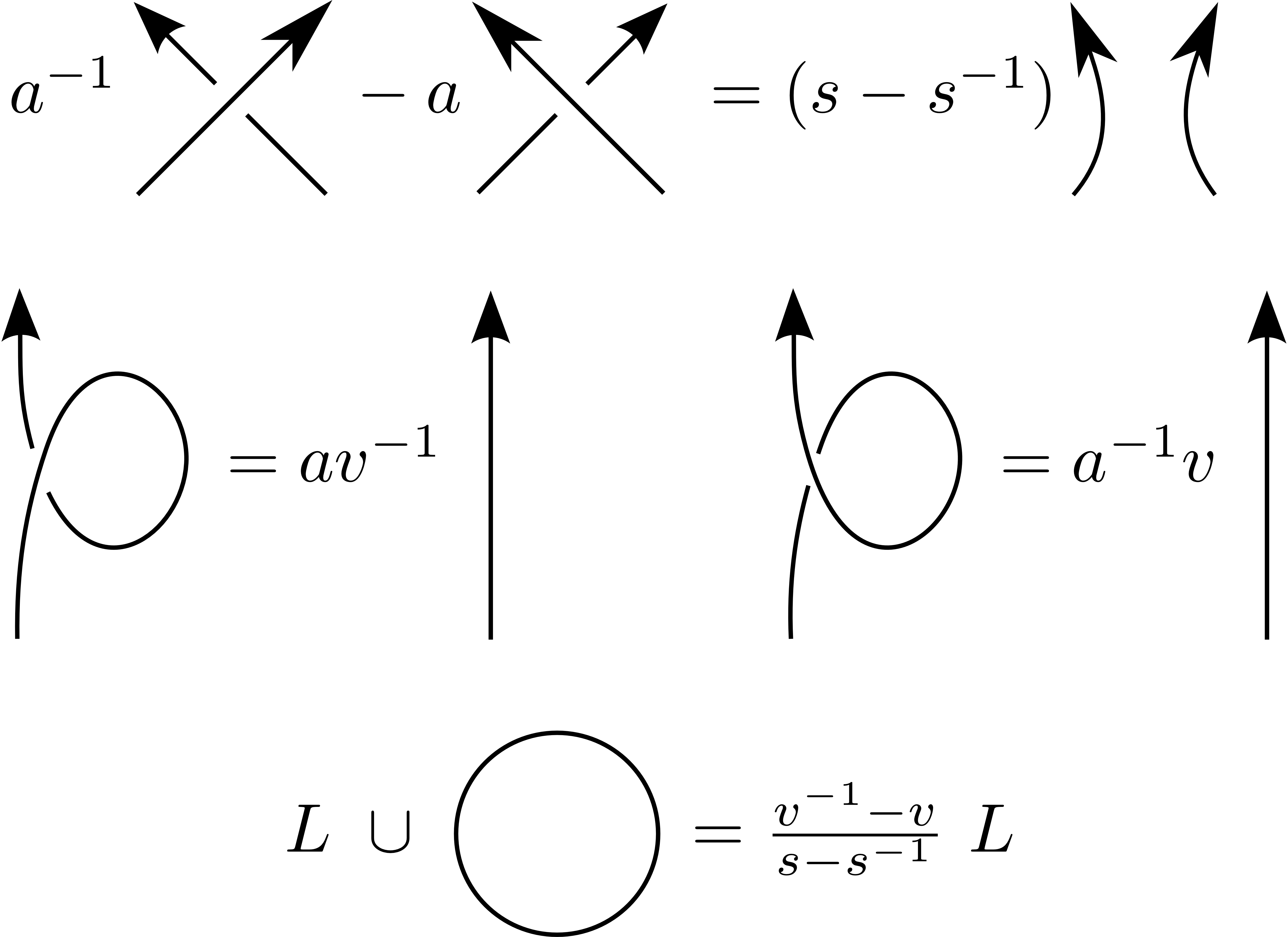}
  \caption{The framed Homflypt relations.}
  \label{homflypt}
\end{figure}
Let $M$ be a compact oriented $3$-manifold. Let $K$ be an integral domain containing invertible elements $a$, $s$, and $v$ such that furthermore $s - s^{-1}$ is invertible. Recall that the Homflypt skein module $\calH(M)$ is the $K$-module generated by isotopy classes of framed links in $M$ modulo the framed Homflypt relations in Figure~\ref{homflypt}. Here, a framed link is an oriented link together with a non-singular normal vector field up to isotopy. As always, there is also the concept of a relative Homflypt skein module; we only use this in Lemma~\ref{smatrix} where its meaning should be obvious (see \cite[p.~195]{Bla} for a precise definition). 

Let $r = k + N$. In the following we will assume that $K = \bbC$, that
\begin{align}
  \label{valgafa}
  a = \exp(-2 \pi i/(2Nr)),
\end{align}
that $s = a^{-N}$, and that $v = s^{-N} = a^{N^2}$. The quantum representations, and indeed the quantum invariants, will depend on these choices and so let us comment on them further: The choices of $v$ and $s$ are natural as these choices are exactly the ones used in \cite{Bla} to construct from the Homflypt module a modular functor. In this construction, it is further required that $a$ is a primitive $2Nr$'th root of unity, and as such, we have some freedom in the choice and the particular choice made here is made only to ensure the connection with the $\PSL(2,\bbZ)$-representations arising in conformal field theory and quantum groups; we elaborate on this in Appendix~\ref{normaliseringer} where we show through explicit calculation that the $\PSL(2,\bbZ)$-representation from Blanchets theory agrees with those of conformal field theory in the case $N = 2$. An important remark here is that these choices will in fact \emph{not} specialize exactly to the corresponding theory for the Kauffman bracket discussed in \cite{BHMV1}, \cite{BHMV2}: The skein formula of that theory comes from using the framed Homflypt relations in Figure~\ref{homflypt}, letting $a = A^{-1}$, $s = A^2$, $v = -A^4$, where here $A$ is the usual Kauffman bracket variable. The two quantum $\SU(2)$-invariants of 3-manifolds without links turn out to be exactly the same as all sign discrepancies appear twice; see \cite[p.~238]{Ohtbog}. One difference, however, is that no specializations of the variables $A$ and $a$ will result in the same invariants of manifolds containing links -- unless the link in question contains an even number of components, as can be deduced from \cite[App.~H]{Ohtbog}, noting that the quantum group skein relations with parameter $q$ for $(\fraksl_2,V)$, $V$ being the vector representation, are exactly those obtained from the Homflypt relations, letting $q = a^{-1/4}$; this in itself is probably reason enough to believe that the above concrete choice of $a$ is the natural one, as it tells us that the resulting $\SU(2)$-invariants agree with those of Reshetikhin and Turaev, constructed to agree with the physical picture. Notation being settled, let $Z_k^{\SU(N)}(M,L)$ denote the \emph{quantum $\SU(N)$-invariant} of a $3$-manifold (with a choice of $2$-framing) containing a link $L$, as defined in Theorem~2.10 of \cite{Bla} together with the comments following Lemma~2.12 of the same paper. These invariants extend to TQFT functors $(V_k,Z_k) = (V_k^{\SU(N)},Z_k^{\SU(N)})$ on the cobordism category of $3$-manifolds (endowed with $2$-framings).

We are now able to describe the construction of $(V_k,\rho_k)$. Let $\Sigma$ be a standard oriented surface in $S^3$ bounding a standard handlebody $H$. The vector space $V_k(\Sigma)$ can be defined as a certain quotient of $\calH(H)$ and is spanned by pairs $(H,L_\alpha)$ for particular (isotopy classes of) framed links $L_\alpha$. In fact, by the axioms of TQFT, any pair $(M,L)$ with $\dd M = \Sigma$ gives rise to an element of $V_k(\Sigma)$.

We now describe the \emph{quantum representation} $\rho_k : \MCG(\Sigma) \to \bbP\Aut(V_k(\Sigma))$. Let $f \in \MCG(\Sigma)$ be a mapping class. The action $\rho_k(f) : V_k(\Sigma) \to V_k(\Sigma)$ of $f$ is given simply by
\begin{align}
  \label{quantumrepdef}
	\rho_k(f)(H,L_\alpha) = (H,L_\alpha) \cup_\Sigma C_\phi \in V_k(\Sigma),
\end{align}
where $\phi$ represents $f$ in $\MCG(\Sigma)$,
\begin{align*}
	C_\phi = (\Sigma \times [0,\tfrac{1}{2}]) \cup_{(x,\tfrac{1}{2}) \sim (\phi(x),\tfrac{1}{2})} (\Sigma \times [\tfrac{1}{2},1])
\end{align*}
is the mapping cylinder of $\phi$, and the gluing in \eqref{quantumrepdef} is performed along $\Sigma \times \{0\}$. This determines a well-defined projective representation $\rho_k$. In other words, $\rho_k(f)$ is nothing but the endomorphism of $V_k(\Sigma)$ determined in TQFT by the mapping cylinder $C_\phi$. From the TQFT axioms, it follows that the quantum $\SU(N)$-invariants of the mapping torus
\begin{align*}
	T_\phi = (\Sigma \times [0,1])/ (x,0) \sim (\phi(x),1)
\end{align*}
are given by
\begin{align*}
	Z_k (T_f) = \tr Z_k(C_f) = \tr \rho_k(f),
\end{align*}
where we now ignore the (non-)dependence on the particular choice of representative of $f$ in our notation. We will further be interested in the case where the mapping torus $T_\phi$ contains a link $L$ avoiding some fiber $\Sigma_0$ of $T_f$ viewed as a $\Sigma$-bundle over $S^1$. In this case, $Z_k(T_f,L)$ can once again be understood through cutting $T_f$ along the fiber $\Sigma_0$, resulting in a mapping torus $(C_f,L)$, whose induced TQFT-endomorphism $Z_k(C_f,L)$ can be understood via its action on $(H,L_\alpha)$ as above. As is well-known, the TQFT under consideration extends to a so-called 3-2-1 TQFT which in practice means that we could also allow $L$ to intersect all fibers; as we only consider a particularly simple example we will not need this but hope to address this question in future work.

Now, we should note that the projective ambiguity in the definition of the quantum representations $\rho_k$ corresponds exactly (see e.g. Remark~\ref{framingrem} and \cite[IV.5.1]{Tu}) to the framing anomalies discussed in Remark~\ref{framings} and they will thus be ignored in the following. As is well-known, one can also obtain linear quantum representations $\tilde{\rho_k} : \widetilde{\MCG}(\Sigma) \to \Aut(V_k)$ of a central extension of the mapping class group.

\subsection{Invariants of torus mapping tori}
\label{invmaptori}
In this section, we fix the notation used throughout the rest of the paper. Let $\Sigma = S^1 \times S^1$ be the torus and let $\mu = \{1\} \times S^1, \nu = S^1 \times \{1\}$ be the meridinal resp. longitudinal curves in $\Sigma$, oriented so that their algebraic intersection number is $i(\mu,\nu) = 1$, and let $\Sigma$ have the orientation determined by the ordered pair $(\mu,\lambda)$. For a simple closed curve $\gamma$ in $\Sigma$, let $t_\gamma$ denote the left Dehn twist about $\gamma$, so that in the standard identification of $\MCG(\Sigma)$ with $\SL(2,\bbZ)$ via its action on homology, we have
\begin{align*}
	t_\mu = \begin{pmatrix} 1 & -1 \\ 0 & 1 \end{pmatrix}, \quad t_\nu = \begin{pmatrix} 1 & 0 \\ 1 & 1 \end{pmatrix}.
\end{align*}
Furthermore, let $\inv = -\id \in \SL(2,\bbZ)$ be the mapping class of elliptic involution. The TQFT vector space $V_k^{\SU(N)}(\Sigma)$ associated to $\Sigma$ has a natural basis $(D^2 \times S^1,\hat{y}_\lambda) \in V_k(\Sigma) = \calH(D^2 \times S^1)/\sim$, consisting of handlebodies containing longitudinal framed links coloured by $\lambda \in \Gamma_{N,k}$. Here the labelling set $\Gamma_{N,k}$ consists of all Young diagrams $\lambda = (\lambda_1 \geq \cdots \geq \lambda_p \geq 1)$ with $\lambda_i \leq k$ and $p < N$, and $\hat{y}_\lambda$ is defined as in \cite[p.~200]{Bla}. For a Young diagram $\lambda$ as above, we index its cells $(i,j)$ and  write
\begin{align*}
	\lambda = \{ (i,j) \mid 1 \leq i \leq p, \, 1 \leq j \leq \lambda_i\}.
\end{align*}

A simple linear algebra argument shows that any element of $\SL(2,\bbZ)$ with trace $\pm 2$ is conjugate to
\begin{align*}
	\begin{pmatrix} \pm 1 & -b \\ 0 & \pm 1 \end{pmatrix}
\end{align*}
for some $b \in \bbZ$. In particular, since $\inv$ is central, the general Dehn twist relation $f t_\gamma f^{-1} = t_{f(\gamma)}$ (see e.g. \cite[Fact.~3.7]{FM}) together with the change of coordinate principle tells us that any mapping torus over a torus with monodromy of trace $\pm 2$ is homeomorphic to either $M^b = T_{t_\mu^b}$ or $\tilde{M}^b = T_{\inv t_\mu^b}$ for some $b \in \bbZ$. To understand the quantum invariants of these $3$-manifolds, it therefore suffices to describe $\rho_k(t_\mu)$ and $\rho_k(\inv)$.

The action of $t_\mu$ on $V_k(\Sigma)$ can be understood as in \eqref{quantumrepdef} via the surgery description of $Z_k$ (see also \cite{Rob} for the general idea). More precisely, $C_{t_\mu}$ can be described as surgery on a meridinal curve in $\Sigma \times \{\tfrac{1}{2}\} \subseteq \Sigma \times [0,1]$ with framing $1$ relative to $\Sigma \times \{\tfrac{1}{2}\}$. Denoting this curve by $L$, the definition of the $Z_k$ implies that
\begin{align*}
	Z_k(C_{t_\mu}) = \Delta^{-1} Z_k(\Sigma \times [0,1],L) \in \End(V_k(\Sigma)),
\end{align*}
where $\Delta$ is the Homflypt polynomial of the unlink with a single positive twist, coloured by Blanchet's surgery element $\omega$ (see \cite{Bla}). Thus, gluing the mapping cylinder $C_{t_\mu}$ to $(H,\hat{y}_\lambda)$ has the net effect of adding to $H$ a meridinal curve encircling $\hat{y}_\lambda$ coloured by $\omega$ with a single positive twist. Since $\omega$ is defined to behave well under Kirby moves, we can remove it at the cost of giving $\hat{y}_\lambda$ a \emph{negative} twist and multiplying the result by $\Delta$.

On the other hand, the result of giving $\hat{y}_\lambda$ a positive twist is given by \cite[Prop.~1.11~(b)]{Bla}. It follows that $\rho_k(t_\mu)$ is diagonal in the basis given by the $\hat{y}_\lambda$ and has entries
\begin{align}
  \label{tblanchet}
	(\rho_k(t_\mu))_{\lambda, \lambda'} = a^{-\abs{\lambda}^2 + N^2\abs {\lambda} + 2N \sum_{(i,j) \in \lambda} \mathrm{cn}(i,j)} \delta_{\lambda,\lambda'},
\end{align}
where the sum is over all cells $(i,j)$ of the Young diagram $\lambda$, $\mathrm{cn}(i,j) = j-i$ is the \emph{content} of $(i,j)$, and $\abs{\lambda}$ denotes the number of cells in $\lambda$. In conclusion,
\begin{align}
  \label{dehntwistinvariant}
	Z_k(T_{t_\mu^b}) = \sum_{\lambda \in \Gamma_{N,k}} (\rho_k(t_\mu^b))_{\lambda,\lambda} = \sum_{\lambda \in \Gamma_{N,k}} a^{b(-\abs{\lambda}^2 + N^2\abs {\lambda} + 2 N\sum_{(i,j) \in \lambda} \mathrm{cn}(i,j))}.
\end{align}

The action of $\inv$ on $V_k(\Sigma)$ can be described similarly. We refer to \cite[IV.5.4]{Tu} for the topological details (notice that Turaev uses the opposite torus orientation but that this makes no difference in representing $\inv$ -- for the twist, one has to be a bit more careful) and note that the result is
\begin{align*}
	(\rho_k(\inv))_{\lambda,\lambda'} = \delta_{\lambda,(\lambda')^\star},
\end{align*}
where $^\star$ denotes the involution on $\Gamma_{N,k}$ defined in \cite[p.~206]{Bla}. In particular this implies that for $N = 2$, $\inv$ is in the (projective) kernel of every $\rho_k$. In general, it follows that
\begin{align*}
	Z_k(T_{\inv t_{\mu}^b}) &= \tr Z_k(C_{\inv t_\mu^b}) = \tr(\rho_k(\inv) \rho_k(t_\mu^b)) = \tr \rho_k(t_\mu)^b = \sum_{\lambda \in \Gamma_{N,k}} \rho_k(t_\mu)_{\lambda,\lambda}^b \delta_{\lambda,\lambda^\star},
\end{align*}
and in particular that
\begin{align}
  \label{su2invo}
	Z_k^{\SU(2)}(T_{\inv t_{\mu}^b}) = Z_k^{\SU(2)}(T_{t_{\mu}^b}).
\end{align}

\section{The case $G = \SU(2)$}
\label{su2afsnit}
\subsection{The quantum invariants}
The main technical tool we need is a generalization of quadratic reciprocity by Jeffrey \cite{Jef} which we recall here. For proofs, see \cite{Jefphd}, \cite[App.]{HT}.
\begin{thm}[Jeffrey]
  \label{quadrecisun}
  Let $r \in \bbZ$, let $V$ be a real vector space of dimension $l$ with an inner product $\langle \, , \,\rangle$, and let $\Lambda$ be a lattice in $V$ with dual lattice $\Lambda^*$. Let $B : V \to V$ be a self-adjoint automorphism and let $\psi \in V$. Assume that
  \begin{align*}
		\tfrac{1}{2} r\langle \lambda , B \lambda \rangle, \, \langle \lambda , B \eta \rangle, \, r \langle \lambda , \psi \rangle, \, \tfrac{1}{2}r\langle \mu,B\mu\rangle, \, r\langle\mu,\xi\rangle,\, r\langle\mu,\psi \rangle \in \bbZ
  \end{align*}
  for all $\lambda , \eta \in \Lambda$, $\mu, \xi \in \Lambda^*$. Then
  \begin{align*}
		\vol(\Lambda^*)& \sum_{\lambda \in \Lambda/r\Lambda} \exp(i\pi \langle \lambda,B\lambda/r\rangle)\exp(2\pi i \langle\lambda, \psi\rangle) \\
		  &= \left( \det \frac{B}{i}\right)^{-1/2} r^{l/2} \sum_{\mu \in \Lambda^* / B \Lambda^*} \exp(-\pi i \langle \mu + \psi , r B^{-1}(\mu + \psi) \rangle).
  \end{align*}
\end{thm}
An immediate corollary of Jeffrey's quadratic reciprocity theorem is the following well-known formula for generalized Gauss sums (where we correct a minor typo in \cite{Jef}).
\begin{thm}
  \label{quadreci}
  Let $a,b,c$ be integers, $a \not= 0$, $c \not= 0$, and assume that $ac+b$ is even. Then
  \begin{align*}
    \sum_{n=0}^{\lvert c \rvert -1} \exp(\pi i(an^2+bn)/c) = \lvert c/a \rvert^{1/2} \exp\left(\pi i \frac{\lvert ac \rvert-b^2}{4ac}\right) \sum_{n=0}^{\lvert a\rvert-1} \exp(-\pi i (c n^2+b n)/a).
  \end{align*}
\end{thm}

As it does not complicate matters to any greater extent and might shed light on the asymptotic expansion conjecture for links (which to the authors' knowledge has never been formulated explicitly), we will be interested in pairs $(M^b,L_\lambda)$, where $L_\lambda$ is a $\lambda$-coloured link parallel to the meridinal curve used in the surgery description of $C_{t_\mu}$. More concretely, $(M^b,L_\lambda)$ has the surgery description shown in Figure~\ref{kirurgipic}.

Note that the pairs $(M^b,L_\lambda)$, $b \not= 0$, are exactly of the type that Beasley considers in \cite{Bea}. It could be interesting to compare the work of Beasley with the present results, but we make no such attempt.

Elements of $\Gamma_{2,k}$ correspond simply to integers $j$, $0 \leq j \leq k$, and we will simply write $j$ for the Young diagram containing $j$ cells.

\begin{thm}
  \label{stnmedlinks}
  For $k \geq j$ and $b \not= 0$, we have
  \begin{align*}
		Z_k^{\SU(2)}(M^b,L_j) = & \,\exp\left(\frac{\pi i b}{2r}\right) \left( \sqrt{\frac{r}{2\lvert b\rvert}}\exp(-\pi i \mathrm{sgn}(b)/4) \right.\\
		&\left.\cdot \Bigg[ \sum_{n=0}^{\lvert b \rvert - 1} \exp\left(2\pi ir\frac{n^2}{b}\right) \sum_{l=0}^j \exp\left(2 \pi i\left(\frac{(2l-j)^2}{4br} + \frac{(2l-j)n}{b}\right) \right)\Bigg]  \right. \\
		  &\left. -\frac{j+1}{2} - \frac{(-1)^j(j+1)}{2}\exp\left(-\frac{\pi i}{2}b r\right) \right).
  \end{align*}
  where $r  = k+2$.
\end{thm}
\begin{rem}
  Recall that $\chi_j$, the composition of the character of the $(j+1)$-dimensional irreducible representation of $\SU(2)$ with $\exp(i \cdot)$ (see e.g. \cite[p.~143]{Fol}), is given by
  \begin{align*}
		\chi_j(t) = \sum_{l=0}^j \exp(i (2l-j)t) =  \begin{cases} j+1 & \text{if $t = 0$,} \\ (-1)^{j+1}(j+1) & \text{if $t = \pi$,} \\ \frac{\sin((j+1)t)}{\sin(t)} & \text{otherwise,} \end{cases}
  \end{align*}
  for $t \in [0,2\pi)$. Combining this with the description of the moduli space of flat connections of the manifolds $M^b$ in terms of their holonomies, given in the proof of Proposition~\ref{su2modulirum} below, one can directly relate the quantum invariant contributions coming from the links $L_j$ to the holonomies of the corresponding connections.
\end{rem}
\begin{proof}[Proof of Theorem~\ref{stnmedlinks}]
First of all, note that for $z \not= 0, \pm 1$,
\begin{align}
  \label{sinuspjat}
	\frac{z^{(j+1)}-z^{-(j+1)}}{z-z^{-1}} = z^{-j} \frac{z^{2j+2}-1}{z^2-1} = \sum_{l=0}^j z^{2l-j}.
\end{align}
\begin{figure}
  \centering
  \includegraphics[scale=0.7]{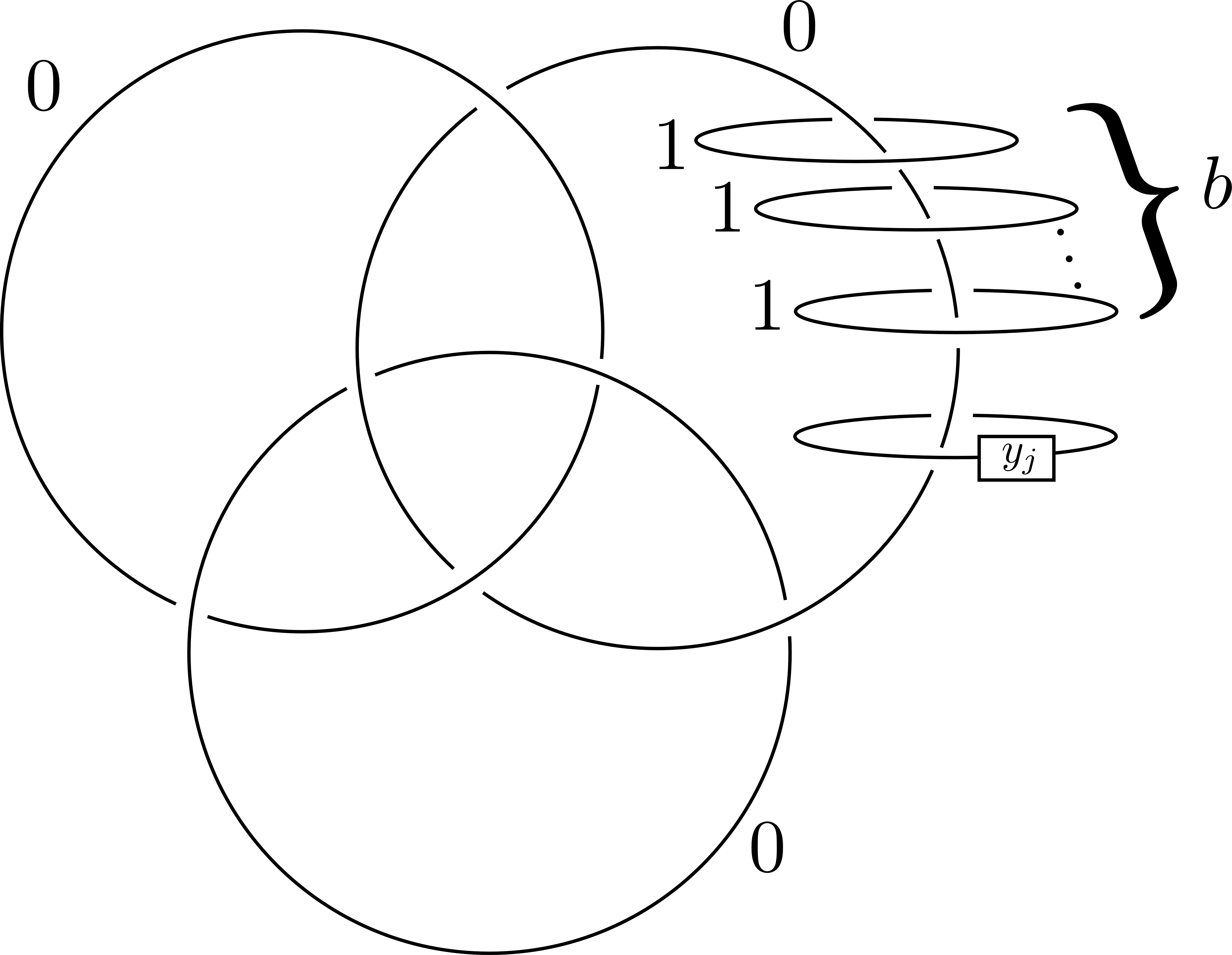}
  \caption{Surgery description of $(M,L_j)$. In this picture a link component with a number next to it means surgery along that component with framing according to the number; this should not be confused with the colourings used for idempotents (cf. also Appendix~\ref{normaliseringer}), such as the one on the component coloured $y_j$ in this picture.}
  \label{kirurgipic}
\end{figure}
In the description of mapping cylinders as acting on $V_k(S^1 \times S^1)$ via gluing, adding the link $L_j$ corresponds to applying the so-called curve operator, acting on the basis described in Section~\ref{invmaptori} by encircling the $\hat{y}_l$ by an unknot coloured $y_j$. This action is diagonal and described in Lemma~\ref{smatrix}. Applying the formula of Lemma~\ref{smatrix} and \eqref{sinuspjat} with $z = \exp(\pi i n/r)$, we therefore find that
\begin{align*}
	Z_k(M^b,L_j) &= \sum_{n=0}^{k} a^{b(n^2+2n)}\frac{a^{2(n+1)(j+1)}-a^{-2(n+1)(j+1)}}{a^{2(n+1)}-a^{-2(n+1)}} \\
	  &= \exp\left(\frac{\pi i b}{2r}\right) \sum_{n=1}^{r-1} \sum_{l=0}^j \exp\left(\frac{\pi i}{r}(-\tfrac{1}{2}n^2 b + 2ln - jn)\right) \\
	  &= \exp\left(\frac{\pi i b}{2r}\right) \sum_{l=0}^j  \sum_{n=1}^{r-1} \exp\left(\frac{\pi i}{2r}(-n^2b + 4ln - 2jn)\right).
\end{align*}
Quadratic reciprocity tells us that
\begin{equation}
\label{gaussreciigen}
\begin{aligned}
	\sum_{n=1}^{2r-1} \exp\left(\frac{\pi i}{2r}(bn^2+4ln - 2jn)\right) =&\,\sqrt{\frac{2r}{\lvert b \rvert}}\exp\left(\frac{\pi i}{4}\mathrm{sgn}(b)\right)\exp\left(-\frac{\pi i(4l-2j)^2 }{8br}\right)  \\
	&\,\cdot \left( \sum_{n=0}^{\lvert b \rvert -1} \exp\left(-\frac{\pi i}{b}(2rn^2+(4l-2j)n)\right) \right)-1.
\end{aligned}
\end{equation}
It now suffices to relate the sum on the left hand side of \eqref{gaussreciigen} to the corresponding one having upper limit $r-1$. Note now that mod $4r$, we have $bn^2 \equiv b(2r-n)^2$ and
\begin{align*}
	4ln-2jn \equiv -(4l-2j)(2r-n).
\end{align*}
Making in the process a change of variables $n \to r-n$, we conclude from this that
\begin{align*}
\sum_{n=r+1}^{2r-1} \exp\left(\frac{\pi i}{2r}(bn^2+4ln - 2jn)\right) &= \sum_{n=1}^{r-1} \exp\left(\frac{\pi i}{2r}(b(n+r)^2+(4l - 2j)(n+r))\right) \\
  &=\sum_{n=1}^{r-1} \exp\left(\frac{\pi i}{2r}(b(2r-n)^2+(4l - 2j)(2r-n))\right) \\
  &= \sum_{n=1}^{r-1} \exp\left(\frac{\pi i}{2r}(bn^2 - (4l-2j)n)\right).
\end{align*}
From this it follows that
\begin{align*}
	\sum_{n=1}^{2r-1} &\exp\left(\frac{\pi i}{2r}(bn^2+4ln - 2jn)\right) \\
	=&\, \sum_{n=1}^{r-1} \exp\left(\frac{\pi i}{2r}(bn^2+4ln - 2jn)\right) + (-1)^j\exp\left(\frac{\pi i}{2}br\right) \\
	&\,+ \sum_{n=r+1}^{2r-1} \exp\left(\frac{\pi i}{2r}(bn^2+4ln - 2jn)\right) \\
	=&\, \sum_{n=1}^{r-1} \exp\left(\frac{\pi i}{2r}bn^2\right)\left(\exp\left(\frac{\pi i}{2r}(4ln-2jn)\right) + \exp\left(-\frac{\pi i}{2r}(4ln-2jn)\right)\right) \\
	&\,+ (-1)^j \exp\left(\frac{\pi i}{2}br\right).
\end{align*}
Note now that even though the individual terms are not, the sum
\begin{align*}
	\sum_{l=0}^j \exp\left(\frac{\pi i}{2r}(4ln-2jn)\right)
\end{align*}
is real as in \eqref{sinuspjat}. We therefore find that
\begin{align*}
	2\sum_{l=0}^j \sum_{n=1}^{r-1} \exp\left(\frac{\pi i}{2r}(bn^2+4ln-2jn)\right) =&\, \sum_{l=0}^j \left( \sum_{n=1}^{2r-1} \exp\left(\frac{\pi i}{2r}(bn^2+4ln - 2jn)\right) \right)\\
	&\, - (-1)^j \exp\left(\frac{\pi i}{2}br\right),
\end{align*}
and combining all of this, performing an overall conjugation, we obtain the claim of the theorem.
\end{proof}
Letting $j = 0$ in Theorem~\ref{stnmedlinks} and using \eqref{su2invo}, we immediately obtain the following formula for the quantum $\SU(2)$-invariants of the link-less manifolds.
\begin{cor}
  \label{stnudenlinks}
  For $b \not= 0$, we have
  \begin{align*}
	  Z_k^{\SU(2)}(M^b) = Z_k^{\SU(2)}(\tilde{M}^b) = \exp\left(\frac{\pi ib}{2r}\right) &\Bigg(\sqrt{\frac{r}{2\abs{b}}}\exp(-\pi i \sgn(b)/4) \sum_{n=0}^{\abs{b}-1}\exp(2\pi in^2/b)\\
	  &\,\,\,\,- \frac{1}{2} - \frac{\exp(-\pi i r b/2)}{2}\Bigg).
  \end{align*}
\end{cor}
Plots of various values of this invariant are contained in Appendix~\ref{plots}.

\subsection{The mapping torus moduli space}
In the case of $G  = \SU(2)$, we can describe the moduli space of flat connections completely explicitly and find the following.
\begin{prop}
  \label{su2modulirum}
	Let $\gamma$ be the isotopy class of an essential simple closed curve in $\Sigma_1$, and let $b \in \bbZ$, $b \not = 0$. The moduli space $\calM$ of flat $\SU(2)$-connections on $M^b$ can be described as follows:
  
  For $b$ odd, it consists of a copy of the pillowcase, $\tfrac{\lvert b\rvert -1}{2}$ copies of the 2-torus $T^2$, as well as a component containing a single point.
  
  For $b$ even, it consists of 2 copies of the pillowcase and $\tfrac{\lvert b \rvert }{2}-1$ copies of $T^2$.
  
  The only irreducible connection is the one in the single point component for $b$ odd. On the various components, the Chern--Simons action takes the following (not necessarily distinct) values:
  \begin{align*}
    \mathrm{CS}(\calM) = \left\{ \begin{array}{ll} \{\tfrac{j^2}{b} \mid j = 0, \dots, \tfrac{\lvert b \rvert -1}{2}\}\cup \{1-\tfrac{b}{4}\} &\text{ if $b$ is odd,} \\ \{\tfrac{j^2}{b} \mid j = 0, \dots, \tfrac{\lvert b \rvert}{2}\} & \text{ if $b$ is even.}\end{array} \right.
  \end{align*}
\end{prop}
In this statement, the term \emph{pillowcase} refers to the space $I \times I/(x \sim -x)$, which is also the $\SU(2)$-moduli space of the torus and as a topological space is homeomorphic to $S^2$.
\begin{proof}
We can describe $\calM$ by describing the representations of $\pi_1(T_{t_\gamma^b})$. In the following, we consider $m = -b$ as this lowers the total number of signs.

It is well-known (see e.g. \cite{Jef}) that for a mapping torus $T_\phi$, $\phi : \Sigma \to \Sigma$, the fundamental group is given by the twisted product
\begin{align*}
  \pi_1(T_\phi) = \bbZ \rtimes \pi_1(\Sigma),
\end{align*}
where $\bbZ$ acts on $\pi_1(\Sigma)$ via $\phi$. In our special case, the fundamental group therefore has the presentation
  \begin{align*}
    \pi_1(T_{t_\gamma^b}) = \langle \alpha, \beta, \delta \mid \alpha \beta = \beta \alpha, \delta \alpha \delta^{-1} = \alpha, \delta \beta \delta^{-1} = \alpha^m \beta \rangle.
  \end{align*}
  Here, we simply note that two essential closed curves are homotopic if and only if they are isotopic (see e.g. \cite[Prop.~1.10]{FM}), and we have simply let $\alpha$ be the homotopy class of any curve representing $\gamma$ and choose $\beta$ so that $i(\alpha,\beta) = 1$. The moduli space of flat connections is identified with a quotient of a subset of $\SU(2)^{\times 3}$ as
  \begin{align*}
    \calM \isom \{ (A,B,C) \in \SU(2)^{\times 3} \mid AB = BA, CAC^{-1} = A, CBC^{-1} = A^mB \}/\sim,
  \end{align*}
  where $\sim$ denotes simultaneous conjugation. Since $A$ and $B$ commute for $[(A,B,C)] \in \calM$, they both lie in the same maximal torus in $\SU(2)$, and by conjugating them simultaneously we may assume that they are both diagonal. In other words, they are both elements of $T := \U(1) \subseteq \SU(2)$. Here, for $a \in \U(1)$, we simply write $a$ for the matrix $\diag(a,\bar{a})$ in $\SU(2)$. We now consider three cases.
  
  \emph{Case 1}. Assume that $A,B \in Z(\SU(2))$. In this case, $B = A^m B$, so $A^m = 1$, and so $A$ must be the identity if $m$ is odd.
  
  \emph{Case 2}. Assume that $A \notin Z(\SU(2))$. Then $C \in N(T)$, where $N(T)$ is the normalizer of $T$, which is given by $N(T) = T \cup L$, where
  \begin{align*}
    L = \left\{ \left( \begin{matrix} 0 & \exp(2\pi it) \\ -\exp(-2\pi it) & 0 \end{matrix} \right) \relmiddle| t \in \bbR \right\}.
  \end{align*}
  If $C \in T$ then $B = BA^m$, and $A^m = 1$ is the only restriction. If $C \in N(T) \setminus T = L$, conjugation by $C$ corresponds to inversion of elements of $T$. Thus, for $C \in L$, we have $A^{-1} = A$ contradicting that $A \notin Z(\SU(2))$.
  
  \emph{Case 3}. Assume that $A \in Z(\SU(2)), B \notin Z(\SU(2))$. Again, $C \in N(T)$. If $C \in T$ we find again that $A^m = 1$, so $A = 1$ if $m$ is odd. If $C \in L$, then $B^{-1} = BA^m$, and $B^2 = A^{-m} = A^m$, which is impossible for $m$ even when $B \notin Z(\SU(2))$, but for $m$ odd, and $A = -1$, we get a contribution for $B = \pm i$.
  
  In conclusion, when $m$ is odd,
  \begin{align*}
    \calM \isom&\, ((\{1\} \times \{\pm 1\} \times \SU(2)) \cup (\{\exp(2\pi ij/m) \mid j = 0, \dots, \lvert m\rvert -1 \} \setminus \{1\} \times T \times T) \\
      &\cup (\{1 \} \times T \setminus \{\pm 1 \} \times T) \cup (\{ -1 \} \times \{ \pm i \} \times L))/\sim,
  \end{align*}
  and when $m$ is even,
  \begin{align*}
    \calM \isom&\, ((\{ \pm 1 \} \times \{ \pm 1 \} \times \SU(2)) \cup (\{\ \exp(2\pi ij/m) \mid j = 0, \dots, \lvert m\rvert -1 \} \setminus \{\pm 1\} \times T \times T) \\
      &\cup (\{\pm 1 \} \times T \setminus \{\pm 1 \} \times T))/ \sim.
  \end{align*}
  
  In the case where $m$ is odd, the last component is a union of two copies of $T$ where all points are identified under conjugation since
  \begin{align*}
    \left( \begin{matrix} 0 & \exp(2\pi is) \\ -\exp(-2\pi is) & 0 \end{matrix} \right) &\left( \begin{matrix} 0 & \exp(2\pi it) \\ -\exp(-2\pi it) & 0 \end{matrix} \right) \left( \begin{matrix} 0 & \exp(2\pi is) \\ -\exp(-2\pi is) & 0 \end{matrix} \right)^{-1} \\
      &= \left( \begin{matrix} 0 & \exp(4\pi i s - 2\pi it) \\ -\exp(-4\pi is + 2\pi i t) & 0 \end{matrix} \right).
  \end{align*}
  This is the single point component of $\calM$. If $m$ is odd or even, for the quotients of the first and third component in the above description, it suffices to consider the quotient of $\{1\} \times T \times T$ or $\{\pm 1\} \times T \times T$ respectively, since we may first identify any element of $\SU(2)$ with its diagonalization. Now $T \times T$ is identified with the torus $\bbR^2 / \bbZ^2$, and the only conjugation action left is the action by the Weyl group $W \isom \bbZ_2$ which acts on $\bbR^2 / \bbZ^2$ by $(t,s) \mapsto (-t,-s)$. Therefore, $\calM$ contains one or two of the pillowcase in the cases $m$ odd or even respectively.
  
  Finally, let $j \in \{0, \dots, \lvert m\rvert -1\}$, and assume that $j/m \notin \{ 0, \tfrac{1}{2}\}$. Arguing as above, the only conjugation action left on $\{\exp(2\pi i j/m)\} \times T \times T$ is that of the Weyl group. Now, in this case, it acts non-trivially on the first factor, mapping $\exp(2\pi i j/m)$ to $\exp(-2\pi i j/m)$, and the resulting quotient becomes a number of copies of $T \times T$ as claimed.
  
  Finding the values of the Chern--Simons action for $G = \SU(2)$ is a well-studied problem, and in our case, the values on the components can be found immediately using e.g. methods of \cite{Jef}. The result can also be seen as a special case of Proposition~\ref{suncs}, where we elaborate on the available techniques.

  The claim about reducibility follows from the fact that an $\SU(2)$-connection is reducible if and only if the corresponding representation has image contained in a maximal torus, which is the case for all representations above but the one mapping $C$ into $L$.
\end{proof}

\begin{cor}
  \label{su2cs}
  The asymptotic expansion conjecture holds for $M^b$ and $\tilde{M}^b$ for all $b \in \bbZ$, $b \not = 0$.
\end{cor}
\begin{proof}
  As in the proof of Theorem~\ref{stnmedlinks}, we note that
  \begin{align}
    \label{omskrivningexp}
    \exp(2\pi ir((b-1)-(j-1))^2/b) = \exp(2\pi irj^2/b),
  \end{align}
  for $j = 0, \dots, \lvert b\rvert-1$. To prove the corollary, it is now a matter of rearranging the terms in the formula for $Z_k(T_{t_\gamma^b})$, obtained in Corollary~\ref{stnudenlinks}.
  
  Assume first that $b$ is even. In this case,
  \begin{align*}
    \exp(-\pi i rb/2) = \exp(\pi i rb/2) = \exp\left(2\pi i r \left(\frac{\lvert b \rvert}{2}\right)^2\frac{1}{b}\right),
  \end{align*}
  and it follows from \eqref{omskrivningexp} that
  \begin{align*}
    Z_k(M^b) =&\, \exp\left(\frac{\pi i b}{2r}\right) \Bigg(\sqrt{\frac{r}{2\lvert b \rvert}} \exp(- \sgn(b) \pi i /4) \left(2 \sum_{n=1}^{{\lvert b \rvert}/2 - 1} \exp(2\pi i rn^2/b) \right.  \\
      & \left.\phantom{\exp\left(\frac{\pi i b}{2r}\right) \left( \right)}+ 1 + \exp(\pi i b r/2) \right) - \frac{\exp(-\pi i rb/2)}{2} - \frac{1}{2}\Bigg) \\
     =&\, \exp\left(\frac{\pi i b}{2r}\right) \Bigg( \sum_{n=1}^{\lvert b \rvert/2 - 1} \exp(2\pi i rn^2/b)\left[ \sqrt{\frac{2r}{\lvert b \rvert}} \exp(-\sgn(b) \pi i /4) \right] \\
     &  \phantom{\exp\left(\frac{\pi i b}{2r}\right) \left( \right) } \,+\exp(2\pi i r\cdot 0/b) \left[ \sqrt{\frac{r}{2\lvert b \rvert}} \exp(-\sgn(b) \pi i/4) - \frac{1}{2} \right] \\
      & \phantom{\exp\left(\frac{\pi i b}{2r}\right) \left( \right) } \, +\exp\left(2\pi i r\left( \frac{\lvert b \rvert}{2} \right)^2 \frac{1}{b}\right) \left[\sqrt{\frac{r}{2 \lvert b \rvert}} \exp(-\sgn(b)\pi i /4) - \frac{1}{2} \right] \Bigg). 
  \end{align*}
  Now, one obtains the full asymptotic expansion of $Z_k(M^b)$ by introducing the Taylor series for $\exp(\pi i b/(2r))$ and $1/\sqrt{r}$, and one obtains the conjecture by comparing the resulting expression with the result of Proposition~\ref{su2modulirum}. For $b$ odd, the exact same argument shows that
  \begin{align*}
    Z_k(M^b) =&\, \exp\left(\frac{\pi i b}{2r}\right) \Bigg( \sum_{n=1}^{(\lvert b \rvert -1)/2} \exp(2\pi i rn^2/b) \left[ \sqrt{\frac{2r}{\lvert b \rvert}} \exp(-\sgn(b)\pi i/4) \right]  \\
     & + \, \exp(2\pi i r 0/b)\left[ \sqrt{\frac{r}{2\lvert b \rvert}} \exp(-\sgn(b)\pi i/4) - \frac{1}{2}\right] - \exp(-\pi i rb/2) \frac{1}{2} \Bigg),
  \end{align*}
  and once again the claim follows from Proposition~\ref{su2modulirum}.
  
  The argument for $\tilde{M}^b$ is identical as one finds the exact same Chern--Simons values for these manifolds. We omit the details.
\end{proof}

Note that the proof of Corollary~\ref{su2cs} gives us explicitly the leading order term of $Z_k(M^b)$, and in particular we are now able to turn to Conjecture~\ref{growthrateconjec} for $M^b$.

\subsection{The growth rate conjecture}
The main technical tool in proving the growth rate conjecture for the manifolds $M^b$ and $\tilde{M}^b$ is the correspondence between twisted deRham cohomology and group cohomology which we briefly recall.

Let $G$ be any group. A \emph{$G$-module} is an abelian group $N$ with a left action of $G$. The elements of $N$ invariant under the action will be denoted $N^G$. A \emph{cocycle on $G$ with values in $N$} is a map $u : G \to N$ satisfying the cocycle condition
\begin{align*}
  u(gh) = u(g) + gu(h).
\end{align*}
A \emph{coboundary} is a cocycle of the form $g \mapsto \delta m(g) := m - gm$ for some $m \in N$. The set of cocycles is denoted $Z^1(G,N)$, and the set of coboundaries is denoted $B^1(G,N)$. We define the first cohomology group of $G$ with coefficients of $N$ as the quotient
\begin{align*}
  H^1(G,N) = Z^1(G,N) / B^1(G,N).
\end{align*}
Notice that an element of $N$ satisfies $\delta m \equiv 0$ exactly when $m \in N^G$. We are led to define
\begin{align*}
  H^0(G,N) = N^G.
\end{align*}
Now, let $P \to M$ be a principal $G$-bundle over a closed connected oriented 3-manifold $M$, $G = \SU(N)$, and let $[A]$ be the gauge equivalence class of a flat connection in $P$, represented by a representation $\rho \in \Hom(\pi_1(M),G)$. The representation $\rho$ defines a $\pi_1(M)$-module structure on $\frakg = \Lie(G)$ through the composition $\Ad \circ \rho : \pi_1(M) \to \Aut(\frakg)$. The following result is well-known.
\begin{lem}
  \label{groupcohomologythm}
  If $M$ has contractible universal covering space, there are isomorphisms
  \begin{align*}
    H^0(M,\Ad_P) \isom H^0(\pi_1(M),\frakg), \quad H^1(M,\Ad_P) \isom H^1(\pi_1(M),\frakg).
  \end{align*}
\end{lem}

Consider the case where $M$ is the mapping torus of a homeomorphism of a surface of genus $g \geq 1$. Let $A$ be a flat connection in $P \to M$, and let $\rho$ be a representative of $[A]$ in the moduli space $\Hom(\pi_1(M),\SU(2))/\SU(2)$. The elements of $\su(2)$ fixed by the action of $\pi_1(M)$ given by $\Ad \circ \rho$ are exactly those in the centralizer of the image $\rho(\pi_1(M))$, and so by Lemma~\ref{groupcohomologythm},
\begin{align*}
  h^0_A = \dim \Lie(Z(\rho(\pi_1(M)))).
\end{align*}
Similarly, Lemma~\ref{groupcohomologythm} gives a description of $h^1_A$ using only the corresponding representation of $\pi_1(M)$.

\begin{thm}
  \label{growthratesu2}
  Let $\calM_{j/b}$, $j = 0, \dots, \big\lceil \tfrac{\lvert b\rvert +1}{2} \big\rceil$, and $\calM_{-b/4}$ be the components of the moduli space of $M^b$ arising from Proposition~\ref{su2modulirum}, and let
  \begin{align*}
    d_i' = \frac{1}{2} \max_{[A] \in \calM_i} (h^1_A - h^0_A),
  \end{align*}
  where the $\max$ is as in Conjecture~\ref{growthrateconjec}. Then $d_{j/b}' = \tfrac{1}{2}$ and $d_{-b/4}' = 0$. In particular, Conjecture~\ref{growthrateconjec} holds true in this case.
\end{thm}
\begin{proof}
  Again, we introduce $m = -b$. Abusing notation slightly, we write $\rho \in \calM$ for the (conjugacy class of a) representation corresponding to a (gauge class of a) flat connection in $\calM$. Let $A, B, C$ denote the images of generators $\alpha, \beta, \delta$ of $\pi_1(M^{m})$ under $\rho$. Using the remark following Lemma~\ref{groupcohomologythm}, we find that if $\rho \in \calM_{j/b}$, then $h^0_\rho = 1$ except in four or eight points in the cases where $m$ is odd or even respectively, those points corresponding to $A,B,C = \pm 1$. When $A,B, C = \pm 1$, we have $h^0_\rho = 3$. For the representation $\rho \in \calM_{-b/4}$, we have $h^0_\rho = 0$.
  
  We now describe $h^1_\rho$. The cocycles $Z^1(\pi_1(M^{-m}),\su(2))$ embed in $\su(2)^3$ under the map
  \begin{align*}
    u \mapsto (u(\alpha),u(\beta),u(\delta)).
  \end{align*}
  The image can be determined since cocycles map the three relators
  \begin{align*}
    R_1 = \alpha\beta \alpha^{-1}\beta^{-1}, \,\, R_2 = \alpha\delta\alpha^{-1}\delta^{-1}, \,\, R_3 = \delta\beta\delta^{-1}\alpha^{-m}\beta^{-1}
  \end{align*}
  of our presentation of $\pi_1(M^{-m})$ to $0 \in \su(2)^3$. One finds that $Z^1(\pi_1(M^{-m}),\su(2))$ can be identified with the kernel of the map $R = (\tilde{R}_1,\tilde{R}_2,\tilde{R}_3) : \su(2)^3 \to \su(2)^3$ determined by $R_1, R_2, R_3$ by the requirement that 
  \begin{align*}
    \tilde{R}_i(u(\alpha),u(\beta),u(\delta)) = u(R_i).
  \end{align*}
  Assume for simplicity that $m > 0$. Noting that in general,
  \begin{align*}
    u(g^{-1}) = -\Ad(\rho(g^{-1}))u(g),
  \end{align*}
  the cocycle condition gives
  \begin{align*}
    u(R_1) =&\,\, u(\alpha) - \Ad(B)u(\alpha) - u(\beta) + \Ad(A) u(\beta), \\
    u(R_2) =&\,\, u(\alpha) - \Ad(C)u(\alpha) - u(\delta) + \Ad(A) u(\delta), \\
    u(R_3) =&\,\, -\Ad(B)(\sum_{n=0}^m \Ad(A^n))u(\alpha) - u(\beta) \\
            &\,\,+ \Ad(C)u(\beta) + u(\delta) - \Ad(A^mB)u(\delta).
  \end{align*}
  Here, the first two equalities are immediate, and the last one follows from
  \begin{align*}
    u(R_3) =&\,\, u(\delta) + \Ad(C) u(\beta\delta^{-1} \alpha^{-m} \beta^{-1}) \\
      =&\,\, u(\delta) + \Ad(C)(u(\beta) + \Ad(B)u(\delta^{-1}\alpha^{-m}\beta^{-1})) \\
      =&\,\, u(\delta) + \Ad(C)u(\beta) + \Ad(CB)(u(\delta^{-1})+\Ad(C^{-1})u(\alpha^{-m}\beta^{-1})) \\
      =&\,\, u(\delta) + \Ad(C)u(\beta) - \Ad(CBC^{-1})u(\delta) \\
       &\, + \Ad(CBC^{-1})(u(\alpha^{-m})+\Ad(A^{-m})u(\beta^{-1})) \\
      =&\,\, u(\delta) + \Ad(C)u(\beta) - \Ad(A^mB)u(\delta) \\
       &\, - \Ad(CBC^{-1}A^{-m})u(\alpha^m) - \Ad(CBC^{-1}A^{-m}B^{-1})u(\beta) \\
      =&\,\, u(\delta) + \Ad(C)u(\beta) - \Ad(A^mB)u(\delta) - \Ad(B)u(\alpha^m) - u(\beta)
  \end{align*}
  since, in general
  \begin{align*}
    u(g^m) = \sum_{n=0}^{m-1} \Ad(\rho(g)^{n}) u(g).
  \end{align*}
  In other words, $R$ is given by
  \begin{align*}
    R&(x_1,x_2,x_3) = \left( \begin{matrix} x_1 - \Ad(B) x_1 - x_2 + \Ad(A)x_2 \\ x_1 - \Ad(C) x_1 - x_3 + \Ad(A)x_3 \\ -\Ad(B)(\sum_{n=0}^m \Ad(A^n) )x_1 - x_2 + \Ad(C) x_2 + x_3 - \Ad(A^mB)x_3 \end{matrix}\right).
  \end{align*}
  Under this identification, the coboundaries $B^1(\pi_1(M^{-m}),\su(2))$ become
  \begin{align*}
    \{(x - \Ad(A)x, x-\Ad(B)x,x-\Ad(C)x) \mid x \in \su(2)\} \subseteq \ker R \subseteq \su(2)^3.
  \end{align*}
  Recall that $\su(2)$ has a basis given by
  \begin{align*}
    \left\{ e_1 = \left( \begin{matrix} 0 & i \\ i & 0 \end{matrix} \right), e_2 = \left( \begin{matrix} 0 & -1 \\ 1 & 0 \end{matrix} \right), e_3 = \left( \begin{matrix} i & 0 \\ 0 & -i \end{matrix} \right) \right\}.
  \end{align*}
  Consider first the case $\rho \in \calM_{j/b}$. Write
  \begin{gather*}
    A = \left( \begin{matrix} \exp(2\pi ij/m) & 0 \\ 0 & \exp(-2\pi ij/m) \end{matrix} \right), \quad B = \left( \begin{matrix} \exp(2\pi i s) & 0 \\ 0 & \exp(-2\pi i s) \end{matrix} \right), \\
    C = \left( \begin{matrix} \exp(2\pi i t) & 0 \\ 0 & \exp(-2\pi i t) \end{matrix} \right)
  \end{gather*}
  for $j \in 0, \dots, \big\lceil \tfrac{m+1}{2} \big\rceil$, and $s,t \in [0,1)$. A direct computation shows that the matrix representation of $R$ in the basis given above is
  \begin{align*}
    R=\left( \begin{matrix}
      P - S(s) & -P+S(j/m) & 0 \\
      P - S(t) & 0 & -P+S(j/m) \\
      T(m,s) & -P + S(t) & P - S(s), \end{matrix} \right),
  \end{align*}
  where $P$, $S$, and $T$ are given by
  \begin{align*}
    P &= \left( \begin{matrix} 1 & 0 & 0 \\ 0 & 1 & 0 \\ 0 & 0 & 0 \end{matrix} \right), \quad S(r) = \left( \begin{matrix} \cos(4\pi r) & -\sin(4 \pi r) & 0 \\ \sin(4\pi r) & \cos(4 \pi r) & 0 \\ 0 & 0 & 0 \end{matrix} \right),\\
     T(m,s) &= \left( \begin{matrix} -\eta \cos(4\pi s) & \eta \sin(4\pi s) & 0 \\ -\eta \sin(4\pi s) & -\eta \cos(4\pi s) & 0 \\ 0 & 0 & -m \end{matrix} \right), \\
     \eta &= \sum_{n=0}^{m-1} \exp(4\pi ijn/m) = \left\{ \begin{array}{ll} 
         m,  & \mbox{if }j/m \in \{0,\tfrac{1}{2}\}, \\
         0, &\mbox{otherwise.} \end{array} \right.
  \end{align*}
  One finds that $\dim(\ker R) = 6$ when $\tfrac{j}{m},s,t \in\{0, \tfrac{1}{2}\}$ and that $\dim(\ker R) = 4$ otherwise. A similar computation shows that 
  \begin{align*}
    B^1(\pi_1(M^{-m}),\frakg) \isom \Span \left\{ \left( \begin{matrix} 1 - \cos(4\pi j/m) \\ -\sin(4\pi j/m) \\ 0 \\ 1- \cos(4\pi s) \\ -\sin(4\pi s) \\ 0 \\ 1-\cos(4\pi t) \\ -\sin(4\pi t) \\ 0 \end{matrix} \right), \left( \begin{matrix} \sin(4\pi j/m) \\ 1-\cos(4\pi j/m) \\ 0 \\ \sin(4\pi s) \\ 1-\cos(4\pi s) \\ 0 \\ \sin(4\pi t) \\ 1 -\cos(4\pi t) \\ 0 \end{matrix} \right),0 \right\},
  \end{align*}
  so the subspace of coboundaries has dimension $0$ when $\tfrac{j}{m}, s,t \in \{0, \tfrac{1}{2}\}$ and dimension $2$ otherwise. Notice that by definition of the generic $\max$, these finitely many special cases have no influence on $d_i'$.
  
  Now, consider the case of $\rho \in \calM_{-b/4}$, and write
  \begin{align*}
    A = \left( \begin{matrix} -1 & 0 \\ 0 & -1 \end{matrix} \right), \quad B = \left( \begin{matrix} i & 0 \\ 0 & -i \end{matrix} \right), \quad C = \left( \begin{matrix} 0 & 1 \\ -1 & 0 \end{matrix} \right).
  \end{align*}
  In this case, $R$ is given by
   \begin{align*}
    R= \left( \begin{matrix}
      2 & 0 & 0 & 0 & 0 & 0 & 0 & 0 & 0 \\
      0 & 2 & 0 & 0 & 0 & 0 & 0 & 0 & 0 \\
      0 & 0 & 0 & 0 & 0 & 0 & 0 & 0 & 0 \\
      2 & 0 & 0 & 0 & 0 & 0 & 0 & 0 & 0 \\
      0 & 0 & 0 & 0 & 0 & 0 & 0 & 0 & 0 \\
      0 & 0 & 2 & 0 & 0 & 0 & 0 & 0 & 0 \\
      m & 0 & 0 & -2 & 0 & 0 & 2 & 0 & 0 \\
      0 & m & 0 & 0 & 0 & 0 & 0 & 2 & 0 \\
      0 & 0 & -m & 0 & 0 & -2 & 0 & 0 & 0
      \end{matrix} \right).
  \end{align*}
  Now $\dim(\ker R) = 3$, and here we find that
  \begin{align*}
    B^1(\pi_1(M^{-m}),\frakg) \isom \Span \left\{ \left( \begin{matrix} 0 \\ 0 \\ 0 \\ 2 \\ 0 \\ 0 \\ 2 \\ 0 \\0 \end{matrix} \right), \left( \begin{matrix} 0 \\ 0 \\ 0 \\ 0 \\ 2 \\ 0 \\ 0 \\ 0 \\ 0 \end{matrix} \right) , \left( \begin{matrix} 0 \\ 0 \\ 0 \\ 0 \\ 0 \\ 0 \\ 0 \\ 0 \\ 2 \end{matrix} \right) \right\},
  \end{align*}
  so all cocycles are coboundaries. Relating $m$ to $b$, we obtain
  \begin{align*}
    d_{j/b}' &= \tfrac{1}{2}((4-2) - 1) = \tfrac{1}{2}, \\
    d_{-b/4}' &= \tfrac{1}{2}((3-3)-0) = 0.
  \end{align*}
\end{proof}

\subsection{Summary of torus bundles}
Having proved the $\SU(2)$-AEC for those torus homeomorphisms of trace $\pm 2$, we collect now what is known for torus bundles in general. Recall the trichotomy of $\MCG(S^1 \times S^1) \isom \SL(2,\bbZ)$ into elements of trace $\abs{\tr} < 2$ called periodic (or finite order), $\abs{\tr} = 2$ called reducible, and $\abs{\tr} > 2$ called hyperbolic (or Anosov). This, of course, not including the elements $\pm \id \in \SL(2,\bbZ)$ which are both finite order.

Jeffrey \cite{Jef} proves the AEC for all hyperbolic elements using the following concrete expression for the quantum invariants:
\begin{thm}[\cite{Jef}]
\label{jeffreythm}
  Let 
  \begin{align*}
    U = \left( \begin{matrix} a & b \\ c & d \end{matrix} \right) \in \Gamma_1 \isom \SL(2,\bbZ)
  \end{align*}
  and assume that $\lvert \Tr(U) \rvert > 2$. Then there exists a canonical choice of framing for $T_U$, and the quantum $\SU(2)$-invariant is given by
  \begin{align*}
    Z_k(T_U) = & \,\exp\left(\frac{2\pi i \psi(U)}{4r}\right) \sgn(d+a \mp 2) \sum_{\pm} \pm \frac{1}{2\lvert c \rvert \sqrt{\lvert d+a \mp 2 \rvert}} \\
      &\,\cdot \sum_{\beta = 0}^{\lvert c\rvert-1} \sum_{\gamma = 1}^{\lvert d +a \mp 2 \rvert} \exp \left( 2\pi i r \frac{-c \gamma^2 + (a-d)\gamma \beta + b \beta^2}{d + a \mp 2} \right),
  \end{align*}
  where $r = k+2$, and $\psi(U) \in \bbZ$ depends only on $U$ and is given by \cite[(4.4)]{Jef}.
\end{thm}
In fact, she gives another formula that can be used to obtain the quantum invariants of finite order elements as well. Up to conjugation, the only finite order elements are the identity and
\begin{align*}
  \inv = \begin{pmatrix} -1 & 0 \\ 0 & -1 \end{pmatrix}, f_3 = \begin{pmatrix} 0 & -1 \\ 1 & -1 \end{pmatrix}, f_4 = \begin{pmatrix} 0 & -1 \\ 1 & 0 \end{pmatrix}, f_6 = \begin{pmatrix} 0 & 1 \\ -1 & 1 \end{pmatrix},
\end{align*}
and their inverses (see \cite[p.~201]{FM}), indexed here by their order in $\SL(2,\bbZ)$. Using \cite[(4.7)]{Jef}, the quantum invariants of the corresponding mapping tori are (up to framing anomalies) given by
\begin{align*}
  Z_k(T_{\inv}) &= Z_k(T_{\id}) = r-1, \\
  Z_k(T_{f_3}) &= \frac{i}{2\sqrt{3}}(2\exp(-2 \pi i r /3) + 1) + \frac{1}{2} \exp(-\pi i/2), \\
  Z_k(T_{f_4}) &= \frac{1}{2}(\exp(\pi i r) + 1), \\
  Z_k(T_{f_6}) &= \frac{i}{2\sqrt{3}}(2\exp(2\pi i r /3) + 1) + \frac{1}{2},
\end{align*}
and similar expressions for the inverses. Thus, the AEC can readily be checked for finite order elements as well. We summarize these results in the following theorem\footnote{Note though that Jeffrey seems to use an orientation convention different than ours. Recall that reversing the orientation of the $3$-manifold conjugates the quantum invariants and reversing the sign of the Chern--Simons values.}.
\begin{thm}
  \label{summarythm}
  The asymptotic expansions of the $\SU(2)$-Witten--Reshetikhin--Turaev invariants of torus bundles $T_U$, $U \in \SL(2,\bbZ)$, are exact and in accordance with Conjecture~\ref{asympconjec}. The phases and growth rates of the invariants are summarized for the conjugacy classes of $\SL(2,\bbZ)$ in Table~\ref{fintabel}.
\end{thm}
\begin{table}[h]
  \centering
  \begin{tabular}{l|c|c}
  $U \in \SL(2,\bbZ)$ & $\{c_j\}$ &$ \{d_j\}$ \\
  \hline
    $\begin{pmatrix} \pm 1 & 0 \\ 0 & \pm 1 \end{pmatrix}$ & $\{0\}$ & $\{1\}$ \\
    $\begin{pmatrix} \pm 1 & -b \\ 0 & \pm 1 \end{pmatrix}$, $b \not= 0$ even & $\{\tfrac{j^2}{b} \mid j = 0, \dots, \tfrac{\abs{b}}{2} \}$ & $\{\tfrac{1}{2}\}$\\
    $\begin{pmatrix} \pm 1 & -b \\ 0 & \pm 1 \end{pmatrix}$, $b$ odd & $\{\tfrac{j^2}{b} \mid j = 0, \dots, \tfrac{\abs{b}-1}{2} \}\cup\{-\tfrac{b}{4}\}$ & $\{\tfrac{1}{2}\} \cup \{0\}$\\
    $\begin{pmatrix} a & b \\ c & d \end{pmatrix}, \,\, \abs{a+d} \not= 2$ &$\left\{ \frac{-c\gamma^2 + (a-d)\gamma\beta + b\beta^2}{d+a\pm 2} \relmiddle| \begin{aligned} &0 \leq \beta < c, \\ &0< \gamma \leq\abs{a+d\pm 2} \end{aligned}	 \right\}$ & $\{0\}$
  
  \end{tabular}
  \caption{Summary of phases and growth rates of quantum invariants of torus bundles.}
  \label{fintabel}
\end{table}

\section{The case $G = \SU(N)$}
\label{sunafsnit}

\subsection{The quantum invariants}
Throughout the following, we use the notation of Appendix~\ref{normaliseringer}. We will need the observation that
\begin{align*}
P_k + \rho &= \{ \mu = \lambda + \rho \mid \lambda \in P_+, \langle \lambda , \theta \rangle \leq k \} \\
	 &= \{ \lambda \in \mathrm{int}(P_+) \mid \langle \lambda , \theta \rangle < r \}.
\end{align*}
For this reason, let  $\tilde{P}_r = P_k + \rho$.
\begin{thm}
  \label{invariantssun}
  The level $k$ quantum $\SU(3)$-invariants of $M^b$, $b \not= 0$, are given by
  \begin{align*}
		Z_k(M^b) =&\, \exp(2\pi i/r) \Bigg( -\frac{\sqrt{3}i}{18b} r \sum_{n=0}^{3\abs{b}-1}\sum_{m=0}^{3\abs{b}-1} \exp\left(2\pi i r\frac{n^2+m^2-nm}{b}\right)\\
		&\,-\frac{1}{2}\sqrt{\frac{3}{2b}}\exp(-\pi i/4) \sqrt{r} \sum_{n=0}^{2\abs{b}-1}\exp\left(\pi ir \frac{3n^2}{2b}\right) \\
		&\, + \frac{1}{3} + \frac{2}{3} \exp\left( -2\pi i r \frac{b}{3}\right)\Bigg),
  \end{align*}
  where $r = k+3$.
\end{thm}
\begin{proof}
  Much of the proof holds for every $N$ and to illustrate how one might proceed in general, we specialize only to $N = 3$ when necessary.

  By \eqref{dehntwistinvariant} and \eqref{dehntwistcft}, it suffices to calculate
  \begin{align*}
		\sum_{\lambda \in P_k} \exp\left( b \frac{\pi i}{r} \langle \lambda + \rho , \lambda + \rho \rangle \right),
  \end{align*}
  and by the remark preceding the statement of the theorem, this sum may also be written as
  \begin{align}
    \label{sumibevis}
		\sum_{\lambda  \in \tilde{P}_r} \exp\left( b \frac{\pi i}{r} \langle \lambda , \lambda \rangle \right).
  \end{align}
  Let $g(\lambda) = \exp\left( b \frac{\pi i}{r} \langle \lambda , \lambda \rangle \right)$ for $\lambda \in \Lambda^w$, in the notation of Section~\ref{normaliseringer}.
  
  The idea of the following follows closely that of the similar theorem in \cite{Jef}. In this, Jeffrey tiles cubes in $\bbR^{N-1}$ by copies of $\tilde{P}_r$ (or more precisely, subsets of the weight space of the form $\{ \sum x_i \Lambda_i \in \bbR^{N-1} \mid x_i > 0, \sum_{i=1}^{N-1} x_i < r \}$), generated by the action of a particular $r$-dependent subgroup of the affine Weyl group, noting that for some large enough number of tiles, one is able to apply Theorem~\ref{quadrecisun} to the weights contained in the resulting cube, and it turns out that when summing $g(\lambda)$ over all weights in the entire cube, all tiles contribute the same; now, the main difference between what we do here and what is considered in \cite{Jef} is that in her proof, the boundaries of the tiles do not contribute to the sum, and so the calculation boils down to a combinatorial count of the tiles used so as to figure out the contribution from a single tile. In our calculation, the boundaries do contribute and so we need to be slightly more careful; the aim is to apply Theorem~\ref{quadrecisun} with $\Lambda = \Lambda^w$, $r = 2rN$, $B = 2bN\Id$, $\psi = 0$, tiling a cube of side lengths $2Nr$ in $\bbR^{N-1}$ by simplices as above. See Figure~\ref{su3tiling} for the case $N = 3$; here, the cube of side lengths $6r$ has been split up into $72$ $2$-dimensional triangles as well as a number of lower dimensional simplices. The point now is that the values of $g(\lambda)$ on $\Lambda^w/2rN\Lambda^w$ are determined by the values on a single tile, i.e. a single $r$-alcove and its boundary. This follows from the fact that $g(\lambda)$ is invariant under the action of the Weyl group and the translations $\lambda \mapsto \lambda + 2h_\alpha$, for any simple root $2h_\alpha$. The group generated by these actions is exactly the group of reflections in faces of the $r$-alcove.
  \begin{figure}
    \centering
    \includegraphics{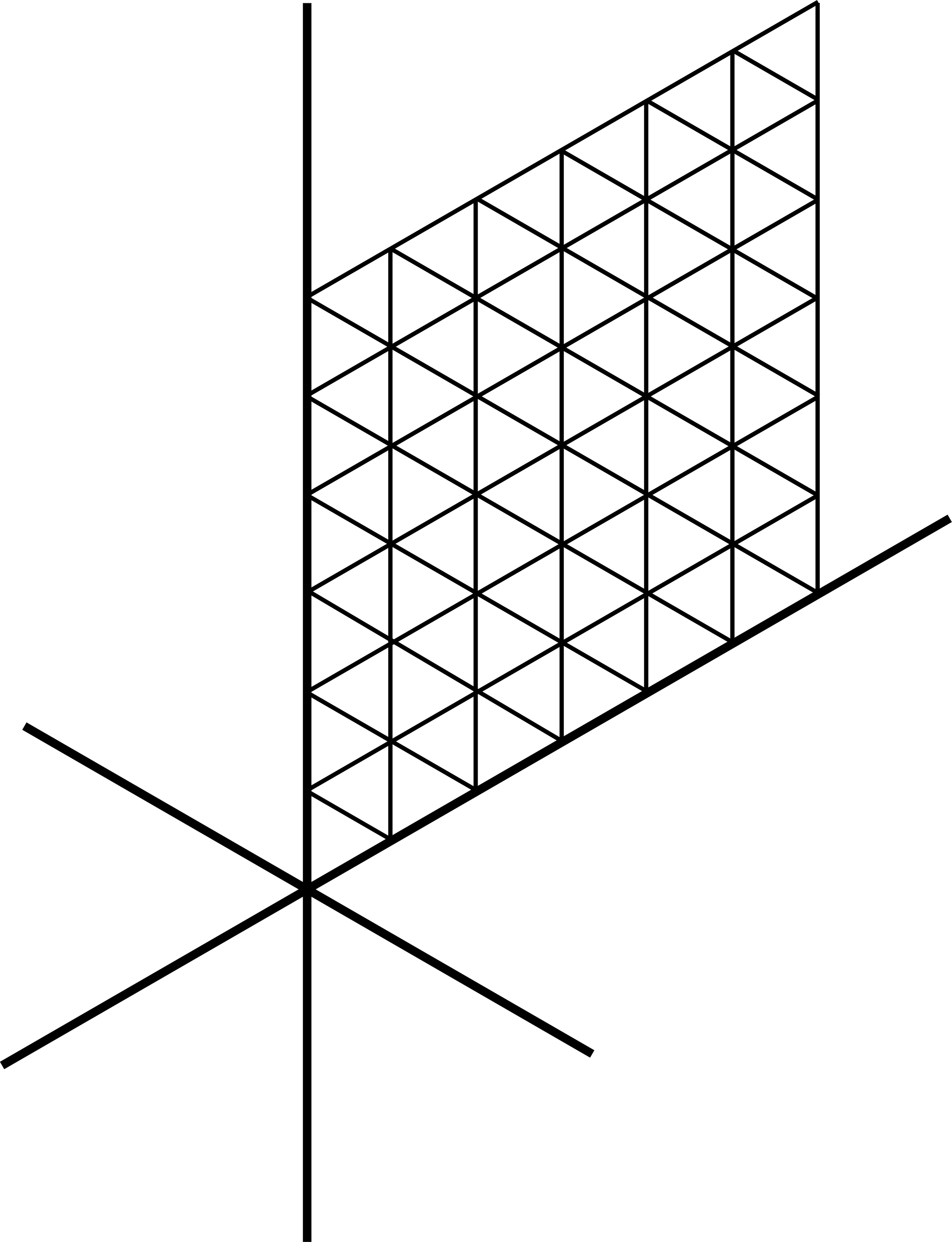}
    \caption{The tiling of $\Lambda^w/2Nr\Lambda^w$ by copies of $\tilde{P}_r$ for $N = 3$.}
    \label{su3tiling}
  \end{figure}
  The root lattice of $\SU(N)$ has volume $\Vol(\Lambda^R) = \sqrt{N}$. This can be seen using a particular identification of the root system with $\bbR^{N-1}$, writing the $i$'th basis vector of the lattice as
  \begin{align*}
		-\sqrt{\frac{j-1}{j}}e_{j-1}+\sqrt{\frac{j+1}{j}}e_{j},
  \end{align*}
  where $e_j$ is the $j$'th standard basis vector of  $\bbR^{N-1}$ (and $e_0 = 0$). Application of Theorem~\ref{quadrecisun} gives\footnote{Note that in \cite{Jef}, the branch cut is along the negative real axis, and that one thus needs to be slightly careful with the signs on the right hand side of Theorem~\ref{quadrecisun} when $N \equiv 3 \,\,\MOD\, 4$.}

  \begin{align*}
		\sum_{\lambda \in \Lambda^w/2rN\Lambda^w} g(\lambda) = \frac{1}{\sqrt{N}} \sqrt{\left(\frac{i}{2bN}\right)^{N-1}} (2rN)^{(N-1)/2} \sum_{\mu \in \Lambda^R/2bN\Lambda^R} \exp\left(-\pi i r \frac{\langle \mu , \mu \rangle}{b}\right).
  \end{align*}
  
  From now on, we specialize to the case $N = 3$ and elaborate on the general case in Remark~\ref{generelsunremark}. In this case $g(\lambda)$ is further invariant under $\lambda \mapsto \lambda + 3r\Lambda_i$, and the above simplifies to
  \begin{align*}
		\sum_{\lambda \in \Lambda^w/3r\Lambda^w} g(\lambda) = \frac{i}{b}\sqrt{3} r \sum_{\mu \in \Lambda^R / 3b\Lambda^R} \exp\left(-\pi i r\frac{\langle \mu,\mu \rangle}{b}\right).
  \end{align*}
  If $\mu = nh_1 + mh_2$, where $h_1,h_2$ are the simple (co)roots spanning $\Lambda^R$, we have
  \begin{align*}
		\langle \mu , \mu \rangle = 2n^2+2m^2-2nm.
  \end{align*}
  The reflection invariance of $g(\lambda)$ implies that we can write $\sum_{\lambda \in \Lambda^w/3r\Lambda^w} g(\lambda)$ in terms of lower dimensional affine subspaces of the weight space. Namely,
  \begin{align*}
		\sum_{\lambda \in \Lambda^w/3r\Lambda^w} g(\lambda) = 18\sum_{\lambda \in \tilde{P}_r} g(\lambda) + 9 \sum_{\lambda \in P^{(1)}_r} g(\lambda) -2\sum_{\lambda \in P^{(0)}_r} g(\lambda),
  \end{align*}
  where
  \begin{align*}
		P^{(1)}_r &= \{ \lambda = a\Lambda_1 + 0\Lambda_2 \mid a = 0, \dots, 3r-1 \}, \\
		P^{(0)}_r &= \{ \lambda = ar\Lambda_1+br\Lambda_2 \mid a, b = 0,1,2\}.
  \end{align*}
  See Figure~\ref{su3invariantfigur}, where the colours indicate the various values of $g(\lambda)$. We find that
  \begin{figure}[t]
    \centering
    \includegraphics[scale=3]{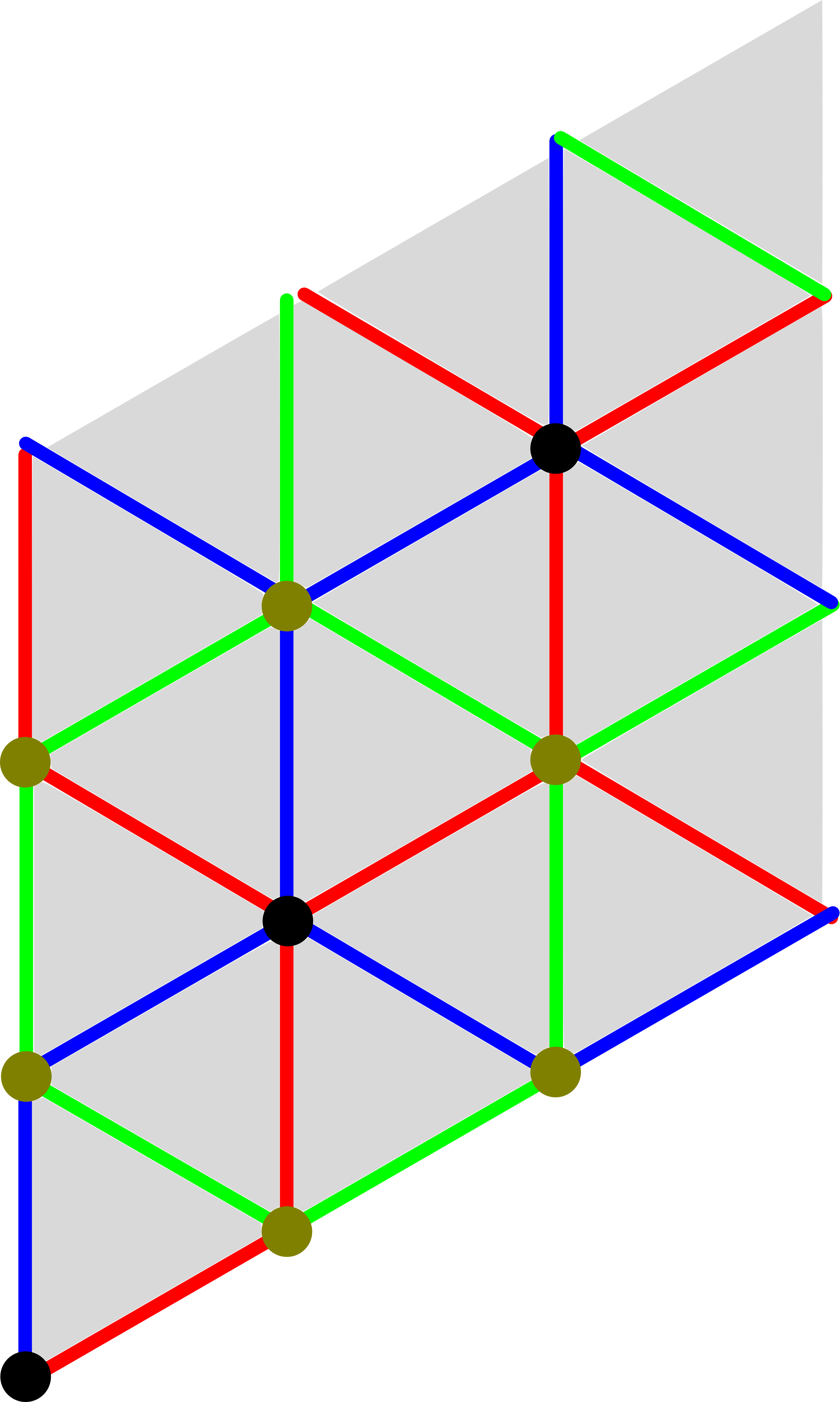}
    \caption{The division of $\Lambda^w/3r\Lambda^w$ into subsets in affine subspaces of lower dimensions.}
    \label{su3invariantfigur}
  \end{figure}
  \begin{align*}
		\sum_{\lambda \in P^{(0)}_r} g(\lambda) = 3 \cdot 1 + 6 \cdot \exp\left(\frac{2\pi i}{3} br\right),
  \end{align*}
  and it follows from Theorem~\ref{quadreci} that
  \begin{align*}
		\sum_{\lambda \in P^{(1)}_r} g(\lambda) = \sum_{n=0}^{3r-1} \exp\left(\frac{2 \pi i}{3r} bn^2\right) = \sqrt{\frac{3r}{2b}} \exp(\pi i/4) \sum_{n=0}^{2\abs{b}-1} \exp\left(-\pi i \frac{3rn^2}{2b}\right).
  \end{align*}
  Putting this together, we obtain the claim of the theorem.
\end{proof}

The case $b = 0$ is handled separately. It could be viewed as a special case of the Verlinde formula for $\SU(N)$ but we include an elementary argument for completeness.
\begin{prop}
  \label{3torusprop}
  The level $k$ quantum $\SU(N)$-invariant of $M^0 = S^1 \times S^1 \times S^1$ is given by
  \begin{align*}
		Z_k^{\SU(N)}(M^0) = \frac{1}{(N-1)!}(r-1)(r-2) \cdots (r-(N-1)).
  \end{align*}
  where as always $r = k + N$.
\end{prop}
\begin{proof}
  In general, $Z_k(\Sigma \times S^1) = \dim V_k(\Sigma \times S^1)$, and it follows from the first part of the proof of Theorem~\ref{invariantssun} that $V_k(S^1 \times S^1)$ has basis given by elements of $\tilde{P}_r$. Elements of $\tilde{P}_r$ correspond to tuples $(a_1,\dots,a_{N-1})$, where $a_i \in \bbZ_{> 0}$ and $\sum_i a_i < r$. Let $f_N(r) = \# \tilde{P}_r$. It follows that
  \begin{align*}
		f_2(r) =  r-1, \quad f_{N+1}(r) = \sum_{l=1}^{r-1} f_N(l).
  \end{align*}
  The claim then follows from induction on both $N$ and $r$.
\end{proof}

For the trace $-2$ homeomorphisms, i.e. the manifolds $\tilde{M}^b$, matters become slightly more involved, as the trace sum is now to be performed only over the Young diagrams $\lambda$ invariant under the involution $\lambda \mapsto \lambda^\star$. This boils down to summing over those $\lambda$ invariant under $\lambda_i \to \lambda_1 - \lambda_{N+1-i}$, $i = 1, \dots, N-1$. Viewing $\lambda = \sum_i \eps_i \Lambda_i$ as a weight, this corresponds to considering those $\lambda$ with $\eps_i = \eps_{N-i}$, $i =  1, \dots, N-1$. As an example of how this complicates the combinatorics, we note the following special case.
\begin{prop}
  \label{skodprop}
  We have
  \begin{align*}
		Z^{\SU(3)}_k(\tilde{M}^1) &=\, \exp(2\pi i/r) \left( \left[\frac{1}{2}\left(\sqrt{r/2}\exp\left(-\frac{\pi i}{4}\right)-1\right)\right] \right. \\
		 &\, \left.+ \left[ \frac{1}{2}\left(\sqrt{r/2}\exp\left(-\frac{\pi i}{4}\right)-\frac{1}{2}\right)\right]\exp(2\pi ir/4)-\frac{1}{4}\exp(-2\pi i r/2) \right).
  \end{align*}
\end{prop}
\begin{proof}
  We proceed as in the proof of Theorem~\ref{invariantssun} and consider
  \begin{align*}
		\sum_{\lambda \in \tilde{P}_r} \exp\left(\frac{\pi i}{r} \langle \lambda , \lambda \rangle\right) \delta_{\lambda , \lambda^\star}.
  \end{align*}
  By the comments before the statement of the proposition, in the case of $\SU(3)$, we find that $\delta_{\lambda , \lambda^\star} = 1$ if and only if $\lambda = n(\Lambda_1 + \Lambda_2) = n \rho$, and $n \rho \in \tilde{P}_r$ if and only if $0 < n \leq \lfloor\tfrac{r-1}{2}\rfloor$. Now $\langle n \rho , n \rho \rangle = 2n^2$ and so
  \begin{align*}
		\sum_{\lambda \in \tilde{P}_r} \exp\left(\frac{\pi i}{r} \langle \lambda , \lambda \rangle\right) \delta_{\lambda , \lambda^\star} = \sum_{n=1}^{\lfloor \frac{r-1}{2} \rfloor} \exp\left(\frac{\pi i}{r} 2n^2\right).
  \end{align*}
  
  \emph{Case 1}. Assume that $r \equiv 0 \,\, \MOD \, 4$ and let $r = 4s$. Then Theorem~\ref{quadreci} immediately implies that
  \begin{align*}
		\sum_{n=1}^{\lfloor \frac{r-1}{2} \rfloor} \exp\left(\frac{\pi i}{r} 2n^2\right) = \sqrt{\frac{r}{2}}\exp\left(\frac{\pi i}{4}\right) - 1.
  \end{align*}
  
  \emph{Case 2}. Assume that $r \equiv 2 \,\, \MOD\,4$ and let $r = 4s+2$. We note that mod $4s+2$, we have $n^2 \equiv (2s+1-n)^2 + (2s+1)$ and thus
  \begin{align*}
		\sum_{n=1}^{\lfloor \frac{r-1}{2} \rfloor} \exp\left(\frac{\pi i}{r} 2n^2\right) = \sum_{n=1}^{2s}\exp\left(\frac{\pi i}{2s+1} n^2\right) = 0,
  \end{align*}
  since the $n$'th and $(2s+1-n)$'th summands cancel.
  
  \emph{Case 3}. Assume that $r \equiv 1 \,\, \MOD\,2$ and let $r = 2s+1$. As above, mod $4s+2$, we have $2n^2 \equiv 2(2s+1-n)^2$, and it follows that
  \begin{align*}
		\sum_{n=1}^{\lfloor \frac{r-1}{2} \rfloor} \exp\left(\frac{\pi i}{r} 2n^2\right) &= \sum_{n=1}^s \exp\left(\frac{\pi i}{2s+1} 2n^2\right) = \frac{1}{2}\sum_{n=1}^{2s} \exp\left(\frac{\pi i}{2s+1}2n^2\right) \\
		&= \frac{1}{2}\left( \sqrt{\frac{2s+1}{2}} \exp(\pi i/4) (1 + \exp(-\pi i(2s+1)/2)) - 1\right).
  \end{align*}
  The result follows.
\end{proof}
For $b = 0$, the situation is easier to handle.
\begin{prop}
  For $N = 2n$ even, we have
  \begin{align*}
		Z_k^{\SU(N)}(\tilde{M}^0) =&\, \frac{1}{n!}\left(\left(\frac{r}{2}-\frac{n}{2}\right)\prod_{l=1}^{n-1}\left(\frac{r}{2}-l\right) + \prod_{l=1}^n\left( \frac{r+1}{2}-l\right)\right) \\
		  &\,+\frac{1}{n!}\left(\left(\frac{r}{2}-\frac{n}{2}\right)\prod_{l=1}^{n-1}\left(\frac{r}{2}-l\right) - \prod_{l=1}^n\left( \frac{r+1}{2}-l\right)\right) \exp(\pi i r),
  \end{align*}
  and for $N = 2n-1$ odd, we have
  \begin{align*}
		Z_k^{\SU(N)}(\tilde{M}^0) =&\, \frac{1}{2(n-1)!}\left(\prod_{l=1}^{n-1}\left(\frac{r}{2}-l\right) + \prod_{l=1}^{n-1}\left(\frac{r+1}{2}-l\right)\right) \\
		 &\,+ \frac{1}{2(n-1)!}\left(\prod_{l=1}^{n-1}\left(\frac{r}{2}-l\right) - \prod_{l=1}^{n-1}\left(\frac{r+1}{2}-l\right)\right)\exp(\pi i r).
  \end{align*}
\end{prop}
\begin{proof}
  As in Proposition~\ref{skodprop}, we simply need to compute the number of elements of $\tilde{P}_r$ invariant under the involution. That is, tuples $(a_1,\dots,a_{N-1})$ with $0 < a_i$, $\sum_i a_i < r$ and $a_i = a_{N-i}$.
  
  \emph{Case 1}. If $N = 2n-1$ is odd, and $r = 2s$ is even, this boils down to counting $b_i > 0$, $i < n$ with $\sum_i 2b_i < r = 2s$. This we already did in Proposition~\ref{3torusprop}, and the result is
  \begin{align*}
		Z_k(\tilde{M}^0) = \frac{1}{(n-1)!} \prod_{l=1}^{n-1} \left(\frac{r}{2} - l\right).
  \end{align*}
  
  \emph{Case 2}. If $N = 2n-1$ is odd, and $r = 2s+1$ is odd, we can proceed as above, noting that now the requirement becomes $\sum_i 2b_i < 2s+1$, or $\sum_i b_i < s+1$, and the result in this case is
  \begin{align*}
		Z_k(\tilde{M}^0) = \frac{1}{(n-1)!} \prod_{l=1}^{n-1} \left(\frac{r+1}{2} - l\right).
  \end{align*}
  This proves the claim for odd $N$.
  
  \emph{Case 3}. Assume now that $N = 2n$ is even. Let $f_n(r) = Z_k^{\SU(N)}(\tilde{M}^0)$ be the number of tuples $(b_1,\dots,b_n)$, $b_i > 0$ with $2\sum_{i=1}^{n-1} b_i + b_n < r$. Then obviously
  \begin{align*}
		f_1(r) = r-1,
  \end{align*}
  and $f_n$ satisfies the recursive relation
  \begin{align}
    \label{rekf}
		f_{n+1}(r) = \begin{cases} \sum_{l=1}^{r/2 - 1} f_n(2l), & \text{$r$ even}, \\ \sum_{l=1}^{(r-1)/2}f_n(2l-1),& \text{$r$ odd.}\end{cases}
  \end{align}
  Introduce $g_n(s) = f_n(2s)$ and $h_n(s) = f_n(2s+1)$. Then \eqref{rekf} implies via induction on $n$ and $s$ that
  \begin{align*}
		g_n(s) = \frac{2}{n!}\left(s-\frac{n}{2}\right) \prod_{l=1}^{n-1}(s-l).
  \end{align*}
  For the $h_n$, one can use a similar argument or simply refer to the proof of Proposition~\ref{3torusprop} and find that
  \begin{align*}
		h_n(s) = \frac{2}{n!}\prod_{l=1}^{n}(s-l).
  \end{align*}
  Now the result follows from
  \begin{align*}
		f_n(r) = \frac{1}{2}((-1)^r+1)g_n\left(\frac{r}{2}\right)+\frac{1}{2}((-1)^{r+1}+1)h_n\left(\frac{r+1}{2}\right).
  \end{align*}
\end{proof}

\subsection{The Chern--Simons values}
We turn now to the question of determining the possible $\SU(N)$-Chern--Simons values for the manifolds under consideration. Whereas it would be nice to have an analogue of Proposition~\ref{su2modulirum} in the general case, we make do with a restriction on possible flat connections, following the proof of Proposition~\ref{su2modulirum}, rather than describing the moduli space explicitly.

\begin{prop}
  \label{suncs}
  Let $\gamma$ be the isotopy class of an essential closed curve in $S^1 \times S^1$, and let $b \in \bbZ$, $b \not= 0$. Let $\calM$ be the moduli space of flat $\SU(N)$-connections on $M^b$. The Chern--Simons action takes the following (not necessarily distinct) values:
  
  Let $a_1, \dots, a_N \in \bbQ$ with $\sum_{l=1}^N a_l \in \bbZ$ and $ba_l \in \bbZ$ for every $l = 1, \dots, N$. Then
  \begin{align*}
		\frac{1}{2}b \left( \sum_{l=1}^N a_l^2 + \left( \sum_{l=1}^N a_l \right)^2 - 2a_N \sum_{l=1}^N a_l\right) \in \CS(\calM),
  \end{align*}
  and is the Chern--Simons value of a completely reducible flat connection; i.e. one whose holonomy is contained in a maximal torus.
  
  The partially reducible and irreducible case: Let $1 \leq i_1 \leq \cdots \leq i_r$, $1 \leq r < N$, be integers with $\sum_l i_l = N$, and let $a_1, \dots, a_r \in \bbQ$ satisfy that $\sum_{l=1}^r i_l a_l \in \bbZ$ and that $bi_l a_l \in \bbZ$ for every $l$. Then
  \begin{align*}
    -\frac{1}{2}b \sum_{l=1}^r i_l a_l(i_r a_r - a_l) + \frac{1}{4}\left( (-1)^{b(1-a_ri_r) \sum i_l a_l} - 1\right) \in \CS(\calM),
  \end{align*}
  and is the Chern--Simons value of a flat connection whose invariant subspaces have dimensions given by the $i_l$. In either case, $\exp(2 \pi i a_l)$ are the eigenvalues of the holonomy about $\gamma$, viewed -- say -- in $S^1 \times S^1 \times \{\tfrac{1}{2}\}$.
  
  Conversely, every value of the Chern--Simons action is of one of the above two forms.
\end{prop}
\begin{proof}
  We follow the approach of the proof of Proposition~\ref{su2modulirum}, and in particular we introduce once again $m = -b$ to reduce the total number of signs. The net result is a sign change of the Chern--Simons values. As in Proposition~\ref{su2modulirum}, we consider $(A,B,C) \in \SU(N)^{\times 3}$ with $AB = BA$, $AC = CA$, and $CBC^{-1} = A^mB$, ignoring in our notation the action of simultaneous conjugation when discussing flat connections. By the first relation, we may assume that both $A$ and $B$ lie in the maximal torus $T \subseteq \SU(N)$ consisting of diagonal $\SU(N)$-matrices, and we write $a_j = A_{jj}$, $b_j = B_{jj}$. The normalizer $N(T)$ of $T$ consists of $N!$ components, naturally identified with the elements of the symmetric group $S_N$, as conjugating an element of $T$ by an element of $N(T)$ permutes the diagonal elements accordingly. We write $N(T)_\sigma$ for the component of $N(T)$ corresponding to $\sigma \in S_N$.
  
  \emph{Case 1}. Assume that $A,B \in Z(\SU(N))$. Then $A^m = \Id$, and so $A = \exp(2\pi i\frac{j}{m}) \Id$ for some $j$, $0 \leq j < b$ with $jN/m \in \bbZ$.
  
  \emph{Case 2a}. Assume that $A \notin Z(\SU(N))$, $C \in T$. Then $a_j = \exp(2\pi i \frac{\hat{a}_j}{m})$ for some $\hat{a}_j$ with $0 \leq \hat{a}_j < m$, not all equal, and satisfying $\tfrac{1}{m}\sum_j \hat{a}_j \in \bbZ$.
  
  \emph{Case 2b}. Assume that $A \notin Z(\SU(N))$, $C \notin T$. Then $C \in N(T) \setminus T$, say $C \in N(T)_\sigma$ for some $\sigma \not= \id$. Recall that conjugacy classes of $S_N$ correspond to increasing sequences $1 \leq i_1 \leq \cdots \leq i_r$, $\sum_l i_l = N$, and assume for simplicity that
  \begin{align}
    \label{konjsigma}
		\sigma = (1 \, 2 \, \dots \, n_1)(n_1+1 \, n_1 + 2\, \dots \, n_2)\cdots(n_{r-1} + 1 \, \dots \, n_r),
  \end{align}
  with $n_1 = i_1$, $n_l - n_{l-1} = i_l$. The conjugate cases are handled similarly. Then $AC = CA$ implies that, in block form, we have $A = \diag(\tilde{a}_1 \Id_{i_1}, \dots, \tilde{a}_{r}\Id_{i_r})$ with $\tilde{a}_j \in \U(1)$, $\prod_{j=1}^r \tilde{a}_j^{i_j} = 1$. For $l = 1, \dots, N$, let $k_l$ be the smallest natural number such that $\sigma^{k_l}(l) = l$, that is, $k_l$ is one of the $i_m$ above. Then
  \begin{align*}
		b_l = b_{\sigma^{k_l}(l)} = (C^{k_l}BC^{-k_l})_{ll} = (A^{k_lb}C^{k_l}BC^{-k_l})_{ll} = a_l^{k_lm} b_l.
  \end{align*}
  This implies that $\tilde{a}^{mi_j}_j = 1$ for all $j = 1, \dots, r$. Furthermore, the $b_l$ must satisfy
  \begin{align*}
		b_{\sigma(l)} = A_{ll}^m b_l.
  \end{align*}
  
  \emph{Case 3}.  Let $A \in Z(\SU(N))$, $B \notin Z(\SU(N))$ so that $A = a\Id$ for some $a \in \U(1)$, $a^N = 1$. As before, $C \in N(T)_\sigma$ for some $\sigma \in S_N$, and for every $l$, we have $b_{\sigma(l)} = a^m b_l$. Arguing exactly as above, we find that $a^{mk_l} = 1$ for every $l$, so $a^{m\gcd(k_l)}  = 1$, once again limiting the possible values of $a$, and thus $b_l$, accordingly.
  
  To our knowledge, explicit expressions for $\SU(N)$-Chern--Simons values of general flat connections exist only in a few very special cases. Perhaps most relevant to our study is \cite[Thm. 3.1]{Nis}, in which Nishi calculates the $\SU(N)$-Chern--Simons values of irreducible flat connections on general Seifert manifolds. However, many of the connections considered above are \emph{not} irreducible. In another direction, \cite[Thm. 5.11]{Jef} determines the $\SU(N)$-Chern--Simons values for a number of torus bundles; whereas this result can not be applied directly to our manifolds (as the map $wU-1$, in the notation of Jeffrey, is not invertible), the proof of the theorem can.
  
  Following Jeffrey and the notation of Appendix~\ref{normaliseringer}, we let $\bbA = \frakt \oplus \frakt$ and consider the lattice $\Lambda = \Lambda^R \oplus \Lambda^R \subseteq \bbA$. Thus it is natural to identify $T \times T = \bbA / \Lambda$, and consider elements of $\frakt$ as traceless $N \times N$-matrices with entries being the coefficients of the simple (co)roots spanning $\Lambda^R$ in such a way that $\exp(2 \pi i \cdot) : \frakt \to T$ maps $\Lambda^R$ to $\Id \in T$. Define (see \cite[(A.7)]{Jef}) the basic symplectic form on $\bbA$ by
  \begin{align*}
		\omega((\xi_1,\eta_1),(\xi_2,\eta_2)) = 2\pi(\langle \xi_1,\eta_2\rangle - \langle \xi_2 , \eta_1 \rangle),
  \end{align*}
  and note that the basic inner product is defined such that $\langle \lambda_1,\lambda_2\rangle = - \tr(\lambda_1\lambda_2)$. Finally, we will need the theta-characteristic $\epsilon : \Lambda \to \{ \pm 1\}$ satisfying
  \begin{align*}
		\epsilon(\lambda_1 + \lambda_2) = \epsilon(\lambda_1)\eps(\lambda_2)(-1)^{\omega(\lambda_1,\lambda_2)/(2\pi)},
  \end{align*}
  which we require to furthermore satisfy that $\eps(h_\alpha,h_\beta) = 1$ for any pair of simple roots, all of them together spanning the lattice $\Lambda$. It follows that
  \begin{align*}
		\eps(0,0) = \eps(h_\alpha,0) = \eps(0,h_\alpha) = \eps(\pm h_\alpha, \pm h_\beta) = 1
  \end{align*}
  for simple coroots $h_\alpha,h_\beta$. Moreover, $\eps(\lambda,0) = \eps(0,\lambda) = 1$ for any $\lambda \in \Lambda^R$, as can easily be seen by induction. In general, for $(\lambda, \mu) \in \Lambda$, we have
  \begin{align}
    \label{thetachar}
		\eps(\lambda,\mu) = \eps((\lambda,0)+(0,\mu)) = \eps(\lambda,0)\eps(\mu,0)(-1)^{\omega((\lambda,0),(0,\mu))/(2\pi)} = (-1)^{\langle \lambda, \mu \rangle} = (-1)^{\tr(\lambda \mu)}.
  \end{align}
  Returning to our setup, note that any element $U = \begin{pmatrix} a & b \\ c & d \end{pmatrix} \in \SL(2,\bbZ)$ acts on $\frakt \oplus \frakt$ by
  \begin{align*}
		U(\lambda,\mu) = (a\lambda + b \mu, c \lambda + d \mu).
  \end{align*}
  Let $(A,B,C) \in T \times T \times \SU(N)$ be a flat connection on $M^b$ as described in the various cases above, and abusing notation let $(A,B) \in \frakt \oplus \frakt$ be such that $\exp(2\pi i A) = A$, $\exp(2 \pi i B) = B \in T$. Then in any case, there will exist $w \in W$ such that
  \begin{align}
    \label{weylelement}
		w U(A,B) - (A,B) =: (\lambda,\mu) \in \Lambda^r.
  \end{align}
  Now, the discussion preceding \cite[Thm.~5.11]{Jef} shows that
  \begin{align*}
		\exp(2\pi i \CS(A,B,C)) = \eps(\lambda,\mu)\exp\left(\frac{i}{2} \omega((A,B),(\lambda,\mu))\right).
  \end{align*}
  Combining this with \eqref{thetachar}, we finally obtain
  \begin{align}
    \label{csexplicit}
		\CS(A,B,C) = \frac{1}{2}(\langle A , \mu \rangle - \langle B , \lambda \rangle) - \frac{1}{4}((-1)^{\tr(\lambda \mu)}-1).
  \end{align}
  Notice that this does not depend on $C$, nor -- $\mathrm{mod} \, \bbZ$ -- on the choice of $(A,B) \in \frakt \oplus \frakt$. We now return to our specific cases.
  
  \emph{Case 1}. This really follows as a special case of considerations below but is included here to show the general idea. It suffices to find the value of $\CS$ on one element of the component under consideration, so take $A = (j/m, \dots, j/m,j/b-Nj/m)$, where $j$ satisfies the condition of the conclusion of our considerations, and let $B$ be arbitrary, fitting with the assumptions of this case. Then in the notation of \eqref{csexplicit}, we have
  \begin{gather*}
		w = \id, \quad (\lambda,\mu) = ((0,\dots,0), (j,\dots,j,j-Nj)),
  \end{gather*}
  and recalling that $\langle\cdot,\cdot\rangle = -\tr(\cdot\cdot)$, it follows from \eqref{csexplicit} that
  \begin{align*}
		\CS(A,B,C) = -\frac{1}{2}\left(\frac{Nj^2}{m}-\frac{Nj^2}{m}-\frac{Nj^2}{m}+\frac{N^2j^2}{m}\right) = -\frac{1}{2}N(N-1)\frac{j^2}{m}.
  \end{align*}
  
  \emph{Case 2}. Again, we may as well assume that $\sigma$ is of the form \eqref{konjsigma}: If $(A,B,C)$ is a representation with $C \in N(T)_\sigma$, and $\eta \in S_N$, we find that for any $D \in N(T)_\eta$, $(DAD^{-1},DBD^{-1},DCD^{-1})$ is another representation with $DCD^{-1} \in N(T)_{\eta \sigma \eta^{-1}}$. Assume first that $\sigma \not= \id$. That is,  we can assume that $\sigma(N) = N-k_N+1\not= N$, $\sigma^{-1}(N) =N-1$. Let $a_1,\dots,a_r,b_1,\dots,b_N \geq 0$ be such that if 
  \begin{align*}
		A &= \diag(a_1, \dots,a_1,\dots, a_{N}, \dots, a_r, a_r-m_A), \\
		B &= \diag(b_1, b_2,\dots, b_N, b_N-m_B),
  \end{align*}
  where the $a_l$ is repeated $i_l$ times, and $m_A = \sum_l a_l$, $m_B = \sum_l b_l$, then $\exp(2 \pi i A)$ and $\exp(2\pi i B)$ satisfy the conclusion of Case 2 or 3 above. This means that $m\gcd(i_l)a_l \in \bbZ$, and that we may write the last $i_r = k_N = N - \sigma(N)-1$ entries of $B$ as
  \begin{align*}
		c, \,\, c+a_rm, \,\, \dots,\,\, c+a_rm(i_r-2), c+a_rm(i_r-1) -m_B
  \end{align*}
   for some $c \geq 0$. We find that $w = \sigma^{-1}$ satisfies the condition of \eqref{weylelement} with $(\lambda,\mu)$ given as follows: For any diagonal matrix $(m_{ll})_{l}$ and $\sigma \in S_N$, let $\sigma M$ be the diagonal matrix whose $l$'th diagonal entry is the $m_{\sigma(l)}$, and let $e_l$ be the matrix whose $(l,l)$'th entry is $1$, all others being $0$. Then
  \begin{align*}
		(\lambda,\mu) = (\sigma A - A, \sigma(mA + A) - B) \in \frakt \oplus \frakt.
  \end{align*}
  Now, since $\tr((\sigma A)(\sigma B)) = \tr(AB)$ and $\tr(A (\sigma B)) = \tr((\sigma^{-1}A)B)$, it follows that
  \begin{align*}
	  \langle A , \mu \rangle = -\tr(A(\sigma(mA+B)-B)) = -\tr(A(\sigma mA)) - \tr(A(\sigma B)) + \tr(AB),
  \end{align*}
  and similarly that
  \begin{align*}
	  \langle B , \lambda \rangle = -\tr(B(\sigma A)) + \tr(BA).
  \end{align*}
  Now, note that $b_{\sigma^{-1}(N)} = ba_r(i_r -2)+b_{\sigma(N)}$. Thus we find that
  \begin{align*}
	  \langle A , \mu \rangle - \langle B , \lambda \rangle &= -m\tr(A (\sigma A)) + \tr((\sigma A - \sigma^{-1}A)B) \\
	    &= -m( \sum_l k_l a_l^2 - 2a_l m_A) - m_Ab_\sigma(N) + m_A b_{\sigma^{-1}(N)} \\
	    &= -m \sum_l i_l a_l^2 + m_A ba_n k_N = m(i_r a_r(\sum_l  i_l a_l) - \sum_l i_l a_l^2) \\
	    &= m \sum_l i_l a_l(i_r a_r - a_l).
  \end{align*}
  For the theta-characteristic, note first that
  \begin{align*}
	  \tr(A^2 - A (\sigma A)) &= m_A^2, \\
	  \tr(AB - (\sigma A)B) &= m_A b_{\sigma(N)} - m_A b_N, \\
	  \tr(AB - A(\sigma B)) &= m_A b_{\sigma^{-1}(N)} - m_A b_N.
  \end{align*}
  We find from this that
  \begin{align*}
    (-1)^{\Tr(\lambda \mu)} &= (-1)^{m\tr(A^2 - A (\sigma A)) + \tr(AB - A(\sigma B)) + \tr(AB - (\sigma A)B)} \\
      &=  (-1)^{mm_A - 2m_Ab_N + m_A(b_{\sigma(N)}+b_{\sigma^{-1}(N)})} \\
      &= (-1)^{mm_A - 2m_A(m_{\sigma^{-1}(N)} + ma_N - m_B) + m_A(b_{\sigma(N)} + b_{\sigma(N)})} \\
      &= (-1)^{mm_A - m_A(ma_r(i_r-2)+b_{\sigma(N)})-2ma_rm_A+m_Ab_{\sigma(N)}} \\
      &= (-1)^{mm_A - m_A ma_r i_r} = (-1)^{m(1-a_rk_r) \sum i_l a_l}.
  \end{align*}
  Putting this together, we finally find that
  \begin{align*}
    \CS(A,B,C) = \frac{1}{2}m \sum_l i_l a_l(i_r a_r - a_l) - \frac{1}{4}\left( (-1)^{m(1-a_ri_r) \sum i_l a_l} - 1\right).
  \end{align*}
  In the case $\sigma = \Id$, $\lambda$ vanishes as in Case 1, and $\mu = bA$, and so
  \begin{align*}
    \CS(A,B,C) &= \frac{1}{2} \langle A , mA \rangle = -\frac{1}{2} \left( \sum_l a_l^2 m + mm_a^2 - m_Ama_N-a_Nmm_A\right) \\
     &= -\frac{1}{2} m\left( \sum_l a_l^2 + \left(\sum_l a_l\right)^2 -2a_N\sum_l a_l\right).
  \end{align*}
  
  \emph{Case 3} follows as the special case of the above calculation where all $a_l$ are equal.
\end{proof}
\begin{example}
  For $N = 2$, we recover the considerations of Proposition~\ref{su2modulirum}. As a non-trivial but perhaps more illuminating example, for $N = 3$, $b = 1$, the moduli space becomes a union of sets of the form
  \begin{gather*}
  \{\Id\} \times Z(\SU(3)) \times \SU(3), \\
		\{ \diag(-1,-1,1) \} \times \{ \diag(b,-b,-b^{-2}) \mid b \in \U(1)\} \times N(T)_{(1 \, 2)}, \\
			\{ \diag(-1,1,-1) \} \times \{ \diag(b,-b^{-2},-b) \mid b \in \U(1)\} \times N(T)_{(1 \, 3)}, \\
	\{ \diag(1,-1,-1) \} \times \{ \diag(-b^{-2},b,-b) \mid b \in \U(1)\} \times N(T)_{(2 \, 3)}, \\
	\{e(\tfrac{1}{3})\Id\} \times \{\diag(e(1/3),e(2/3),1),\diag(e(2/3),1,e(1/3)),\diag(1,e(1/3),e(2/3)) \} \times N(T)_{(1 \, 2 \, 3)}, \\
	\{e(\tfrac{2}{3})\Id\} \times \{\diag(e(2/3),e(1/3),1),\diag(e(1/3),1,e(2/3)),\diag(1,e(2/3),e(1/3)) \} \times N(T)_{(1 \, 2 \, 3)}, \\
	\{e(\tfrac{1}{3})\Id\} \times \{\diag(e(1/3),1,e(2/3)),\diag(e(2/3),e(1/3),1),\diag(1,e(2/3),e(1/3)) \} \times N(T)_{(1 \, 3 \, 2)}, \\
	\{e(\tfrac{2}{3})\Id\} \times \{\diag(e(2/3),1,e(1/3)),\diag(e(1/3),e(2/3),1),\diag(1,e(1/3),e(2/3)) \} \times N(T)_{(1 \, 3 \, 2)}, \\
	\{\Id\} \times \{ \diag(-1,-1,1) \} \times N(T)_{(1 \, 2)}, \\
	\{\Id\} \times \{ \diag(-1,1,-1) \} \times N(T)_{(1 \, 3)}, \\
	\{\Id\} \times \{ \diag(1,-1,-1) \} \times N(T)_{(2 \, 3)},
  \end{gather*}
  up to conjugation. Here, we have used the notation $e(x) = \exp(2\pi i x)$. The resulting Chern--Simons values in this case are $0,\tfrac{3}{4},\tfrac{3}{4},\tfrac{3}{4},\tfrac{2}{3},\tfrac{2}{3},\tfrac{2}{3},\tfrac{2}{3},0,0,0$ respectively, as can be verified by the method explained in the proof.

\end{example}

\begin{cor}
  \label{sunasympconjec}
  Conjecture~\ref{asympconjec} holds for the 3-manifolds $M^b$ when $G = \SU(3)$.
\end{cor}
\begin{proof}
  We will show that for $N = 3$, the values arising from Proposition~\ref{suncs} are exactly the phases of the expression in Theorem~\ref{invariantssun}.
  
  \emph{The irreducible case}: In the notation of the second half of Proposition~\ref{suncs}, let $r=1, i_1 = 3, a_1 = 1/3$. Then one finds the Chern--Simons value $-\tfrac{1}{3}b$, corresponding to the phase of the term whose growth rate of $r$ is $0$ in Theorem~\ref{invariantssun}.
  
  \emph{The partially reducible case}: Considering again the second half of the proposition, let $r=2, i_1 = 1, i_2 = 2$, let $a_2 = n/(2b)$ for $n \in \{0,\dots,2b-1\}$, and let $a_1 = (b-n)/b$. Then through a short computation, the proposition gives the Chern--Simons value $\tfrac{3n^2}{4b}$, corresponding in Theorem~\ref{invariantssun} to the terms of growth rate $\tfrac{1}{2}$.
  
  \emph{The completely reducible case}: Finally, consider the first half of the proposition and let $a_1 = n/b$, $a_2 = m/b$, and $a_3 = -(n+m)/b$, where $n,m \in \{0,\dots,3b-1\}$. Then we find the Chern--Simons values
  \begin{align*}
		\frac{1}{b}(n^2+m^2+nm),
  \end{align*}
  and the Corollary follows from the observation that
  \begin{align*}
	  &\left\{ \exp\left(2\pi i r \frac{n^2+m^2+mn}{b}\right) \relmiddle| n,m = 0, \dots, 3b-1 \right\} \\
	    &= \left\{ \exp\left(2\pi i r \frac{n^2+m^2-mn}{b}\right) \relmiddle| n,m = 0, \dots, 3b-1 \right\}.
  \end{align*}
  
\end{proof}
\begin{rem}
  \label{generelsunremark}
  From this argument, it is clear what should be expected to happen in the general case $G = \SU(N)$, $N > 3$. Namely, the sum $\sum_{\lambda \in \tilde{P}_r} g(\lambda)$ of the proof of Theorem~\ref{invariantssun} is expressed in terms of sums of elements of various affine subspaces of the weight space, each of these giving a contribution of a particular order in $r$, and whose corresponding phases are the Chern--Simons values of connections whose invariant subspaces have dimensions depending on the given subspace of the weight space under consideration: In the case $N = 3$ irreducible connections, partially reducible, and completely reducible connections correspond to $0$-dimensional affine subspaces, $1$-dimensional ones, and $2$-dimensional subspaces (the alcoves themselves) as shown in Figure~\ref{su3invariantfigur}. 
\end{rem}

\section{Further remarks}
\subsection{Stretch factors of Anosov homeomorphisms}
\label{pseudoafsnit}
Having discussed in some detail the asymptotic behaviour of the $\SU(2)$-quantum invariants of Dehn twist mapping tori, we now make a few comments on what will happen for the mapping tori of Anosov homeomorphisms, whose invariants are described in Theorem~\ref{jeffreythm}. The growth rates $d$ of Conjecture~\ref{asympconjec} are $0$ for these $3$-manifolds, and the leading order coefficients have the following interpretation, which is really nothing but a complicated procedure for calculating eigenvalues of $2 \times 2$-matrices.
\begin{prop}
  \label{stretchfactorprop}
  Let $\phi : \Sigma_1 \to \Sigma_1$ be an Anosov mapping class of the closed torus, given by the $\SL(2,\bbZ)$ matrix
  \begin{align*}
    \phi = \begin{pmatrix} a & b \\ c & d \end{pmatrix}.
  \end{align*}
  Then the stretch factor $\lambda$ of $\phi$ is given by
  \begin{align*}
    \lambda = \abs{ Z_{n(a^2+2ad+d^2-4)-2}(T_\phi)}^{-2}
  \end{align*}
  for all $n \in \bbN$.
\end{prop}
\begin{proof}
  In the torus case, the stretch factor $\lambda$ is nothing but the spectral radius of $\phi$. In other words,
  \begin{align*}
    \lambda = \max_{\pm} \left\lvert\frac{(a+d)\pm\sqrt{(a+d)^2 - 4}}{2}\right\rvert.
  \end{align*}
  By Theorem~\ref{jeffreythm},
  \begin{align*}
    \abs{Z_{n(a^2+2ad+d^2-4)-2}&(T_\phi)} = \abs{Z_{n(a+d+2)(a+d-2)-2}(T_\phi)} \\
     =&\, \left\lvert e^{2\pi i\psi(U)/(4n(a+d+2)(a+d-2))} \sum_{\pm} \pm \frac{1}{2\lvert c \rvert	\sqrt{\abs{d+a \mp 2}}} \sum_{\beta = 0}^{\lvert c \rvert - 1}\sum_{\gamma=1}^{\abs{d+a \mp 2}} 1 \right\rvert \\
     =&\, \left\lvert\sum_{\pm} \pm \frac{1}{2\sqrt{\abs{d+a\mp 2}}} \abs{d + a \mp 2}\right\rvert \\
     =&\, \left\lvert\frac{\sqrt{\abs{d+a- 2}} - \sqrt{\abs{d + a +2}}}{2}\right\rvert.
  \end{align*}
  It follows that
  \begin{align*}
    \abs{Z_{n(a^2+2ad-4)-2}(T_\phi)}^2 = \min_{\pm} \left\lvert\frac{1}{2} \left(a+d\pm\sqrt{(a+d)^2-4}\right)\right\rvert = \lambda^{-1}.
  \end{align*}
\end{proof}

This result should also be seen in view of Question~1.1~(2) of \cite{AMU} which asks if one may determine the stretch factors of general pseudo-Anosov homeomorphisms by using the quantum representations. In \cite[Cor.~5.8]{AMU}, the authors show that this is the case for the four punctured sphere.

In hopes of generalizing this result to other surfaces, let us briefly discuss a different reason why Proposition~\ref{stretchfactorprop} is true. In her discussion of the semi-classical approximation conjecture for torus bundles, Jeffrey notes the following: for mapping tori of Anosov torus homeomorphisms, the mapping torus moduli space may be understood in terms of fixed points of the action of the Anosov on the moduli space, the path integral localizes to the fixed point set, and the Reidemeister torsion of a connection is given in terms of the differential of the moduli space action (see \cite[Prop.~3.10]{Jefphd}, \cite[Prop.~5.6]{Jef}). Now, in the torus case for $G = \SU(2)$, as we saw in the previous sections, the action on the moduli space $(T \times T) / W$ is essentially the action on the torus itself, and it is therefore not surprising that we can recover the stretch factor: the sequence chosen in Proposition~\ref{stretchfactorprop} is set up to reduce the quantum invariant to a sum of the Reidemeister torsion contributions.

On the other hand, the fixed point set itself carries information about the dynamics of the Anosov in the sense of the following proposition.
\begin{prop}
  For $G = \SU(2)$ and $\phi \in \SL(2,\bbZ)$ an Anosov mapping class on the torus with stretch factor $\lambda > 1$, let $c_m$ denote the number of fixed points of the action of $\phi^m$ on the moduli space of the torus. Then $c_{m+1}/c_m \to \lambda$ as $m \to \infty$.
\end{prop}
\begin{proof}
	By the discussion preceding \cite[Prop.~5.12]{Jef}, the number of fixed points of the action of $U = \begin{pmatrix} a & b \\ c & d \end{pmatrix} \in \SL(2,\bbZ)$, $\tr U \not= \pm 2$, is
	\begin{align*}
		\abs{2+a+d} + \abs{2-a-d} - n,
  \end{align*}
  where $n \leq 4$ is the number of points in $\bbR^2/\bbZ^2$ fixed by both $U$ and $-U$. Now of course, we have $a + d = \lambda + \lambda^{-1}$, where $\lambda$ is an eigenvalue of $U$, and the result follows since if $\lambda(\phi^m)$ denotes the largest eigenvalue of $\phi^m$, then
  \begin{align*}
		\lim_{m \to \infty} \frac{c_{m+1}}{c_m} = \lim_{m \to \infty} \frac{\lambda(\phi^{m+1})}{\lambda(\phi^m)} = \lambda.
  \end{align*}
\end{proof}
The quantum counterpart of this proposition is the following.
\begin{prop}
  \label{sfconjectorus}
  Let $G = \SU(2)$, and let $\phi \in \SL(2,\bbZ)$ be an Anosov mapping class on the torus with stretch factor $\lambda > 1$. Let $k_n = a_n^2+2a_nd_n+d_n^2-6$, where $a_n$ and $d_n$ are defined by $\phi^n = \begin{pmatrix} a_n & b_n \\ c_n & d_n \end{pmatrix}$. Then
  \begin{align*}
		\lim_{n \to \infty} \sqrt[n]{\abs{Z_{k_n}(T_{\phi^n})}} = \lambda^{-1/2}.
  \end{align*}
\end{prop}
\begin{proof}
  As in the proof of Propostion~\ref{stretchfactorprop}, we find that
  \begin{align*}
		\abs{Z_{k_n}(T_{\phi^n})} = \left\lvert \frac{\sqrt{\abs{d_n+a_n- 2}} - \sqrt{\abs{d_n + a_n +2}}}{2} \right\rvert.
  \end{align*}
  The result follows immediately, as $\lim_{n\to\infty} a_n/a_{n-1} = \lim_{n\to\infty} d_n/d_{n-1} = \lambda$.
\end{proof}
This proposition illustrates how one might hope to illuminate the general AMU conjecture by considering the quantum representations of iterates of pseudo-Anosovs. In general, we propose the following.
\begin{conjec}[Stretch Factor Conjecture]
  \label{stretchconjec}
  \index{Stretch Factor Conjecture}Let $G = \SU(N)$, let $\Sigma_g$ be a closed genus $g$ surface, and let $\phi \in \Gamma_g$ with stretch factor $\lambda$. Then there exists a rational number $c \in \bbQ$, and a sequence $\{k_n\}_n \subseteq \bbN$ such that
  \begin{align*}
		\lim_{n \to \infty} \sqrt[n]{\abs{Z^G_{k_n}(T_{\phi^n})}} = \lambda^{c}.
  \end{align*}
\end{conjec}
In Proposition~\ref{sfconjectorus}, the $k_n$ were chosen such that $(k_n+2)\CS(A) \in \bbZ$ for all flat connections $A$ on $T_{\phi^n}$, and if one knew that all values of the Chern--Simons functional on any given mapping torus were rational, this condition on the $k_n$ would be a natural addition to the above conjecture. More generally, one could assume e.g. that the $k_n$ were chosen to satisfy
\begin{align*}
	\abs{\exp(2\pi i (k_n+N)\CS(A)) - 1} < 1/n
\end{align*}
for all $A$.

\appendix
\section{Correspondence results}
\label{normaliseringer}
As described in Section~\ref{quantuminvariants}, understanding quantum invariants of the Dehn twist torus bundles boils down to understanding the quantum representations of $\MCG(S^1 \times S^1)$. In this appendix, we collect most of the notation used in the previous sections and show that the projective representations of $\MCG(S^1 \times S^1)$ from conformal field theory agree with those arising from Blanchet's modular category (for the $T$-matrices in the case of $G = \SU(N)$, but for the $S$-matrices only for $G = \SU(2)$). This equivalence follows abstractly from the isomorphism \cite{AU4} mentioned in the introduction, but we show here how to obtain it through direct calculation. Recall that $\MCG(S^1 \times S^1) = \SL(2,\bbZ)$ is generated by the matrices
\begin{align*}
	\begin{pmatrix} 0 & -1 \\ 1 & 0 \end{pmatrix}, \quad \begin{pmatrix} 1 & 1 \\ 0 & 1 \end{pmatrix}.
\end{align*}
We refer to the quantum representations of these elements as the \emph{$S$-matrix} and the \emph{$T$-matrix} respectively.

\subsection{The $T$-matrix}
Let as always $G =\SU(N)$ for some $N$. Let $\frakt$ be the Lie algebra of a maximal torus in $G$. Denote by $\langle \,,\, \rangle$ the basic inner product on $\frakt^*$ such that the highest (and therefore all) root, denoted $\alpha_m$, has length $\sqrt{2}$. Denote by $\Lambda^R$ the root lattice in $\frakt^*$ and by $\Lambda^w$ the lattice dual to $\Lambda^R$ under the basic inner product. We refer to $\Lambda^w$ as the weight lattice, silently identifying elements of $\frakt^*$ and $\frakt$ using the basic inner product. The weight lattice has a basis consisting of fundamental weights $\Lambda_i$, $i = 1, \dots, N-1$. Now, let $k \in \bbZ_{\geq 0}$ be a level, and let
\begin{align*}
	c = \frac{(\dim G)k}{k+N}
\end{align*}
be the \emph{level $k$ central charge}.

The relevant labelling set $P_k$ in conformal field theory is the set of highest weight integrable representations of the loop group $LG$, which in the above notation is given by
\begin{align*}
	P_k = \{ \lambda \in \Lambda^w \cap P_+ \mid \langle \lambda , \alpha_m \rangle \leq k \},
\end{align*}
where $P_+$ denotes the positive Weyl chamber. Now $\MCG(S^1 \times S^1)$ acts on the vector space spanned by these labels. The $T$-matrix at level $k$ is given by (see e.g. \cite{GW} or \cite{Kac}) the diagonal matrix
\begin{align}
  \label{tcft}
	T^\mathrm{CFT}_{\lambda,\lambda'} = \delta_{\lambda,\lambda'} \exp\left(\frac{\pi i}{r} \langle \lambda + \rho, \lambda + \rho \rangle - \frac{i\pi}{N}\langle \rho , \rho \rangle \right),
\end{align}
for $\lambda , \mu \in P_k$, where as always $r = k + N$. 
We will argue that this matrix differs from the matrix $T$ considered in Section~\ref{quantuminvariants} by the scalar factor $\exp(2\pi i c/24)$.

Now let $\lambda = (\lambda_1 \geq \cdots \geq \lambda_{N-1} \geq 0) \in \Gamma_{N,k}$ be a Young diagram with at most $k$ columns. Corresponding to this is the weight $\lambda = \sum_i (\lambda_{i} - \lambda_{i+1})\Lambda_i \in P_k$, letting $\lambda_N = 0$ (see Figure~\ref{youngweights} for the case $N = 3$, $k = 3$), giving a bijection $\Gamma_{N,k} \to P_k$. We prove the following seemingly well-known lemma, expressing the quadratic Casimir of a weight in terms of the length and contents of a Young diagram.

\begin{figure}[h]
  \centering
  \includegraphics{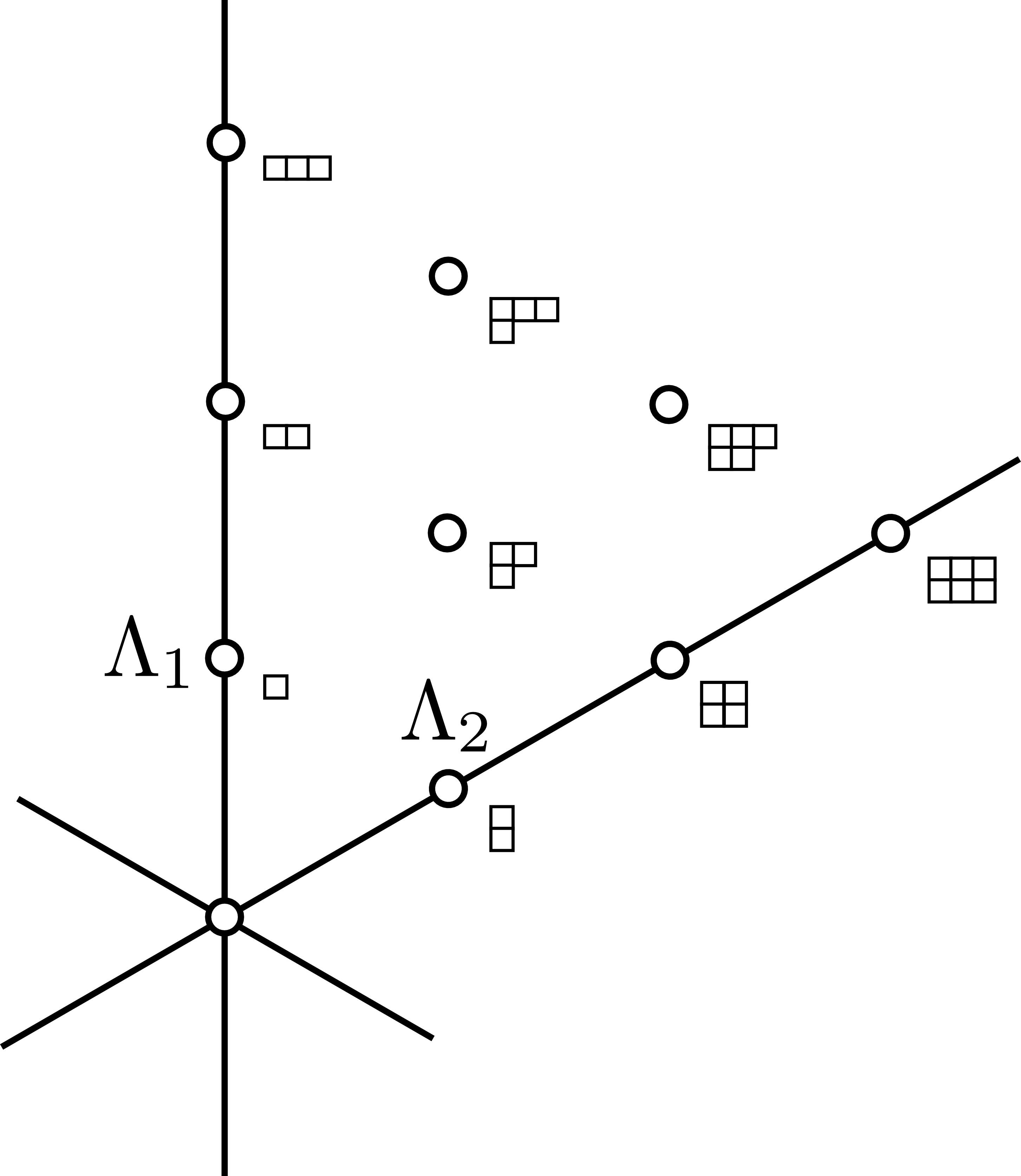}
  \caption{The correspondence between $P_3$ for $\mathfrak{sl}_3$ and $\Gamma_{3,3}$.}
  \label{youngweights}
\end{figure}
\begin{lem}
  \label{tlemma}
  Under the above correspondence between Young diagrams and weights,
  \begin{align*}
		\langle \lambda + \rho , \lambda + \rho \rangle - \frac{\dim \SU(N) \cdot N}{12} = -\frac{1}{N}\left( \lvert\lambda\rvert^2 - N^2 \lvert\lambda\rvert - 2N\sum_{(i,j) \in \lambda} \mathrm{cn}(i,j) \right).
  \end{align*}
\end{lem}

Here $\abs{\lambda} = \sum_{i=1}^{N-1} \lambda_i$ denotes the number of cells in $\lambda$ and should not be confused with $\sqrt{\langle\lambda,\lambda\rangle}$.

\begin{proof}
  First of all, note that by Freudenthal's strange formula
  \begin{align}
    \label{freud1}
		\langle \lambda + \rho , \lambda + \rho \rangle - \frac{\dim \SU(N) \cdot N}{12} = \langle \lambda, \lambda \rangle + 2 \langle \lambda, \rho \rangle.
  \end{align}
  It is well-known (see e.g. \cite[Lemma~13.3.A]{Hum}) that $\rho = \sum_{j=1}^{N-1} \Lambda_j$, and so the right hand side of \eqref{freud1} becomes
  \begin{align*}
		\langle \lambda, \lambda \rangle + 2 \langle \lambda, \rho \rangle = \sum_{i,j=1}^{N-1} (\lambda_{i} - \lambda_{i+1})(\lambda_{j} - \lambda_{j+1}) \langle \Lambda_i , \Lambda_j \rangle + 2\sum_{i,j=1}^{N-1} (\lambda_{i} - \lambda_{i+1}) \langle \Lambda_i , \Lambda_j \rangle.
  \end{align*}
  Now $\langle \Lambda_i , \Lambda_j \rangle$ is simply the $(i,j)$'th entry of the inverse of the Cartan matrix of $\SU(N) = A_{N-1}$, and so can be written
  \begin{align*}
		\langle \Lambda_i , \Lambda_j \rangle = \min(i,j) - \frac{ij}{N}.
  \end{align*}
  Note also that
  \begin{align*}
	  \sum_{(i,j) \in \lambda} \mathrm{cn}(i,j) = \sum_{i=1}^{N-1}\frac{\lambda_i(\lambda_i+1)}{2} - i\lambda_i.
  \end{align*}
  Thus, expressed in terms of $\lambda_i$, the formula of the lemma is
  \begin{align*}
		\sum_{i,j=1}^{N-1} (\lambda_{i} - \lambda_{i+1})(\lambda_{j}-\lambda_{j+1} + 2)\left(N\min(i,j) - ij\right) = -\lvert\lambda\rvert^2 + N^2 \lvert\lambda\rvert + 2N\left(\sum_{i=1}^{N-1}\frac{\lambda_i(\lambda_i+1)}{2} - i\lambda_i\right).
  \end{align*}
  Let $\LHS(\lambda)$, $\RHS(\lambda)$ denote the left hand side and right hand side of this equation. It suffices to show that $\LHS(0) = \RHS(0)$, which is obvious, and that the difference of the two expression is invariant under $\lambda \to \lambda + \Lambda_l = \bar{\lambda}$ for all $1 \leq l \leq N -1$, viewing $\lambda$ as an element in the weight lattice. Under this transformation, the Young diagram becomes $\lambda_i \to \lambda_i + 1$ for $i \leq l$ and $\lambda_i \to \lambda_i$ for $i > l$. One easily finds that
  \begin{align*}
	  \RHS(\bar{\lambda}) - \RHS(\lambda) = -(l^2 + 2\abs{\lambda}l)+N^2l + 2N\sum_{i=1}^l (\lambda_i+1-i),
  \end{align*}
  and that
  \begin{align*}
		\LHS(\bar{\lambda}) - \LHS(\lambda) &= 2\sum_{j=1}^{l-1}(\eps_j+1)(N-l)j+2\sum_{j=l+1}^{N-1}(\eps_j+1)(N-j)l+(2\eps_l+3)(Nl-l^2) \\
		  &= 2\sum_{j=1}^l (\eps_j+1)(N-l)j + 2\sum_{j=l+1}^{N-1}(\eps_j+1)(N-j)l + Nl-l^2,
  \end{align*}
  where $\eps_j = \lambda_j - \lambda_{j+1}$. For the latter of the above, simply notice that the transformation is chosen such that $\eps_l \to \eps_l + 1$ and that the sum in $\LHS$ changes only when $i=l$ or $j=l$. Rewriting the expressions slightly, it now suffices to prove that
  \begin{align*}
		-2\abs{\lambda}l + N^2l &+ 2N\sum_{j=1}^l (\lambda_j-j) + Nl \\
		  &= 2\sum_{j=1}^l (\lambda_j - \lambda_{j+1}+1)(N-l)j + 2\sum_{j=l+1}^{N-1}(\lambda_j - \lambda_{j+1}+1)(N-j)l.
  \end{align*}
  The right hand side may be further rewritten as
  \begin{align*}
		2\sum_{j=1}^l \eps_j(N-l)j + 2\sum_{j=l+1}^{N-1}\eps_j(N-j)l + N(-l^2+Nl),
  \end{align*}
  and we therefore need to prove that
  \begin{align*}
		-2\abs{\lambda}l + 2N\sum_{j=1}^l (\lambda_j - j) + Nl = 2\sum_{j=1}^l (\lambda_j-\lambda_{j+1})(N-l)j + 2\sum_{j=l+1}^{N-1}(\lambda_j-\lambda_{j+1})l-Nl^2.
  \end{align*}
  For $l = 1$ this is easily checked and we proceed by induction on $l$, assuming that the equality holds true for some $l < N -1$. Under $l \to l+1$, the excess term on the left hand side is
  \begin{align*}
		-2\abs{\lambda} + 2N(\lambda_{l+1}-(l+1)) + N,
  \end{align*}
  and on the right hand side, it is
  \begin{align*}
		-2\sum_{j=1}^l\eps_j  j + 2\eps_{l+1}(N-(l+1)) + 2\sum_{j=l+2}^{N-1}\eps_j(N-j) - 2Nl - N.
  \end{align*}
  Finally, that these two excess terms agree follows from a calculation in terms of $\eps_j$, using that
  \begin{align*}
		\lambda_j = \sum_{m=j}^{N-1} \eps_m, \quad \abs{\lambda} = \sum_{j=1}^{N-1} j \eps_j.
  \end{align*}
\end{proof}

From Lemma~\ref{tlemma} it follows that the $T$-matrix $T^\mathrm{CFT}$ defined in \eqref{tcft} agrees with the inverse of \eqref{tblanchet} up to a factor of $\exp(2\pi ic/24)$. More precisely, let
\begin{align*}
	f_r(\lambda) = -\frac{\pi i}{rN}\left( \lvert\lambda\rvert^2 - N^2 \lvert\lambda\rvert - 2N\sum_{(i,j) \in \lambda} \mathrm{cn}(i,j) \right),
\end{align*}
and note that
\begin{align*}
	\exp\left(\frac{\pi i}{r}\langle \lambda + \rho , \lambda + \rho \rangle\right) &= \exp\left(f_r(\lambda) + \frac{\pi i N \dim G}{12r}\right) \\
	 &= \exp\left(f_r(\lambda) - \frac{\pi i \dim G(r-N)}{12 r} + \pi i \frac{\dim G}{12}\right) \\
	 &= \exp\left(f_r(\lambda) - \frac{\pi i c}{12} + \frac{i \pi}{N}\langle \rho , \rho \rangle\right).
\end{align*}
Thus, combining \eqref{valgafa} with \eqref{tblanchet}, we obtain
\begin{align}
  \label{dehntwistcft}
	\rho_k(t_\mu)_{\lambda, \lambda'} = \exp(-f_r(\lambda)) = \exp(-2\pi i c/24) (T^\mathrm{CFT}_{\lambda,\lambda'})^{-1}.
\end{align}

\subsection{The $S$-matrix}
Extending further our analogies between the theories, we now consider the $S$-matrices. Let in the following $N = 2$. The following results could be extracted from the litterature, using that our skein relations agree with those of $U_q(\fraksl_2)$ (cf. e.g. \cite{Ohtbog}), but are included here for the sake of completeness.

In conformal field theory (or in the quantum group picture), the $S$-matrix is given by
\begin{align}
  \label{smatricen}
	S^\mathrm{CFT}_{jl} = \sqrt{\frac{2}{r}} \sin\left(\frac{\pi(j+1)(l+1)}{r}\right),
\end{align}
where $j,l = 0, \dots, k$. Let us find the corresponding matrix arising from Blanchet's modular category. Recall that we denote elements of $\Gamma_{2,k}$ simply by their number of cells. Note that by \cite[Prop.~2.6]{Bla}, the framed Homflypt polynomial of coloured framed links is invariant under choice of orientation of any of the components of the link, and as such we leave out the orientations in the pictures drawn below.
\begin{lem}
  \label{smatrix}
  Let $0 \leq j \leq k$, $0 \leq  l  \leq  k$. Then we have the following relation of links viewed as elements of the relative version (cf. \cite[p.~195]{Bla}) of $\calH(I \times I \times I)$:
  \\[0.1cm]
  \centerline{
    \includegraphics{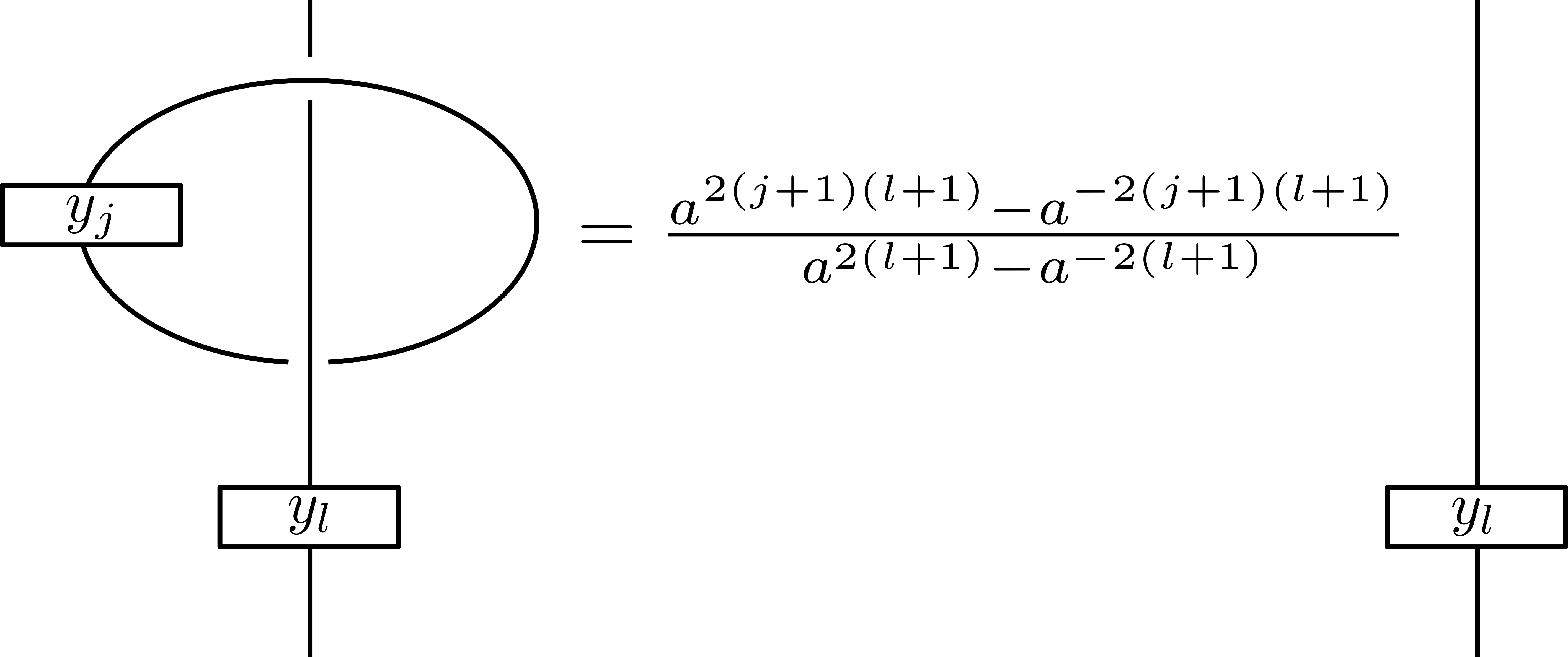}
  }
\end{lem}
\begin{proof}
  The proof is a direct translation of the similar proof in the Kauffman skein module (see e.g. \cite[Lemma~14.2]{Lic}), where one obtains the same result up to a factor of $(-1)^{j+l}$. We emphasize that we do not use the Kauffman bracket typically considered when talking about $\SU(2)$.
  
  By definition (see \cite[p.~200]{Bla}), $y_j$ is the composition of a number of quasi-idempotents, most of which are the identity in the case $N = 2$. One finds that $y_j = f_j$, in the notation of \cite[p.~197]{Bla}. Recall that we write $[j] = (a^{2j}-a^{-2j})/(a^2-a^{-2})$. Using that
  \begin{align*}
		-[j-1] +[j](a^2+a^{-2}) = [j+1],
  \end{align*}
  the recursive relation given for $y_j = f_j$ can be rewritten to look as follows:
  \\[0.1cm]
  \centerline{
    \includegraphics{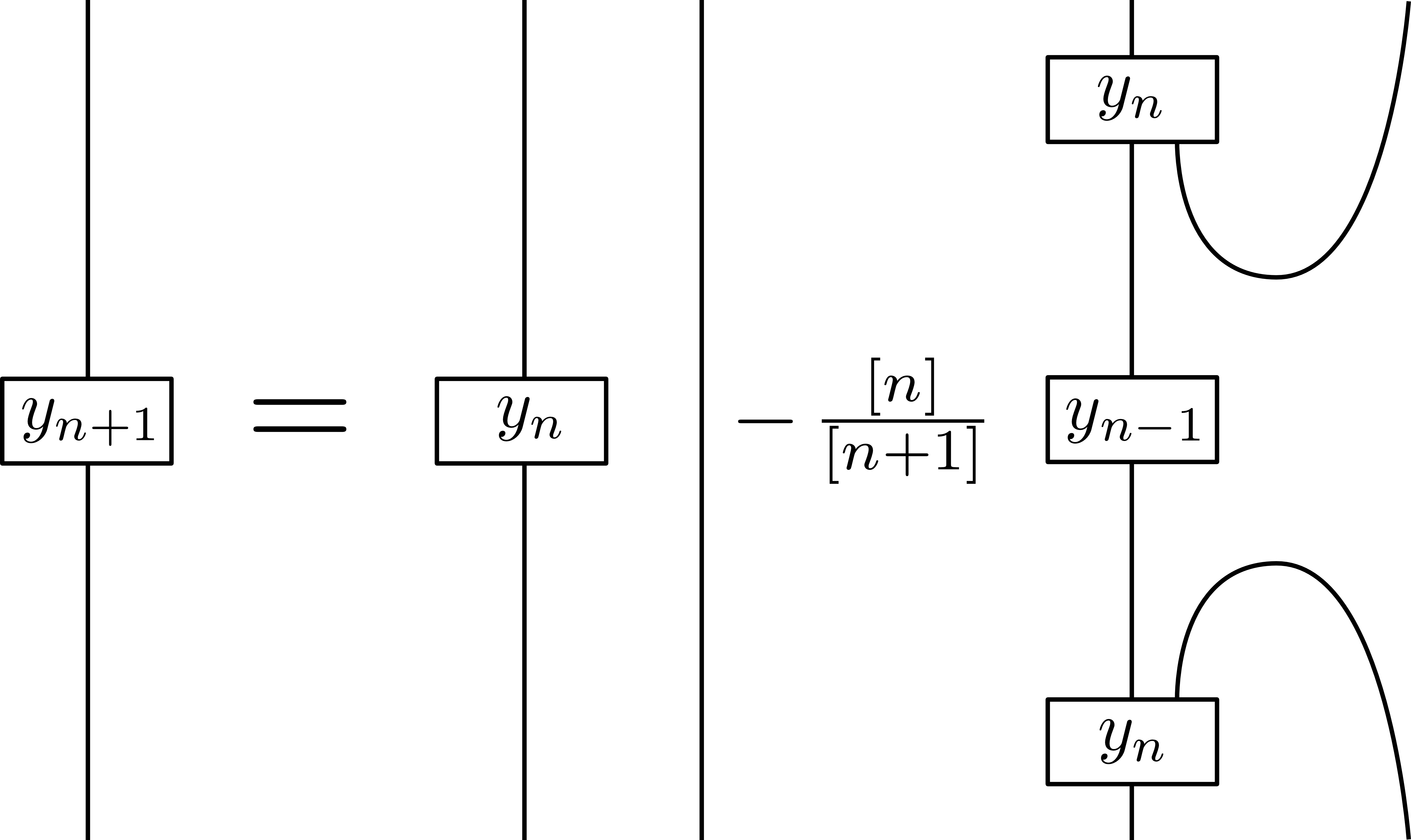}
  }
  
  Of course, this resembles closely Wenzl's recursive definition of the Jones--Wenzl idempotents. Now, a simple inductive argument using the quantum dimensions given in \cite[Prop.~1.14]{Bla}, proves the lemma in the case of $l = 0$. This can be used to give a Chebyshev polynomial style recursive relation for the idempotents traced in $\calH(S^1 \times I \times I)$ (which turns out to be exactly the same as the one for the Jones--Wenzl idempotents, but with all signs positive), which by the exact same argument as the one for \cite[Lemma~14.2]{Lic} can be used to prove the lemma.
\end{proof}
Following \cite{Tu}, let $s$ be the matrix whose $(j,l)$'th entry is the Homflypt polynomial evaluated on the Hopf link with components coloured $j$ and $l$. It follows immediately from Lemma~\ref{smatrix} that 
\begin{align*}
	s_{jl} = \frac{a^{2(j+1)(l+1)}-a^{-2(j+1)(l+1)}}{a^{2(l+1)}-a^{-2(l+1)}} [l+1] = \frac{a^{2(j+1)(l+1)}-a^{-2(j+1)(l+1)}}{a^2-a^{-2}}.
\end{align*}
Now, according to \cite{Tu}, the $S$-matrix is given by $\calD^{-1}s$, where $\calD$ is the \emph{rank} of the modular category, satisfying $\calD^2 = \sum_{j=0}^k s_{0j}^2$. We find that for $k \geq 1$,
\begin{align*}
	\calD^2 &= \sum_{j=0}^k [j+1]^2 = \frac{1}{(a^{-2}-a^{2})^2} \sum_{j=1}^{k+1} (a^{-4j}+a^{4j}-2) \\
	  &= \frac{1}{(2i\sin(\pi /r))^2}(-1-1-2(r-1)) = \frac{-2r}{-4\sin^2(\pi/r)} = \frac{r}{2} \frac{1}{\sin^2(\pi/r)}.
\end{align*}
A natural\footnote{Compare again this computation with the corresponding one in the theory coming from the Kauffman bracket. Here one obtains the same result but with the opposite sign, making the choice of square root seem somewhat more arbitrary.} choice of rank is thus
\begin{align*}
	\calD = \sqrt{\frac{r}{2}} \frac{1}{\sin(\pi /r)},
\end{align*}
and the $S$-matrix becomes
\begin{align*}
  S_{jl} = \calD^{-1}s_{jl} = \sqrt{\frac{2}{r}} \sin\left(\frac{\pi (j+1)(l+1)}{r}\right) = S^\mathrm{CFT}_{jl}
\end{align*}
\begin{rem}
  \label{framingrem}
  According to \cite{Tu}, changing the framing of a closed $3$-manifold will change its level $k$ quantum $\SU(2)$-invariant by a power of $\calD^{-1}\Delta$, where
  \begin{align*}
		\Delta = \sum_{j=0}^k T_{jj}^{-1} [j+1]^2.
  \end{align*}
  Arguing exactly as in the proof of Theorem~\ref{stnmedlinks}, we find
  \begin{align*}
		\Delta &= \frac{a}{(a^2-a^{-2})^2} \sum_{j=1}^{k-1} a^{-j^2-4j}+a^{-j^2+4j}-2a^{-j^2} \\
		  &= \frac{a}{(a^2-a^{-2})^2}(\sqrt{2r} \exp(\pi i(2r-16)/(8r)) - 2\sqrt{r/2} \exp(\pi i/4)) \\
		  &= - \frac{\exp(\pi i/(2r))}{\sin^2(\pi /r)} \sqrt{2r} \exp(\pi i/4) \left(\exp(-2 \pi i/r)-1\right).
  \end{align*}
  It follows that
  \begin{align*}
		\calD^{-1} \Delta &= -\frac{1}{2}\frac{\exp(-\pi i/(2r))}{\sin(\pi /r)} \exp(\pi i/4)(\exp(-2\pi i /r)-1)\\
		 &=\exp\left(-3 \frac{\pi i}{2r}\right) \exp(3 \pi i/4) = \left( \exp\left(\frac{2 \pi i c}{24}\right) \right)^3.
  \end{align*}
\end{rem}
\begin{rem}
  We have made explicit the difference between the specialization \eqref{valgafa} and the corresponding ones in the Kauffman bracket theory to make it clear how one needs to be a little cautious in comparing the theories and using known results from skein theory. Indeed one also obtains a modular category from the Kauffman bracket \cite[Ch.~XII]{Tu} which corresponds closely to the one we have discussed here. Turaev obtains the $S$-matrix \eqref{smatricen} through the choice $a = \pm i \exp(\pi i/(2r))$, killing off all signs, but in this specialization, the twist coefficients of the theories will not agree.
\end{rem}
\begin{rem}
  It would be nice to generalize Lemma~\ref{smatrix} and see that the $S$-matrices of the various theories agree in the general case $G = \SU(N)$, but as we do not use it in the main body of the paper anyway, we leave it at this and refer to \cite{LM} for a description of the Homflypt polynomial of coloured Hopf links.
\end{rem}

\section{Plots}
\label{plots}
As a picture is worth a thousand words, we collect here plots of some of the invariants that are considered in the paper. Firstly, in Figures~\ref{sporplot}--\ref{torus5} we consider the values of $Z_k^{\SU(2)}$ of the bare manifolds $M^b$.

\begin{figure}[h]
  \centering
  \includegraphics[scale=0.5]{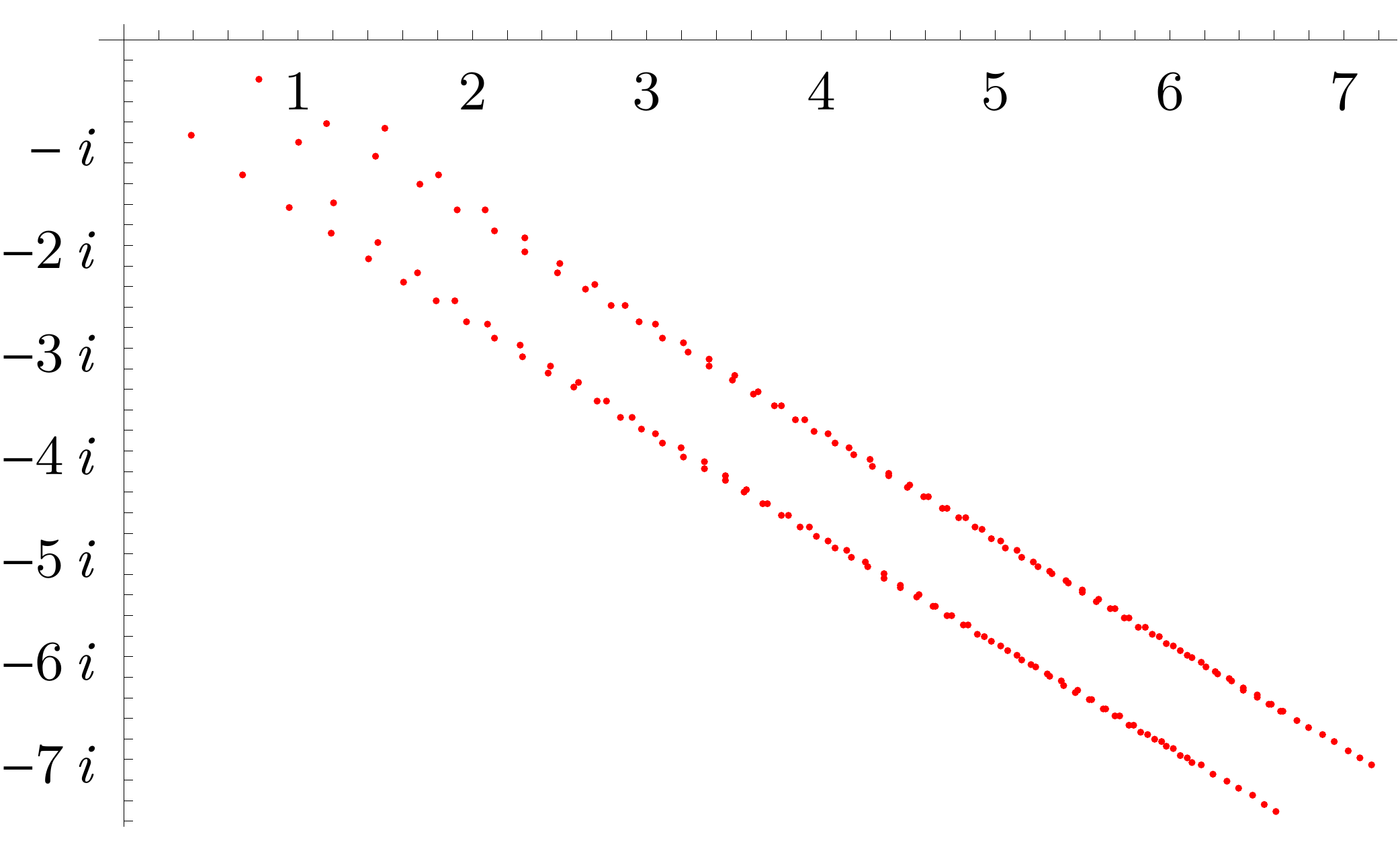}
  \caption{The values of $Z_k^{\SU(2)}(M^1)$ in the complex plane for $k = 0, \dots, 200$.}
  \label{sporplot}
\end{figure}
\begin{figure}[h]
  \centering
  \begin{minipage}[c]{0.47\textwidth}
    \centering
    \includegraphics[scale=0.3]{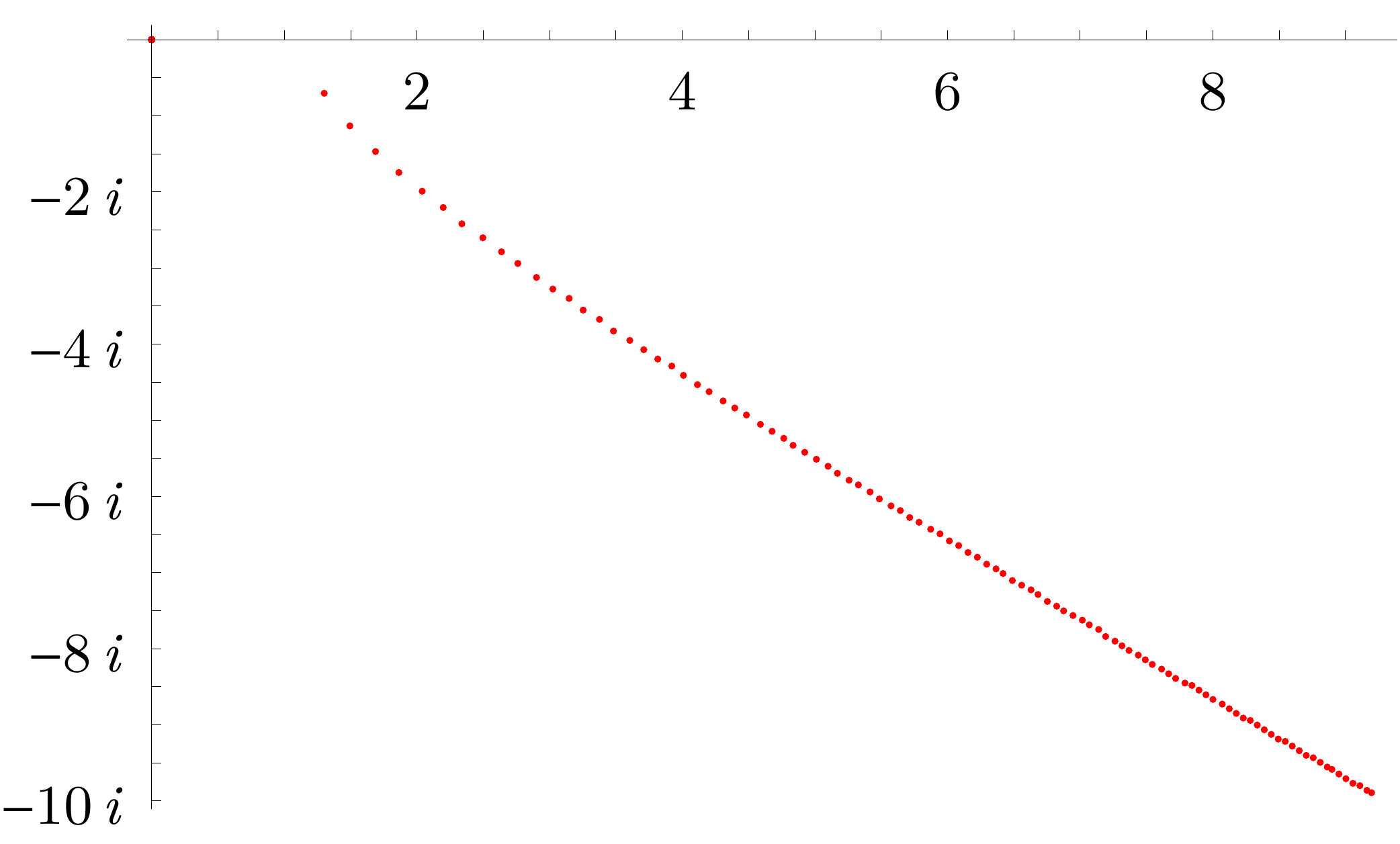}
    \caption{$Z_k^{\SU(2)}(M^2)$.}
    \label{torus2}
  \end{minipage}
  \hfill
  \begin{minipage}[c]{0.47\textwidth}
    \centering
    \includegraphics[scale=0.3]{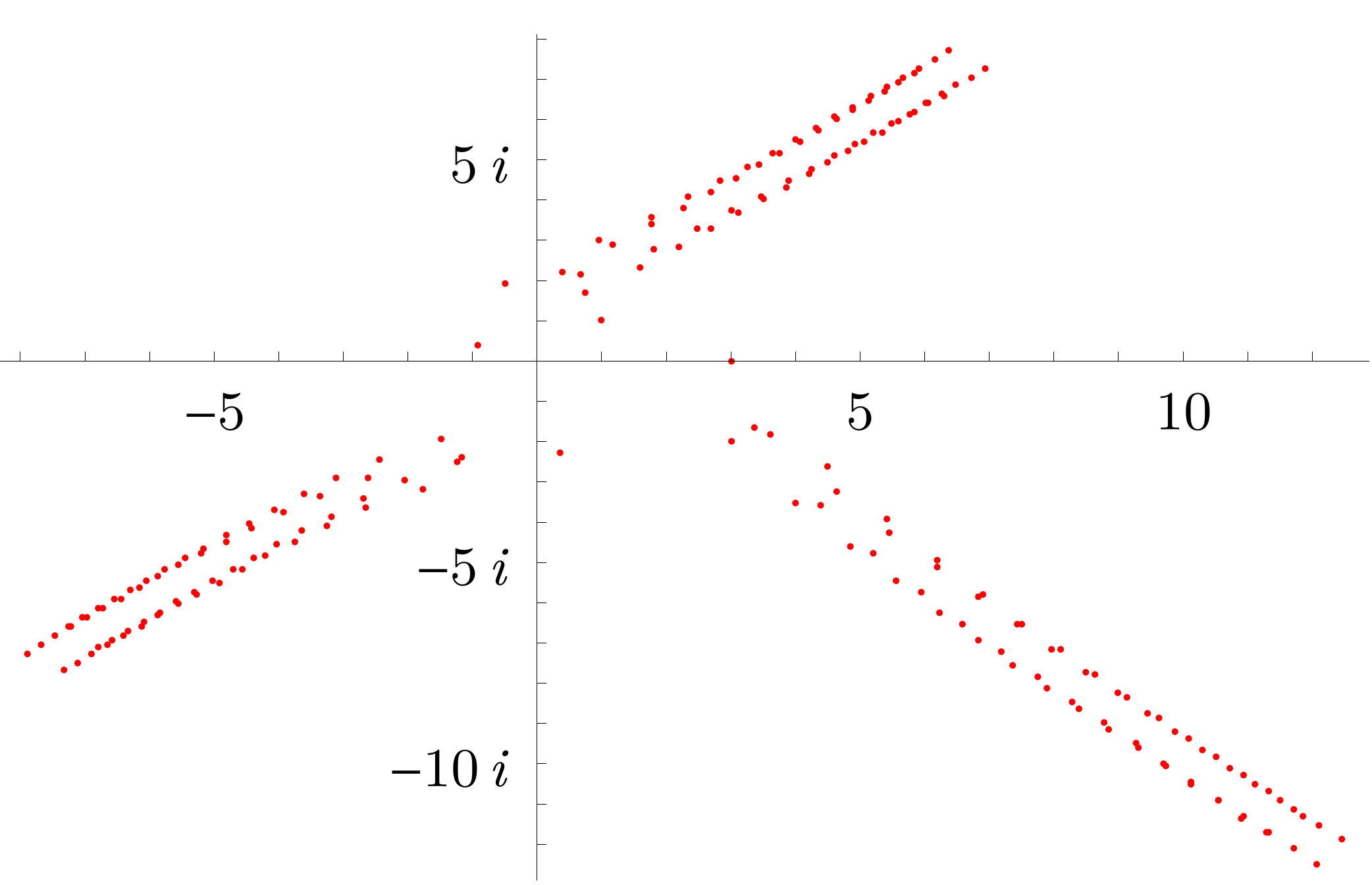}
    \caption{$Z_k^{\SU(2)}(M^3)$.}
    \label{torus3}
  \end{minipage}
  \hfill
  \begin{minipage}[c]{0.47\textwidth}
    \centering
    \includegraphics[scale=0.3]{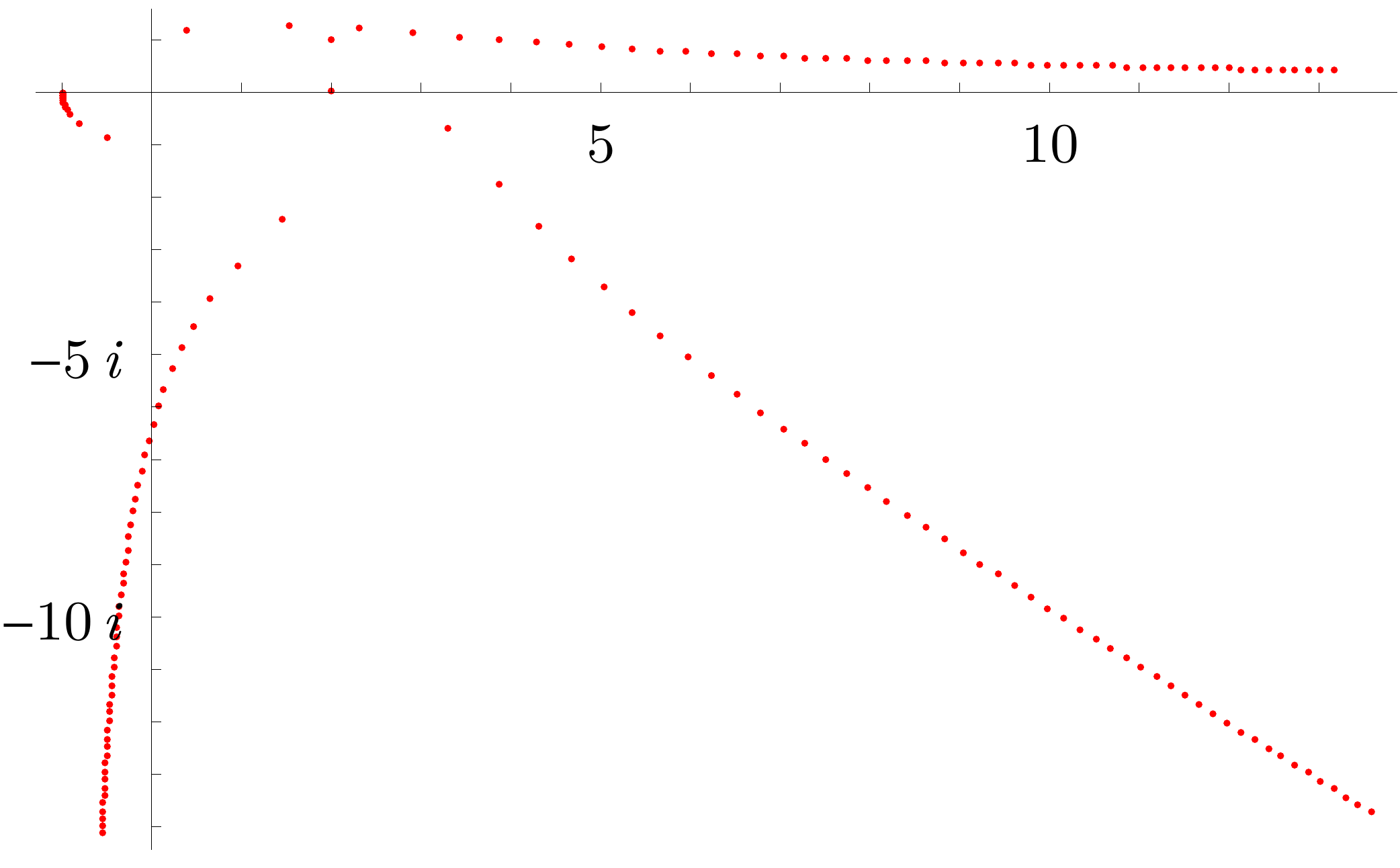}
    \caption{$Z_k^{\SU(2)}(M^4)$.}
    \label{torus4}
  \end{minipage}
  \hfill
  \begin{minipage}[c]{0.47\textwidth}
    \centering
    \includegraphics[scale=0.3]{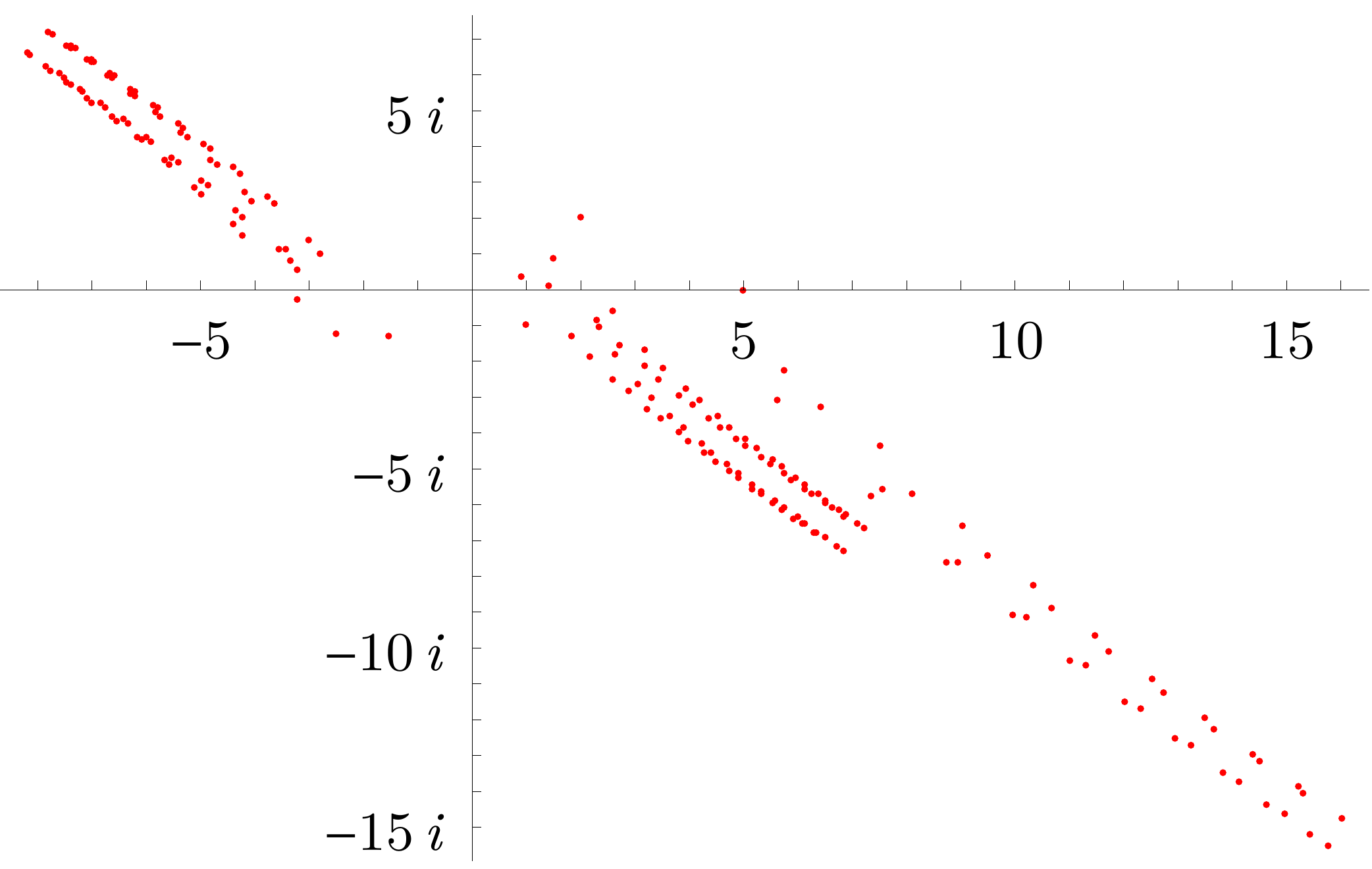}
    \caption{$Z_k^{\SU(2)}(M^5)$.}
    \label{torus5}
  \end{minipage}
\end{figure}

A particularly interesting case of Theorem~\ref{stnmedlinks} arises when the link colour is correlated with the level. An animation of the resulting behaviour -- which we hope to discuss in future work -- is available at \url{http://maths.fuglede.dk/WRTDehnTwist}.

Finally, in Figure~\ref{SU3plot} and Figure~\ref{SU4plot} we include plots of values of $Z_k^{\SU(N)}(M^1)$ for $N = 3,4$ and $k = 0, \dots, 100$.

\begin{figure}[h]
  \centering
  \begin{minipage}[c]{0.47\textwidth}
    \centering
    \includegraphics[scale=0.3]{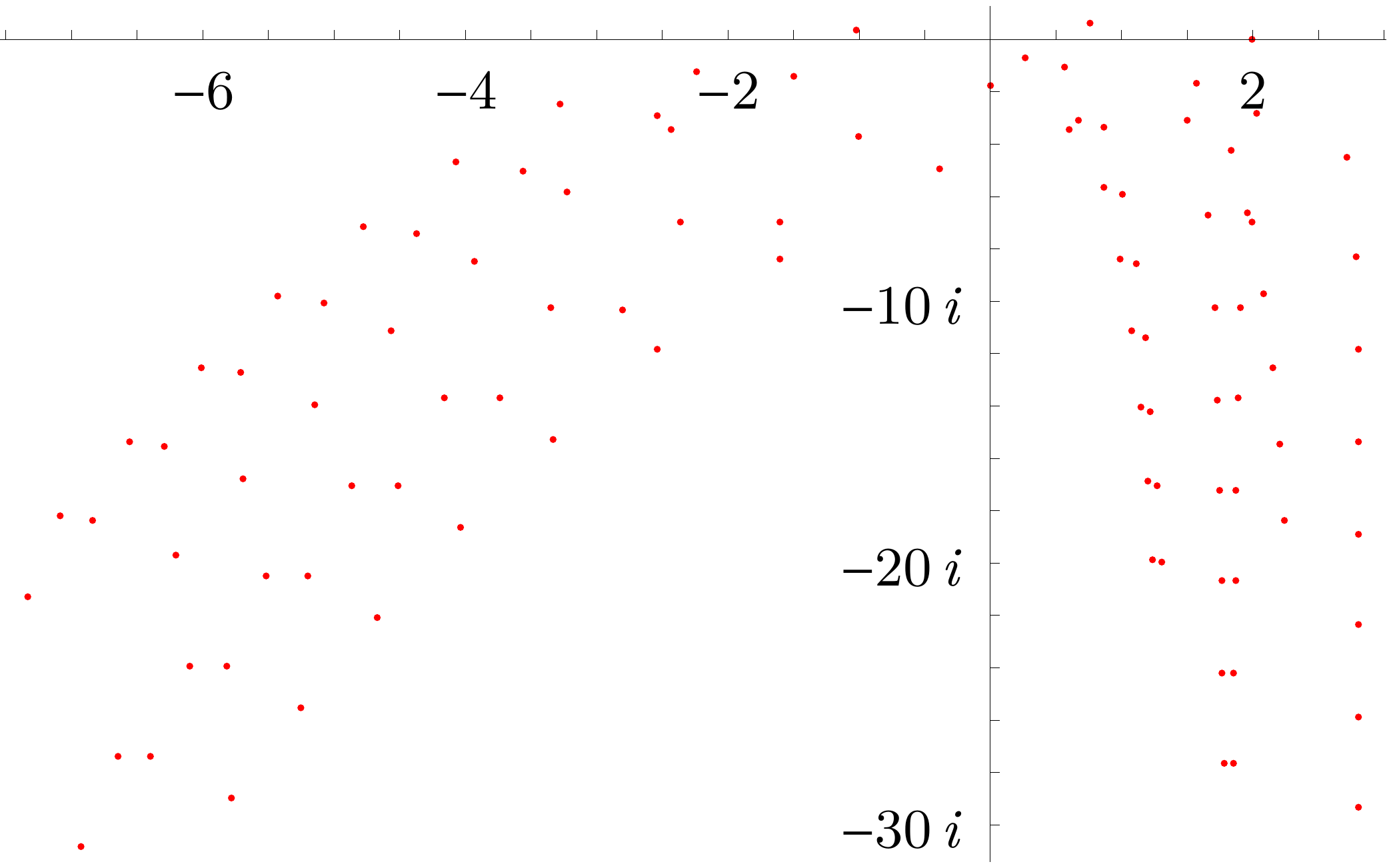}
    \caption{$Z_k^{\SU(3)}(M^1)$.}
    \label{SU3plot}
  \end{minipage}
  \hfill
  \begin{minipage}[c]{0.47\textwidth}
    \centering
    \includegraphics[scale=0.3]{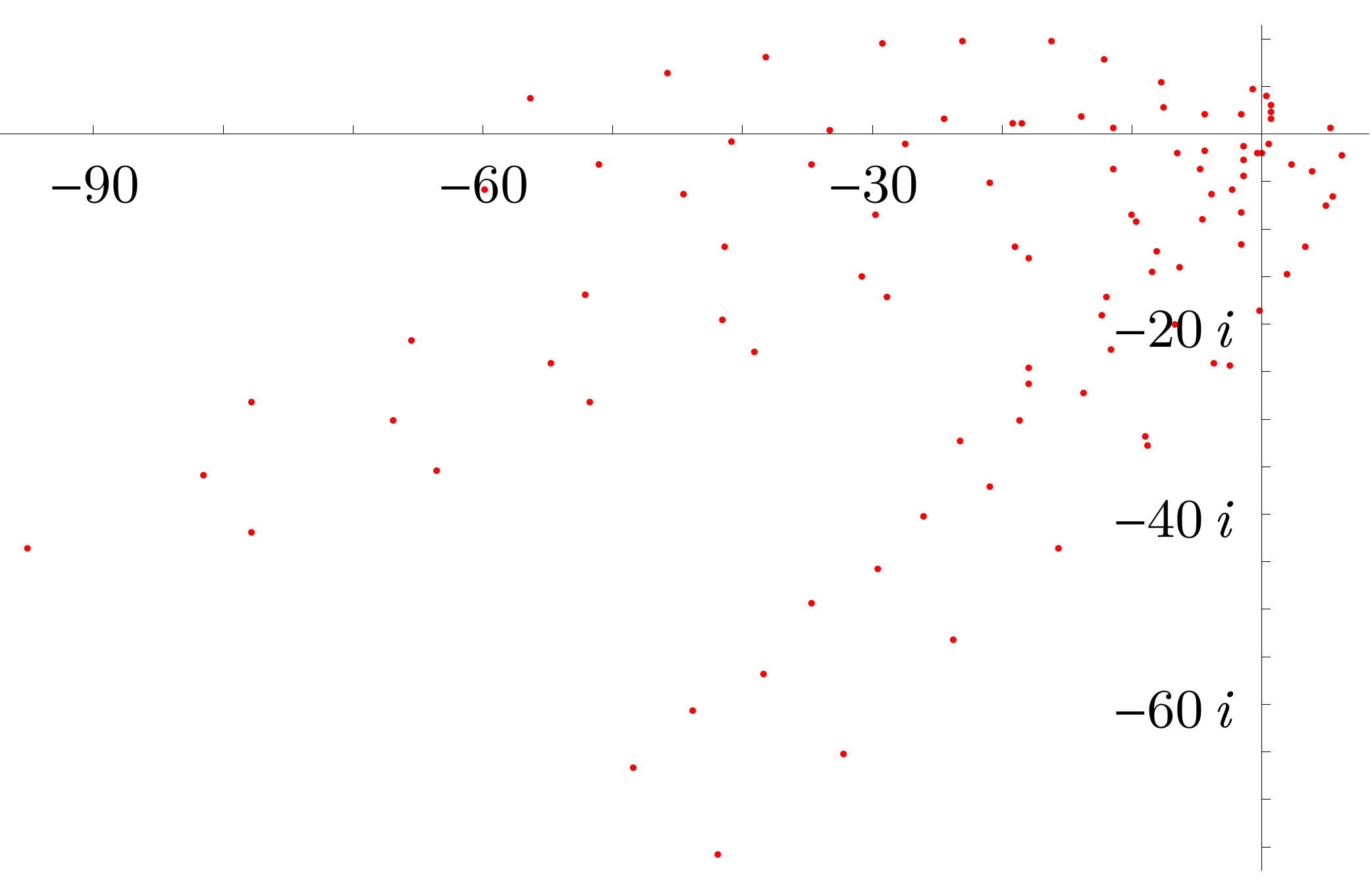}
    \caption{$Z_k^{\SU(4)}(M^1)$.}
    \label{SU4plot}
  \end{minipage}
\end{figure}

\bibliographystyle{is-alpha}
\bibliography{references}

\begin{thebibliography}{BHMV95}
\ifx \showCODEN  \undefined \def \showCODEN #1{CODEN #1}  \fi
\ifx \showISBN   \undefined \def \showISBN  #1{ISBN #1}   \fi
\ifx \showISSN   \undefined \def \showISSN  #1{ISSN #1}   \fi
\ifx \showLCCN   \undefined \def \showLCCN  #1{LCCN #1}   \fi
\ifx \showPRICE  \undefined \def \showPRICE #1{#1}        \fi
\ifx \showURL    \undefined \def \showURL {URL }          \fi
\ifx \path       \undefined \input path.sty               \fi
\ifx \ifshowURL \undefined
     \newif \ifshowURL
     \showURLtrue
\fi

\bibitem[AH06]{AndHan}
J{\o}rgen~Ellegaard Andersen and S{\o}ren~Kold Hansen.
\newblock Asymptotics of the quantum invariants for surgeries on the figure 8
  knot.
\newblock {\em J. Knot Theory Ramifications}, 15\penalty0 (4):\penalty0
  479--548, 2006.

\bibitem[AH12]{AH}
J{\o}rgen~Ellegaard Andersen and Benjamin Himpel.
\newblock {The Witten-Reshetikhin-Turaev invariants of finite order mapping
  tori. II.}
\newblock {\em Quantum Topol.}, 3\penalty0 (3-4):\penalty0 377--421, 2012.

\bibitem[AMU06]{AMU}
J{\o}rgen~Ellegaard Andersen, Gregor Masbaum, and Kenji Ueno.
\newblock Topological quantum field theory and the {N}ielsen-{T}hurston
  classification of {$M(0,4)$}.
\newblock {\em Math. Proc. Cambridge Philos. Soc.}, 141\penalty0 (3):\penalty0
  477--488, 2006.

\bibitem[And02]{Oht}
J{\o}rgen~Ellegaard Andersen.
\newblock The asymptotic expansion conjecture, {C}hapter 7.2 of {P}roblems on
  invariants of knots and 3-manifolds.
\newblock In Tomotada Ohtsuki, editor, {\em Geom. Topol. Monogr., vol. 4},
  pages 474--481, 2002.

\bibitem[And06]{asympgeom}
J{\o}rgen~Ellegaard Andersen.
\newblock Asymptotic faithfulness of the {$SU(n)$} representations.
\newblock {\em Annals of Mathematics}, 163:\penalty0 347--368, 2006.

\bibitem[And12a]{Agenerel}
J{\o}rgen~Ellegaard Andersen.
\newblock {A geometric formula for the Witten-Reshetikhin-Turaev Quantum
  Invariants and some applications}.
\newblock \url{http://arxiv.org/abs/1206.2785}, 2012.

\bibitem[And12b]{Andfiniteorder}
J{\o}rgen~Ellegaard Andersen.
\newblock The {W}itten-{R}eshetikhin-{T}uraev invariants of finite order
  mapping tori {I} ({U}pdated version of 1995 {U}niv. of {A}arhus preprint.).
\newblock Published online 23/4 2012: DOI: 10.1515/crelle-2012-0033. Available
  at
  \url{http://www.degruyter.com/view/j/crelle.ahead-of-print/crelle-2012-0033/crelle-2012-0033.xml?format=INT},
  2012.

\bibitem[Ati90]{AtiFraming}
Michael Atiyah.
\newblock On framings of {$3$}-manifolds.
\newblock {\em Topology}, 29\penalty0 (1):\penalty0 1--7, 1990.
\newblock \showCODEN{TPLGAF}.
\newblock \showISSN{0040-9383}.
\newblock \ifshowURL {\showURL
  \path|http://dx.doi.org/10.1016/0040-9383(90)90021-B|}\fi.

\bibitem[AU07a]{AU1}
J{\o}rgen~Ellegaard Andersen and Kenji Ueno.
\newblock {Abelian Conformal Field theories and Determinant Bundles}.
\newblock {\em International Journal of Mathematics}, 18:\penalty0 919--993,
  2007.

\bibitem[AU07b]{AU2}
J{\o}rgen~Ellegaard Andersen and Kenji Ueno.
\newblock {Constructing modular functors from conformal field theories}.
\newblock {\em Journal of Knot theory and its Ramifications}, 16\penalty0
  (2):\penalty0 127--202, 2007.

\bibitem[AU11]{AU4}
J{\o}rgen~Ellegaard Andersen and Kenji Ueno.
\newblock {Construction of the Reshetikhin-Turaev TQFT from conformal field
  theory}.
\newblock \url{http://arxiv.org/pdf/1110.5027}, 2011.

\bibitem[AU12]{AU3}
J{\o}rgen~Ellegaard Andersen and Kenji Ueno.
\newblock {Modular functors are determined by their genus zero data}.
\newblock {\em Quantum Topology}, 3:\penalty0 255--291, 2012.

\bibitem[{Bea}13]{Bea}
Chris {Beasley}.
\newblock {Localization for Wilson loops in Chern-Simons theory}.
\newblock {\em {Adv. Theor. Math. Phys.}}, 17\penalty0 (1):\penalty0 1--240,
  2013.
\newblock \showISSN{1095-0761; 1095-0753/e}.

\bibitem[BHMV92]{BHMV1}
Christian Blanchet, Nathan Habegger, Gregor Masbaum, and Pierre Vogel.
\newblock Three-manifold invariants derived from the {K}auffman bracket.
\newblock {\em Topology}, 31:\penalty0 685--699, 1992.

\bibitem[BHMV95]{BHMV2}
Christian Blanchet, Nathan Habegger, Gregor Masbaum, and Pierre Vogel.
\newblock Topological quantum field theories derived from the {K}auffman
  bracket.
\newblock {\em Topology}, 34:\penalty0 883--927, 1995.

\bibitem[Bla00]{Bla}
Christian Blanchet.
\newblock Hecke algebras, modular categories and $3$-manifolds quantum
  invariants.
\newblock {\em Topology}, 39\penalty0 (1):\penalty0 193--223, 2000.

\bibitem[Cha10]{Chaasympmcg}
Laurent Charles.
\newblock Asymptotic properties of the quantum representations of the mapping
  class group.
\newblock \url{http://arxiv.org/pdf/1005.3452v2}, 2010.

\bibitem[Cha11]{Cha}
Laurent Charles.
\newblock Torus knot state asymptotics, 2011.
\newblock \url{http://arxiv.org/abs/1107.4692}.

\bibitem[Cha12]{Chaasymp}
Laurent Charles.
\newblock Asymptotic properties of the quantum representations of the modular
  group.
\newblock {\em Trans. Amer. Math. Soc.}, 364:\penalty0 5829--5856, 2012.

\bibitem[CM11]{CMII}
Laurent Charles and Julien Marche.
\newblock {Knot state asymptotics II, Witten conjecture and irreducible
  representations}, 2011.
\newblock \url{http://arxiv.org/abs/1107.1646}.

\bibitem[FG91]{FG}
Daniel~S. Freed and Robert~E. Gompf.
\newblock {Computer calculation of Witten's 3-manifold invariant.}
\newblock {\em Commun. Math. Phys.}, 141\penalty0 (1):\penalty0 79--117, 1991.

\bibitem[FM11]{FM}
Benson Farb and Dan Margalit.
\newblock {\em {A primer on mapping class groups.}}
\newblock {Princeton Mathematical Series. Princeton, NJ: Princeton University
  Press. xiv, 492~p.}, 2011.

\bibitem[Fol95]{Fol}
Gerald~B. Folland.
\newblock {\em A course in abstract harmonic analysis}.
\newblock Studies in Advanced Mathematics. CRC Press, Boca Raton, FL, 1995.
\newblock \showISBN{0-8493-8490-7}.
\newblock x+276 pp.

\bibitem[GW86]{GW}
Doron Gepner and Edward Witten.
\newblock String theory on group manifolds.
\newblock {\em Nuclear Phys. B}, 278\penalty0 (3):\penalty0 493--549, 1986.
\newblock \showCODEN{NUPBBO}.
\newblock \showISSN{0550-3213}.
\newblock \ifshowURL {\showURL
  \path|http://dx.doi.org/10.1016/0550-3213(86)90051-9|}\fi.

\bibitem[Han99]{Han}
S{\o}ren~Kold Hansen.
\newblock {\em Reshetikhin-{T}uraev invariants of {S}eifert 3--manifolds, and
  their asymptotic expansions}.
\newblock PhD thesis, Aarhus University, 1999.

\bibitem[Han01]{Han1}
S{\o}ren~Kold Hansen.
\newblock Reshetikhin-{T}uraev invariants of {S}eifert 3-manifolds and a
  rational surgery formula.
\newblock {\em Algebr. Geom. Topol.}, 1:\penalty0 627--686, 2001.
\newblock \showISSN{1472-2747}.
\newblock \ifshowURL {\showURL
  \path|http://dx.doi.org/10.2140/agt.2001.1.627|}\fi.

\bibitem[Han05]{Han2}
S{\o}ren~Kold Hansen.
\newblock Analytic asymptotic expansions of the {R}eshetikhin--{T}uraev
  invariants of {S}eifert 3-manifolds for {SU}(2), 2005.

\bibitem[Hik05]{Hik}
Kazuhiro Hikami.
\newblock On the quantum invariant for the {B}rieskorn homology spheres.
\newblock {\em Internat. J. Math.}, 16\penalty0 (6):\penalty0 661--685, 2005.
\newblock \showISSN{0129-167X}.
\newblock \ifshowURL {\showURL
  \path|http://dx.doi.org/10.1142/S0129167X05003004|}\fi.

\bibitem[HT02]{HT}
S{\o}ren~Kold Hansen and Toshie Takata.
\newblock Quantum invariants of {S}eifert 3-manifolds and their asymptotic
  expansions.
\newblock In {\em Invariants of knots and 3-manifolds ({K}yoto, 2001)},
  volume~4 of {\em Geom. Topol. Monogr.}, pages 69--87 (electronic). Geom.
  Topol. Publ., Coventry, 2002.
\newblock \ifshowURL {\showURL
  \path|http://dx.doi.org/10.2140/gtm.2002.4.69|}\fi.

\bibitem[Hum79]{Hum}
Stephen~P. Humphries.
\newblock Generators for the mapping class group.
\newblock {\em Topology of Low-Dimensional Manifolds (Proc. Second Sussex
  Conf., Chelwood Gate, 1977)}, Lecture Notes in Math., Vol. 722 (1979),
  Springer:\penalty0 44--47, 1979.

\bibitem[Jef91]{Jefphd}
Lisa~C. Jeffrey.
\newblock {\em On some aspects of {C}hern--{S}imons gauge theory}.
\newblock PhD thesis, University of Oxford, 1991.

\bibitem[Jef92]{Jef}
Lisa~C. Jeffrey.
\newblock Chern-{S}imons-{W}itten {I}nvariants of {L}ens {S}paces and {T}orus
  {B}undles, and the {S}emiclassical {A}pproximation.
\newblock {\em Commun. Math. Phys.}, 147:\penalty0 563--604, 1992.

\bibitem[J{\o}r11]{Joer}
S{\o}ren~Fuglede J{\o}rgensen.
\newblock Quantum representations of mapping class groups, progress report,
  2011.
\newblock Available at \url{http://maths.fuglede.dk}.

\bibitem[J{\o}r13]{JoerThesis}
S{\o}ren~Fuglede J{\o}rgensen.
\newblock {\em Semiclassical properties of the quantum representations of
  mapping class groups}.
\newblock PhD thesis, Aarhus University, 2013.
\newblock Available at \url{http://maths.fuglede.dk}.

\bibitem[Kac90]{Kac}
Victor~G. Kac.
\newblock {\em Infinite-dimensional {L}ie algebras}.
\newblock Cambridge University Press, Cambridge, third edition, 1990.
\newblock \showISBN{0-521-37215-1; 0-521-46693-8}.
\newblock xxii+400 pp.
\newblock \ifshowURL {\showURL
  \path|http://dx.doi.org/10.1017/CBO9780511626234|}\fi.

\bibitem[KSV97]{KSV}
Michael Karowski, Robert Schrader, and Elmar Vogt.
\newblock Invariants of three-manifolds, unitary representations of the mapping
  class group, and numerical calculations.
\newblock {\em Experiment. Math.}, 6\penalty0 (4):\penalty0 317--352, 1997.
\newblock \showISSN{1058-6458}.
\newblock \ifshowURL {\showURL
  \path|http://projecteuclid.org/getRecord?id=euclid.em/1047047192|}\fi.

\bibitem[Las98]{Las}
Yves Laszlo.
\newblock Hitchin's and {WZW} connections are the same.
\newblock {\em J. Differential Geom.}, 49\penalty0 (3):\penalty0 547--576,
  1998.
\newblock \showCODEN{JDGEAS}.
\newblock \showISSN{0022-040X}.
\newblock \ifshowURL {\showURL
  \path|http://projecteuclid.org/getRecord?id=euclid.jdg/1214461110|}\fi.

\bibitem[Lic97]{Lic}
William B.~Raymond Lickorish.
\newblock {\em An introduction to knot theory}, volume 175 of {\em Graduate
  Texts in Mathematics}.
\newblock Springer-Verlag, New York, 1997.
\newblock \showISBN{0-387-98254-X}.
\newblock x+201 pp.

\bibitem[ML03]{LM}
Hugh~R. Morton and Sascha~G. Lukac.
\newblock {The Homfly polynomial of the decorated Hopf link.}
\newblock {\em J. Knot Theory Ramifications}, 12\penalty0 (3):\penalty0
  395--416, 2003.

\bibitem[Nis98]{Nis}
Haruko Nishi.
\newblock {$SU(n)$-Chern-Simons invariants of Seifert fibered 3-manifolds.}
\newblock {\em Int. J. Math.}, 9\penalty0 (3):\penalty0 295--330, 1998.

\bibitem[Oht02]{Ohtbog}
Tomotada Ohtsuki.
\newblock {\em Quantum {I}nvariants, {A} {S}tudy of {K}nots, 3-{M}anifolds, and
  {T}heir {S}ets}.
\newblock World Scientific Publishing Co. Pte. Ltd., 2002.

\bibitem[Orl72]{Orl}
Peter Orlik.
\newblock {\em Seifert manifolds}.
\newblock Lecture Notes in Mathematics, Vol. 291. Springer-Verlag, Berlin,
  1972.
\newblock viii+155 pp.

\bibitem[Rob94]{Rob}
Justin Roberts.
\newblock Skeins and mapping classes.
\newblock {\em Math. Proc. Camb. Phil. Soc.}, 115:\penalty0 53--77, 1994.

\bibitem[Roz96]{Rozformulas}
Lev Rozansky.
\newblock Residue formulas for the large {$k$} asymptotics of {W}itten's
  invariants of {S}eifert manifolds. {T}he case of {${\mathrm SU}(2)$}.
\newblock {\em Comm. Math. Phys.}, 178\penalty0 (1):\penalty0 27--60, 1996.
\newblock \showCODEN{CMPHAY}.
\newblock \showISSN{0010-3616}.
\newblock \ifshowURL {\showURL
  \path|http://projecteuclid.org/getRecord?id=euclid.cmp/1104286553|}\fi.

\bibitem[RT91]{RT2}
Nicolai~Yu. Reshetikhin and Vladimir~G. Turaev.
\newblock Invariants of {$3$}-manifolds via link polynomials and quantum
  groups.
\newblock {\em Invent. Math.}, 103\penalty0 (3):\penalty0 547--597, 1991.
\newblock \showCODEN{INVMBH}.
\newblock \showISSN{0020-9910}.
\newblock \ifshowURL {\showURL \path|http://dx.doi.org/10.1007/BF01239527|}\fi.

\bibitem[Tur10]{Tu}
Vladimir~G. Turaev.
\newblock {\em Quantum invariants of knots and 3-manifolds}, volume~18 of {\em
  de Gruyter Studies in Mathematics}.
\newblock Walter de Gruyter \& Co., Berlin, revised edition, 2010.
\newblock \showISBN{978-3-11-022183-1}.
\newblock xii+592 pp.
\newblock \ifshowURL {\showURL
  \path|http://dx.doi.org/10.1515/9783110221848|}\fi.

\bibitem[TUY89]{TUY}
Akihiro Tsuchiya, Kenji Ueno, and Yasuhiko Yamada.
\newblock Conformal field theory on universal family of stable curves with
  gauge symmetries.
\newblock In {\em Integrable systems in quantum field theory and statistical
  mechanics}, volume~19 of {\em Adv. Stud. Pure Math.}, pages 459--566.
  Academic Press, Boston, MA, 1989.

\bibitem[TW93]{TW}
Vladimir Turaev and Hans Wenzl.
\newblock Quantum invariants of {$3$}-manifolds associated with classical
  simple {L}ie algebras.
\newblock {\em Internat. J. Math.}, 4\penalty0 (2):\penalty0 323--358, 1993.
\newblock \showISSN{0129-167X}.
\newblock \ifshowURL {\showURL
  \path|http://dx.doi.org/10.1142/S0129167X93000170|}\fi.

\bibitem[Wit89]{WitJones}
Edward Witten.
\newblock Quantum field theory and the {J}ones polynomial.
\newblock {\em Comm. Math. Phys.}, 121:\penalty0 351--399, 1989.

\end{thebibliography}
\end{document}